\documentclass[reqno]{amsart}
\usepackage{amsmath}
\usepackage{amssymb}
\usepackage{amsthm}
\usepackage{amsfonts}
\usepackage{graphicx}
\usepackage{color}
\usepackage{enumerate}
\usepackage{eucal}
\usepackage[utf8]{inputenc}
\usepackage[font=small,labelfont={bf,sf}, textfont={sf},margin=0em]{caption}
\usepackage{url}
\usepackage{hyperref}

\definecolor{limegreen}{rgb}{0.196,0.804,0.196}
\definecolor{darkgreen}{rgb}{0.0,0.5,0.0}
\definecolor{darkbluegreen}{rgb}{0,0.3,0.6}
\definecolor{badgerred}{rgb}{0.715,0.004,0.004}
\hypersetup{colorlinks, 
citecolor=badgerred,
filecolor=black,
linkcolor=blue,
urlcolor=darkgreen,
bookmarksopen=false,
pdftex}

\newcommand{\interior}{\mathop{\rm int}}

\newcommand{\hu}{{\mathcal H}}  
\newcommand{\gvol}{\mathfrak{m}}         
\newcommand{\rto}{\stackrel{\rho}\to}
\newcommand{\Tx}{{T_{\max}}}

\newcommand{\sphere}{{\mathbb S}}

\newcommand{\R}{{\mathbb R}}

\newcommand{\Z}{{\mathbb Z}}
\newcommand{\N}{{\mathbb N}}
\newcommand{\SO}{{\mathrm{SO}}}
\newcommand{\taum}{\tau_{\max}}
\newcommand{\bu}{{\bar u}}

\newcommand{\cC}{{\mathcal C}}
\newcommand{\cD}{{\mathcal D}}

\newcommand{\hM}{C} 

\newcommand{\mf}[1]{{\mathfrak #1}}
\newcommand{\mc}[1]{{\mathcal #1}}

\newcommand{\bs}[1]{{\boldsymbol #1}}

\newtheorem{theorem}{Theorem}[section]
\newtheorem{lemma}[theorem]{Lemma}
\newtheorem{definition}[theorem]{Definition}
\newtheorem{prop}[theorem]{Proposition}

\newtheorem{conjecture}{Conjecture}

\newtheorem{claim}[theorem]{Claim}

\theoremstyle{remark}
\newtheorem{remark}[theorem]{Remark}

\numberwithin{equation}{section}

\newcommand{\ba}{{\bar a}}
\newcommand{\e}{\epsilon}

\begin{document}

\title{Dynamics of Convex Mean Curvature Flow} \author{S.B.Angenent}
\author{P.Daskalopoulos} \author{N.Sesum}
\begin{abstract}
There is an extensive and growing body of work analyzing convex ancient solutions to
Mean Curvature Flow (MCF), or equivalently of Rescaled Mean Curvature Flow
(RMCF). The goal of this paper is to complement the existing literature, which
analyzes ancient solutions one at a time, by considering the space $X$ of all convex
hypersurfaces $M\subset\R^{n+1}$, regard RMCF as a semiflow on this space, and study
the dynamics of this semiflow.  To this end, we first extend the well known existence
and uniqueness of solutions to MCF with smooth compact convex initial data to include
the case of arbitrary non compact and non smooth initial convex hypersurfaces.  We
identify a suitable weak topology with good compactness properties on the space $X$
of convex hypersurfaces and show that RMCF defines a continuous local semiflow on $X$
whose fixed points are the shrinking cylinder solitons $S^k\times\R^{n-k}$, and for
which the Huisken energy is a Lyapunov function.  Ancient solutions to MCF are then
complete orbits of the RMCF semiflow on $X$.  We consider the set of all
hypersurfaces that lie on an ancient solution that in backward time is asymptotic to
one of the shrinking cylinder solitons and prove various topological properties of
this set.  We show that this space is a path connected, compact subset of $X$, and,
considering only point symmetric hypersurfaces, that it is topologically trivial in
the sense of Čech cohomology.  We also give a strong evidence in support of the
conjecture that the space of all convex ancient solutions with a point symmetry in
$\R^{n+1}$ is homeomorphic to an $n-1$ dimensional simplex.

\end{abstract}

\maketitle
\setcounter{tocdepth}{1} 
\begingroup\footnotesize\sffamily\tableofcontents \endgroup

\section{Introduction}

\subsection{Ancient solutions and rescaled MCF}

We consider convex, not necessarily compact hypersurfaces $\hat M_t\subset\R^{n+1}$
that evolve by Mean Curvature Flow, i.e.~that have a parametrization
$\hat F:\mc M^n\times(0,T)\to\R^{n+1}$ satisfying
\begin{equation}
\label{eq-mcf}
\bigl(\partial_t\hat F\bigr)^{\perp} = \Delta_{\hat F}(\hat F)
\tag{MCF}
\end{equation}
Here $X^\perp$ is the component perpendicular to $T_{\hat F(p, t)}\hat M_t$ of any
vector $X\in T_{\hat F(p,t)}\R^{n+1}$, and $\Delta_{\hat F}$ is the Laplacian of the
pullback of the Euclidean metric under the immersion $p\mapsto \hat F(p, t)$.

If the hypersurfaces are compact and convex, then Huisken's theorem
\cite{Hu} implies that they contract to a point in finite time.

A family of hypersurfaces $M_t$ evolving by MCF is by definition an
\emph{ancient solution} to \eqref{eq-mcf} if it exists for all $t\in
(-\infty,T)$ for some $T\in\R$.

For any family $\hat M_t$ of hypersurfaces that evolves by MCF
throughout some time interval $t_0 < t <T$ with $t_0 < 1$, the
hypersurfaces
\[
M_\tau := e^{\tau/2}\, \hat M_{1-e^{-\tau}}, \qquad (-\ln(1-t_0)\leq \tau<\taum)
\]
evolve by \emph{Rescaled Mean Curvature Flow.}  If $\hat F$ parametrizes $\hat M_t$,
then $F:\mc M\times(0, \infty) \to \R^{n+1}$ defined by
\[
F(p, \tau) := e^{-\tau/2}\hat F\bigl(p, 1-e^{-\tau}\bigr), \text{ i.e.  } \hat F(p,
t) = \sqrt{1-t}\,F(p,t)
\]
parametrizes $M_\tau$ and satisfies
\begin{equation}
\label{eq-rmcf}
\bigl(\partial_\tau F\bigr)^\perp
= \Delta_F(F)+\frac 12 F^\perp
\tag{RMCF}
\end{equation}
Its lifespan $\taum$ is given by
\[
\taum =
\begin{cases}
+\infty \quad & \text { if } \,\, T  \geq 1, \\
-\ln(1-T) \quad & \text{ if } \,\, T<1.
\end{cases}
\]

\medskip

Much is known about the existence and classification of convex ancient solutions to
MCF \cite{ADS1,ADS2,BLL,BLT21,CDDHS,CHHW,DH,HH,HuSin2015,Wh}.  More precisely, for
ancient noncollapsed flows in $\mathbb{R}^3$, or more generally in $\mathbb{R}^{n+1}$
under the additional assumption that the flow is uniformly two-convex, a complete
classification has been obtained in significant works by Brendle-Choi \cite{BC1,BC2}
and in our papers \cite{ADS1,ADS2}.  Specifically, any such flow is, up to parabolic
rescaling and space-time rigid motions, either the flat plane, the round shrinking
sphere, the round shrinking neck, the rotationally symmetric translating bowl soliton
from \cite{AltschulerWu}, or the rotationally symmetric ancient oval from
\cite{Wh,HH}.  In \cite{CHHW} the same classification was obtained, but under the
only assumption that a tangent flow at infinity to an ancient solution was a round
cylinder $S^{n-1}\times\R$. In stark contrast, in higher dimensions, in general,
there are multi-parameter families of examples of ancient noncollapsed flows that are
not rotationally symmetric, and not even cohomogeneity-one, see \cite{W,HIMW,DH}. For
this reason, the classification of ancient noncollapsed flows in higher dimensions
without assuming two-convexity (a condition that guarantees that the tangent flow is
a round cylinder) is significantly more difficult. The solutions constructed in
\cite{DH} will be of special interest in this paper. These solutions are compact
ancient solutions in $\mathbb{R}^{n+1}$, with $O(k)$ symmetry where $2 \le k \le n$
(but in general not $O(k) \times O(n+1-k)$ symmetric) and they are asymptotic to
$S^k \times \mathbb{R}^{n-k}$.

In \cite{DH}, the authors introduced a classification program for ancient
noncollapsed flows in $\mathbb{R}^4$. Recall first that if $M_t$ is an ancient
noncollapsed flow in $\mathbb{R}^4$, then its tangent flow at $-\infty$ is always
either a round shrinking sphere, a round shrinking neck, a round shrinking
bubble-sheet\footnote{In the terminology of \cite{CDDHS,DH} a \textit{bubble-sheet}
  is a generalized cylinder, i.e.~up to a rigid motion and dilation, a bubble sheet
  is the hypersurface $\R^{n-k}\times S^k\subset\R^{n+1}$ with $k>1$.}, or a static
plane.  The first and last scenario are of course trivial, and ancient noncollapsed
flows whose tangent flow at $-\infty$ are a round cylinder have been discussed above.
Hence, assume that the tangent flow at $-\infty$ is a bubble-sheet, specifically,
$ \lim_{\lambda \to 0} \lambda M_{\lambda^{-2}t}=\mathbb{R}^{2}\times
S^{1}(\sqrt{2|t|})$. The classification of complete noncompact ancient solutions
under this assumption has been considered in \cite{CHH_wing} and
\cite{CHH_translator}. The uniqueness of the one-parameter family of compact
$O(2)$-symmetric ancient solutions constructed in \cite{DH} has been proved in
\cite{CDDHS}.  Analogous results in higher dimensions are believed to be true but
they haven't been shown yet.

Our goal in this paper is to apply ideas from topological dynamics to the initial
value problem defined by \eqref{eq-rmcf}.  Specifically, we introduce a topology on
the space $X$ of convex hypersurfaces of $\R^{n+1}$, and show that RMCF defines a
continuous local semiflow $\phi^\tau$ on this space.  Shrinking solitons are then
fixed points of the semiflow while ancient solutions are connecting orbits between
these fixed points.  The topology of the space $X$ of hypersurfaces together with the
continuity of the semiflow will allow us to draw conclusions about the existence and
multiplicity of connecting orbits, i.e.~of ancient convex solutions to MCF.

\subsection{The space of convex hypersurfaces}
We call a set $M\subset\R^{n+1}$ a \emph{complete convex hypersurface} if it is the
boundary of some closed convex set $C\subset \R^{n+1}$ with nonempty interior.  The
phase space on which we intend to let RMCF act is
\begin{equation}
X=\left\{ C\subset\R^{n+1} \mid
C\text{ is closed, convex and has nonempty interior}
\right\}.
\end{equation}
We do not impose any regularity conditions beyond what follows from convexity, and
the sets $C\subset \R^{n+1}$ and their boundaries $M=\partial C$ need not be bounded.

Occasionally we abuse language and refer to a convex set $C$ and its boundary
$M=\partial C$ interchangeably.  We will often refer to a family of closed convex
sets $\{C_t\subset\R^{n+1}\mid t_0<t<t_1\}$ and say that it evolves by MCF (or RMCF)
if the boundaries $M_t=\partial C_t$ are smooth hypersurfaces in $\R^{n+1}$ evolving
by MCF (or RMCF).  Sometimes we will also say that a family of convex sets
$\{C_t\}_{t\in (-\infty,T)}$ is an ancient solution to MCF if the boundaries
$M_t = \partial C_t$ are an ancient solution to MCF.

To define a topology on the space $X$ we use the distance function $d_C$ to a convex
set $C$, defined by
\begin{equation}
\label{eq-dist-fun}
d_C(x) \stackrel{\rm def}{=} \sup_{y\in C} \|x-y\|.
\end{equation}
We endow $X$ with a topology in which a sequence $C_k \in X$ converges to $C\in X$ if
the distance functions $d_{C_k}$ converge pointwise to $d_C$.  This topology is
metrizable and turns out to have good compactness properties which will allow us to
use arguments from topological dynamics.  The topology is equivalent to the topology
of local Hausdorff convergence.  See section~\ref{sec-convex-topology}.

\subsection{The semiflow on $X$}
To construct the semiflow on $X$ we establish well-posedness of the initial value
problem.  This is well-known in many cases, such as for compact initial
hypersurfaces, smooth hypersurfaces with bounded curvature, or hypersurfaces that are
graphs, but we could not find a reference for the existence, and especially
uniqueness of solutions starting from arbitrary convex hypersurfaces without any
further assumptions\footnote{Existence and uniqueness in the sense of viscosity
  solutions is established in full generality in \cite{GigaBook}, however viscosity
  solutions may ``fatten,'' in which case the viscosity solution becomes the region
  contained between two different smooth solutions.  In the terminology of viscosity
  solutions our uniqueness result shows that solutions do not fatten.}.  By combining
Ilmanen's construction \cite{IlmanenIUMJ} of the ``maximal subsolution'' for
existence with recent work of the second author and M.Saez \cite{DS} for uniqueness
we obtain in section \ref{sec-ex-uniq} the following theorem.

\begin{theorem}\label{thm-mcf-exist-unique} 
Let $C_0\subset \R^{n+1}$ be any closed convex set, bounded or unbounded, with
nonempty interior.  Then there exist a number $T >0$ and unique family of smooth
convex sets $\{ \hat C_t \}_{t \in (0,T)}$ with non-empty interior, such that
$\hat M_t := \partial \hat C_t$ evolves by classical MCF and for which
$\hat C_t \to C_0$, locally in the Hausdorff metric, as $t \to 0$. Equivalently,
$M_t$ is a smooth MCF solution starting at~$M_0 = \partial C_0$.
\end{theorem}
\bigskip

This theorem provides local existence and uniqueness for evolution by the unrescaled
MCF.  To define solutions to RMCF we let $C_0\in X$ be any given initial set, and
consider the solution $\{\hat C_t \mid 0\leq t<T\}$ to MCF that
Theorem~\ref{thm-mcf-exist-unique} provides.  Then the evolution by RMCF starting at
$C_0$ is
\[
\phi^\tau(C_0) = C_\tau := e^{\tau/2} \, \hat C_{1-e^{-\tau}}
\]
for $0\leq \tau<\taum$ where $\taum = -\ln(1-T)$ if $T<\infty$ and $\taum=\infty$ if
$T\geq 1$.

The map $(C_0, \tau)\mapsto \phi^\tau(C_0)$ defines a local semiflow on the space
$X$.  Its domain is an open subset of $X\times[0, \infty)$ and the map is continuous
on its domain (see section~\ref{ss-semiflow-continuous}).

\subsection{Rotation invariance}
The group of rotations $\SO(n+1, \R) = \SO_{n+1}$ acts on $\R^{n+1}$, and hence it
acts on the space $X$ of convex hypersurfaces.  The group action of $\SO_{n+1}$ on
$X$ is continuous, $X$ is metrizable, and because the group $\SO_{n+1}$ is compact,
the quotient space $X/\SO_{n+1}$ is again metrizable.

Since MCF and RMCF are invariant under rotations, the semiflow $\phi^\tau$ is
equivariant with respect to the action of $\SO_{n+1}$,
i.e.~$\phi^\tau\bigl(\mc R\cdot C\bigr)=\mc R\cdot \phi^\tau(C)$ for all $C\in X$ and
$\mc R\in \SO_{n+1}$.  It follows that $\phi^\tau$ defines a local semiflow on
$X/\SO_{n+1}$, which by abuse of notation we again denote by $\phi^\tau$.

\subsection{Huisken's functional}\label{sec-Huisken}\itshape
For each convex subset $C\in X$ we define the \emph{Huisken energy}
\begin{equation}\label{eqn-intro-HE}
\mc H(C) = \frac{1}{(4\pi)^{n/2}}\int_{\partial C} e^{-\|x\|^2/4} dH^n_C
\end{equation}
in which $H^n_C$ is $n$-dimensional Hausdorff measure on $\partial C$.
\upshape\medskip

Huisken's monotonicity formula \cite{Hu1990} implies that
$\tau\mapsto \hu\bigl(\phi^\tau(C)\bigr)$ is a non increasing function, which is in
fact strictly decreasing unless $C$ is a shrinking soliton.

The normalizing factor $(4\pi)^{n/2}$ is chosen so that
\[
(4\pi)^{-n/2} \int_Pe^{-\|x\|^2/4} dH^n = 1
\]
for any hyperplane $P\subset\R^{n+1}$ containing the origin.

\begin{conjecture}
$\hu(C) < 2$ for every closed convex set $C$ with nonempty interior.
\end{conjecture}

\subsection{The fixed points}\label{sec-fixed-points}
A convex set $C\in X$ is a fixed point for the semiflow $\phi^\tau$ exactly if the
hypersurface $\partial C$ is a shrinking soliton.  Huisken \cite{Hu1990} showed that
there are only two kinds of convex shrinking solitons:
\begin{enumerate}[-]
\item \emph{hyperplanes: } up to a rotation $\partial C$ is the hyperplane
$\mathbb P :=\{0\}\times\R^n$ and $C\subset\R^{n+1}$ is the half-space
$[0, \infty)\times\R^n$
\item \emph{generalized cylinders: } up to a rotation
$\partial C = \sphere^k\times\R^{n-k}$ for some $k\in\{1, \dots, n\}$, where
$\sphere^k$ is the $k$-dimensional sphere in $\R^{k+1}$ with radius $\sqrt{2k}$.  The
convex set~$C$ is $\mathbb B^{k+1}\times \R^{n-k}$, i.e.~the convex hull of
$\sphere^k\times\R^{n-k}$.
\end{enumerate}
The semiflow $\phi^\tau$ on the quotient space $X/\SO_{n+1}$ therefore has finitely
many fixed points, which we denote by $\Sigma^1$, \dots, $\Sigma^n$, and $\Pi$: by
definition $\Sigma^k\in X/\SO_{n+1}$ is the equivalence class of the cylinders
$\sphere^k\times\R^{n-k}\subset\R^{n+1}$, i.e.
\begin{equation}
\label{eq-sigma-k}
\Sigma^k := \SO_{n+1}\cdot \bigl(\sphere^k\times\R^{n-k}\bigr),
\end{equation}
where we abuse language and identify the hypersurfaces $\sphere^k\times\R^{n-k}$ with
the convex sets they enclose.  Similarly, $\Pi$ is the equivalence class of
hyperplanes through the origin,
\begin{equation}
\label{eq-plane}
\Pi := \SO_{n+1}\cdot \mathbb P
\end{equation}
where we have committed another abuse of language by identifying the half-space
$[0, \infty)\times\R^n$ with its boundary $\mathbb P$.

We will refer to the numbers $\hu(\Pi)$, and $\hu(\Sigma^k) (k=1, \dots, n)$ as the
\emph{critical values} of the Huisken functional.  They are ordered by
\begin{equation}\label{eq-critical-values}
1=\hu(\Pi) < \hu(\Sigma^n) < \hu(\Sigma^{n-1}) < \cdots < \hu(\Sigma^1)<2.
\end{equation}

Even though $\Sigma^0$, the generalized cylinder with $k=0$, is excluded from the
class~$X$, it is tempting to include it and regard $\sphere^0 = S^0_0$ as the ``zero
dimensional sphere in $\R^1$ with radius zero,'' i.e.~the origin in $\R^1$ counted
with multiplicity two.  In this interpretation the self-shrinkers $\Sigma^0$ are
\emph{double} hyperplanes in $\R^{n+1}$ through the origin, and could be seen as the
backward limits of the Bourni-Langford-Tinaglia pancake solutions \cite{BLT21}.  The
Huisken energy of $\Sigma^0$ would be twice that of the plane $\Pi$,
i.e.~$\hu(\Sigma^0) = 2$, which is why the assumption $\hu(C)<2$ appears frequently
in this paper.

\subsection{The invariant set  $I(h_0, h_1)$}\label{dfn-invariant}
\itshape For any $h_0, h_1>0$ with $h_0<h_1$ we set
\[
X(h_0,h_1) := \bigl\{ C\in X \mid h_0\leq \hu(C) \leq h_1\bigr\},
\]
and we define the \emph{invariant set} $I(h_0, h_1)$ of RMCF to be the set of all
$C\in X(h_0, h_1)$ for which there is an entire solution $\{C_\tau\}_{\tau\in\R}$ of
RMCF with $C_0=C$ and with $h_0\leq \hu(C_\tau)\leq h_1$ for all
$\tau\in\R$.\upshape\medskip

\noindent
Isolated invariant sets in dynamical systems were introduced by Charles Conley who
described his theory in \cite{Conley}.  While our topological arguments will be
inspired by Conley's theory this paper is meant to be self-contained, and will not
directly rely on results from \cite{Conley}.  In Conley's terminology $X(h_0,h_1)$ is
an \emph{isolating block} for the semiflow on $X$ defined by RMCF, and $I(h_0,h_1)$
is the \emph{isolated invariant set} it contains.

If $0<h_0<h_1<2$ then $I(h_0, h_1)$ is a compact subset of the space $X$, as we show
in Lemma~\ref{lemma-I-compact}.  Since the orthogonal group $\SO_{n+1}$ is compact,
the quotient $I(h_0,h_1)/\SO_{n+1}$ is also a compact metrizable space.  By
construction $I(h_0,h_1)$ is invariant under the semiflow $\phi^\tau$, and Huisken's
monotonicity formula implies that $\hu$ is a Lyapunov function for the semiflow.  It
follows that the invariant set $I(h_0, h_1)$ consists of fixed points of the flow and
orbits connecting these fixed points (see Proposition \ref{prop-descr-I}).  The
semiflow on the quotient has finitely many fixed points, so the invariant quotient
set $I(h_0, h_1)/\SO_{n+1}$ consists exactly of those fixed points $\Sigma^k$ and
possibly $\Pi$ whose Huisken energy lies between $h_0$ and $h_1$, as well as all
connecting orbits between such fixed points.

\subsection{Point symmetric sets}\label{sec-point-symm}
We now consider the set of closed convex $C\subset\R^{n+1}$ that are invariant under
point reflection
\[
X_s = \{C\in X \mid C = -C\},
\]
as well as the invariant set $I_s(h_0,h_1) := I(h_0,h_1)\cap X_s$.  We will also
consider the corresponding quotient space $X_s /\SO_{n+1}$, and its invariant set
$I_s(h_0,h_1)/\SO_{n+1}$.

The generalized cylinders $\Sigma^k$ belong to $X_s$, but the plane $\Pi$ does not,
because the plane is the boundary of the half-space $[0,\infty)\times\R^n$, which is
not point symmetric.

Our motivation for restricting to this class is that it provides a natural way to
single out solutions $M_t$ of RMCF for which $M_t$ is either compact or of the form
$K^k\times \R^{n-k}$ for some compact convex hypersurface $K^k\subset\R^{k+1}$.  See
Lemma~\ref{lemma-rescaling-time}.

If $h_0 \in\bigl(0, \mc H(\Sigma^n)\bigr)$ and $h_1\in\bigl(\mc H(\Sigma^1), 2\bigr)$
then any $C\in I(h_0, h_1)$ either is a fixed point or it lies on an orbit connecting
two fixed points.  In both cases monotonicity of the Huisken energy implies
$\hu(\Sigma^n)\leq \hu(C)\leq \hu(\Sigma^1)$ and so the set
\begin{equation}
\label{eq-Isn}
I_s^n := I(h_0, h_1)
\end{equation}
does not depend on $h_0,h_1$ as long as $h_0\in\bigl(0, \hu(\Sigma^n)\bigr)$ and
$h_1\in\bigl(\hu(\Sigma^1), 2\bigr)$.

\begin{figure}[t]
\centering \includegraphics[width=\textwidth]{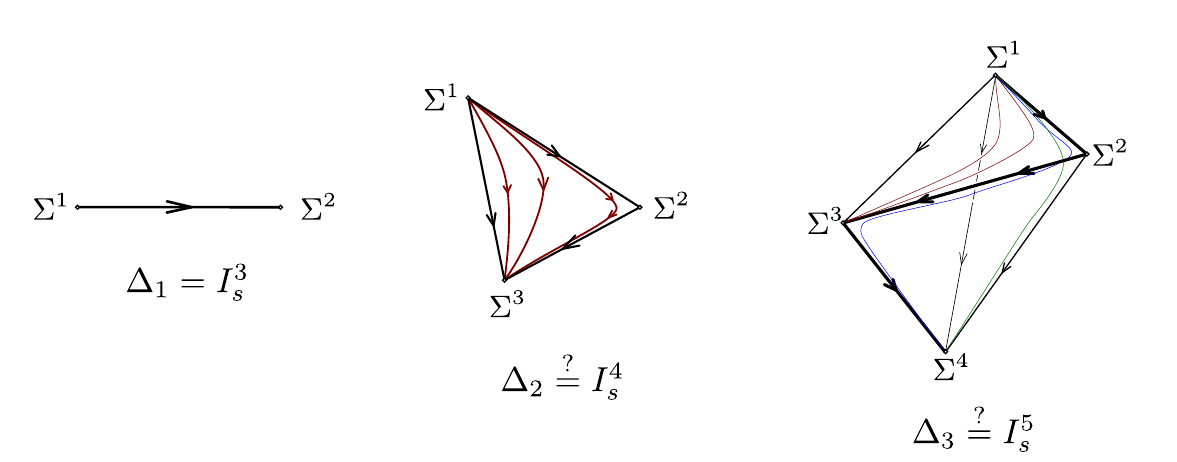}
\caption{The standard flow on the $n$ simplex for $n=1,2, 3$ (from left to right).
  Conjecture~\ref{conj-main} asserts that the flow on the invariant set $I_s^n$ is
  topologically conjugate to the standard flow on $\Delta_{n-1}$.  The fixed points
  $\Sigma^k$ are the equivalence classes under rotation of the generalized cylinders
  $S^k_{\sqrt{2k}}\times\R^{n-k}\subset\R^{n+1}$.}
\label{fig:standardflow}
\end{figure}

The set $I_s^n$ and the semiflow $\phi^\tau$ on $I_s^n$ are the objects that interest
us in this paper.  The following theorem gathers several topological properties of
the set $I_s^n$.

\begin{theorem}
\label{thm-main}
Assume that $0 < h_0 < \mc H(\Sigma^n)$ and $\mc H(\Sigma^1) < h_1 < 2$.  Then
\begin{enumerate}[{\em (a)}]
\item The invariant set $I_s^n$ consists exactly of all fixed points $\Sigma^k$
$(1\leq k\leq n)$, and all connecting orbits between these fixed points.
\item $I_s^n$ is path connected.
\item $I_s^n$ is a compact subset of $X_s$.
\item There exists a nested sequence of contractible compact neighborhoods
$N_1\supset N_2\supset\cdots$ of $I_s^n$ with $I_s^n = {\bigcap}_{k\geq 1}N_k$, and
hence the Čech cohomology groups of $I_s^n$ are those of a one-point space,
i.e.~$\check H^k(I_s^n)=0$ for $k\geq 1$ and $\check H^0(I_s^n)=\Z$.
\end{enumerate}
\end{theorem}

\subsection{Uniqueness of simple connecting orbits}\label{sec-simple-connections-unique}
Our classification \cite{ADS1,ADS2} of noncollapsed solutions connecting cylinders
$\sphere^{n-1}\times\R$ with the sphere $\sphere^n$, and the subsequent
classification by Du and Haslhofer \cite{DH} of noncollapsed solutions connecting
generalized cylinders $\sphere^{n-k}\times\R^k$ with $\sphere^n$, combined with the
non-collapsing results of Bourni, Langford, and Lynch \cite{BLL} imply the following:
\begin{theorem}\label{thm-uniqueness-of-simple-connections}~
\begin{enumerate}[\upshape(a)~]
\item For each $k\in\{1, \dots, n-2\}$ there is a unique connecting orbit from
$\Sigma^k$ to $\Sigma^{k+1}$.
\item For any integers $k,\ell$ with $1\leq k < \ell\leq n-1$ there is a unique
connecting orbit from $\Sigma^k$ to $\Sigma^\ell$ that is
$\SO_{k+1}\times\SO_{\ell-k}$ symmetric.
\end{enumerate}
\end{theorem}
\begin{proof}
(a)~Since $\Sigma^k = \SO_{n+1}\cdot\bigl( S^k\times\R^{n-k}\bigr)$ a connecting
orbit from $\Sigma^k$ to $\Sigma^{k+1}$ must be of the form $C_\tau\times\R^{n-k}$
for $\tau\in\R$, where $C_\tau$ is a $k$ dimensional solution to RMCF in $\R^{k+2}$
which connects $S^k\times\R$ with $S^{k+1}$.  By \cite{BLL} such a solution is
non-collapsed, and hence the uniqueness result in \cite{ADS2} applies.

(b)~This was shown by Haslhofer and Du in \cite{DH} under the assumption that the
solutions are non collapsed.  Bourni, Langford, and Lynch \cite{BLL} showed that the
non-collapsedness condition is always satisfied.
\end{proof}

\subsection{There are no isolated orbits} The following result, which we prove in
\S\ref{sec-proof-no-isolation}, is a direct consequence of Theorem~\ref{thm-main}.
It shows that any complete solution to RMCF can be approximated by other complete
orbits of the semiflow.

\begin{theorem}
\label{thm-no-isolated-orbits} If $n > 3$ then the invariant set $I_s^n$ contains no
isolated orbits in the following sense: if
$\Gamma=\{C_\tau \mid \tau\in\R\}\subset I_s^n$ is a complete orbit, and if
$\mc U\subset X_s$ is an open set with $\mc U\cap \Gamma\neq\varnothing$, then there
is another complete orbit
\(\tilde \Gamma = \{\tilde C_\tau \mid \tau\in\R\}\subset I_s^n\) with
$\tilde\Gamma\cap\mc U\neq \varnothing$ and $\Gamma\cap\tilde\Gamma=\varnothing$.
\end{theorem}

By letting the open set $\mc U$ be an arbitrarily small neighborhood of some given
point $C_{\tau_0}\in\Gamma$ we conclude that there is a sequence of orbits
$\Gamma_i = \{C_\tau^i\mid \tau\in\R\}$ with $C_{\tau_0}^i\to C_{\tau_0}$, and
$\Gamma_i\cap\Gamma=\varnothing$ for all $i$.  \medskip

\subsection{A shadowing lemma}
The classical Shadowing Lemma from smooth dynamical systems \cite[Theorem 18.1.7,
page 569]{KatokHasselblatt95} concerns hyperbolic invariant sets for a smooth flow on
a finite dimensional manifold.  In particular, if suitable transversality and
hyperbolicity conditions hold, it implies that a flow with three given fixed points
$A$, $B$, and $C$, and two connecting orbits, one from $A$ to $B$, and another from
$B$ to $C$ will also have connecting orbits from $A$ to $C$ which are arbitrarily
close to the union of the two orbits $A\to B$ and $B\to C$. The proof involves a
gluing construction involving the Banach fixed point theorem and the linearization of
the flow near the two connecting orbits $A\to B\to C$.

We prove an analogous result for the semiflow $\phi^\tau$ on $I^n_s$ in the case
$n=4$.  A remarkable difference with the proof of the Shadowing Lemma from smooth
dynamics is that our arguments do not require any transversality or hyperbolicity.
Instead, the main technical ingredient of the proof is the Huisken-Gage-Hamilton
theorem which says that smooth compact convex solutions of MCF shrink to round
points.

To formulate our shadowing lemma, we let $\Gamma(k,\ell)$ be the unique connecting
orbit from $\Sigma^k$ to $\Sigma^\ell$ with $\SO_{k+1}\times\SO_{\ell-k}$ symmetry
mentioned in Theorem \ref{thm-uniqueness-of-simple-connections}.
\begin{theorem}\label{thm-shadowing}
Let $\mc N\subset X_s$ be an open neighborhood of $\Gamma(1, 2)\cup\Gamma(2, 3)$.
Then $\mc N$ contains a connecting orbit $\Gamma_*$ from $\Sigma^1$ to $\Sigma^3$
$($see Figure \ref{fig:shadowing}$)$.
\end{theorem}
The proof is given in Section~\ref{sec-hetero}.
\begin{figure}[hb]
\centering \includegraphics[width=0.3\textwidth]{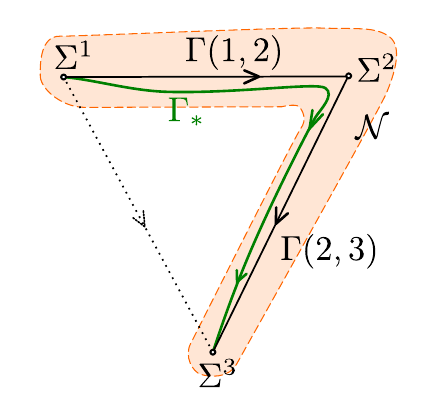}
\caption{The shadowing theorem~\ref{thm-shadowing} }
\label{fig:shadowing}
\end{figure}

\subsection{The standard flow on the $n$-simplex}
Let
\[
\Delta_n = \bigl\{(a_0, \dots, a_n)\in\R^{n+1} \mid a_i\geq 0,\;
a_0+\cdots+a_n=1\bigr\}
\]
be the standard $n$-dimensional simplex.  The standard flow on $\Delta_n$ is the flow
given by
\[
\sigma^\tau (a) = \frac{e^{\tau A}a}{\langle u, e^{\tau A}a\rangle}
\]
where $u = (1, 1, \dots, 1)\in \R^{n+1}$ and where $A$ is the $(n+1)\times(n+1)$
diagonal matrix with diagonal entries $a_{kk}=k$ ($k=0, 1, 2\dots, n$).

In this flow the unit vectors $e_k$ are fixed points, $e_n$ is an attractor, and
$e_0$ is a repeller.

The existing uniqueness results about mean curvature flow ancient solutions
(\cite{ADS1}, \cite{CHHW}, \cite{DH} and \cite{CDDHS}) and arguments in the proof of
Theorem \ref{thm-main} inspire us to propose the following conjecture.

\begin{conjecture}\label{conj-main}
Let $n\geq 2$ and suppose that $0<h_0<\hu(\Sigma^n)$ and $\hu(\Sigma^1)<h_1<2$.  Then
the invariant set $I_s(h_0,h_1)/\SO_{n+1} = I(h_0, h_1) \cap X_s/\SO_{n+1}$ is
homeomorphic to an $n-1$ dimensional simplex, and the semiflow $\phi^t$ on
$I_s(h_0,h_1)/\SO_{n+1}$ is topologically conjugate\footnote{~This means that there
  is a homeomorphism $h:\Delta_{n-1}\to I_s^n$ such that
  $h\circ \sigma^\tau = \phi^\tau\circ h$ for all~$\tau\geq 0$.} to the standard flow
on $\Delta_{n-1}$.
\end{conjecture}
The uniqueness theorem~\ref{thm-uniqueness-of-simple-connections} implies that the
conjecture holds for $n=2$, and in section \ref{sec-conj} we present strong evidence
that the conjecture holds in dimension $n = 3$.

\subsection{Outline of the paper} In section \ref{sec-ex-uniq} we prove Theorem
\ref{thm-mcf-exist-unique}.  We use Ilmanen's definition \ref{def:emcf-solution} to
get the existence of a solution to the mean curvature flow equation \ref{eq-mcf} for
a short time, and then show that the solution becomes instantaneously smooth for
$t > 0$.  After that, using the arguments in \cite{DS} and convexity we show the
uniqueness of our solution.  In section \ref{sec-convex-topology} we introduce the
weighted measure on a set $\bar X$ of all closed convex sets in $\R^{n+1}$, a
weighted distance function between any two closed convex sets in $\R^{n+1}$.  We show
some properties of that distance function, continuity of the Huisken's functional
$\mathcal{H} : X \to \R$, and compactness of certain subsets of $\bar{X}$ with
respect to the topology induced by the introduced weighted distance function (here
$X$ is the set of all closed convex sets in $\R^{n+1}$ with nonempty interior).  In
section \ref{sec-semiflow} we define the semiflow by the rescaled mean curvature flow
(RMCF) n $X$.  We show that the semiflow defined that way is a continuous function in
our topology induced by the distance function.  In section \ref{sec-invariant-set} we
introduce the invariant set $I(h_0,h_1)$, as the set consisting of all $C \in X$ for
which there is an entire solution $\{C_{\tau}\}_{\tau\in \R}$ of RMCF with $C_0 = C$
and with $h_0 \le \mc H(C_{\tau} \le h_1$, for all $\tau\in \R$.  We show that this
invariant set is a compact subset of $X$.  We also give here a nice description of
the invariant set as the set consisting exactly of fixed points of RMCF, as well as
all connecting orbits between the fixed points.  In section \ref{sec-topology} we
show several topological features of the invariant set $I_s(h_0,h_1)$.  More
precisely, we show that the set $I_s(h_0,h_1)$ is path connected in the distance
topology and that it can be written as an intersection of a nested sequence of
contractible neighborhoods of $I_s(h_0,h_1)$, implying that the Čech cohomology
groups of $I_s^n$ are those of one point space.

In sections \ref{sec-hetero} and \ref{sec-conj} we mostly consider the case $n = 3$,
in which case Conjecture \ref{conj-main} states that the set $I(h_0,h_1)\cap X_s)$ is
homeomorphic to a two simplex, whose edges are $\Gamma(i,j)$, where $i < j$ and
$i,j\in \{1,2,3\}$ are defined by \eqref{eq-edges}.  Section \ref{sec-hetero} is
dedicated to proving Theorem \ref{thm-shadowing}, which says that every
neighborhood of $\Gamma(1,2)\cup \Gamma(2,3)$ contains an orbit from $\Sigma^1$ to
$\Sigma^3$, or in other words that there is a sequence of orbits from $\Sigma^1$ to
$\Sigma^3$ converging to $\Gamma(1,2)\cup\Gamma(2,3)$.  Note that the uniqueness of
$\Gamma(1,2)$ and $\Gamma(2,3)$, up to rigid motions follows by \cite{CHHW}, and the
uniqueness of one parameter family of orbits connecting $\Sigma^1$ to $\Sigma^3$
`almost' follows by \cite{CDDHS}.  Theorem \ref{thm-shadowing} follows by
Lemmas \ref{prop-nbhd-sigma2} and \ref{prop-nbhd-sigma2-backwards} which discuss the
behavior of rescaled mean curvature solution $\{C_{\tau}\}_{\tau\in \R}$, as
$\tau\to\infty$ and $\tau\to -\infty$, respectively, given that $C_0$ lies in a
neighborhood of $\Sigma^2$.  In section \ref{sec-conj} we prove Theorem
\ref{thm-conjecture} which gives us a strong evidence that Conjecture \ref{conj-main}
holds at least for dimension $n = 3$.  The proof of Theorem \ref{thm-conjecture}
yields the proof of Conjecture \ref{conj-main} modulo a hypothesis (see Remark
\ref{rem-bel} that is expected to be removed in future works.)

\section{Mean curvature flow with convex initial data}
\label{sec-ex-uniq}

If $C_0 \subset \R^{n+1}$ is compact and if its boundary $\partial C_0$ is a smooth
hypersurface, then the short time existence theory for parabolic equations provides a
smooth family of convex sets $C_t$ on a short time interval whose boundaries
$\partial C_t$ evolve by MCF.  Huisken \cite{Hu} showed that this smooth solution
shrinks to a point as $t\nearrow T$ for some $T<\infty$.  We will often denote this
solution by $C_t$ (rather than $M_t:=\partial C_t$) and say that it evolves by Mean
Curvature Flow, meaning that $C_t$ is a family of smooth convex sets whose boundaries
$M_t=\partial C_t$ are hypersurfaces in $\R^{n+1}$ evolving by classical MCF and
$C_t \to C_0$, as $t \to 0$.

The goal in this section is to establish Theorem \ref{thm-mcf-exist-unique} as stated
in the Introduction, that is to show the existence, uniqueness, and regularity for
mean curvature flow with convex initial data, without assuming the initial
hypersurface is necessarily bounded or smooth beyond what is required by convexity.

In the next subsection \ref{ss-defn} we will define a notion of weak MCF starting at
any closed convex set $C_0$ (bounded or unbounded) simply by approximation by smooth
compact convex solutions.  For general closed convex initial sets $C_0$ our solutions
coincide with Ilmanen's maximal subsolutions and for this reason we will refer to
them as the \emph{Ilmanen evolution by MCF}.  We will then present the basic
properties of such solutions which will lead us to the regularity result of Lemma
\ref{lemma-reg}, which states that the Ilmanen evolution by MCF is smooth for $t >0$
and therefore a classical solution to MCF.  This leads to Theorem
\ref{prop-existence} establishing the existence of a smooth MCF solution $C_t$, as
stated in Theorem \ref{thm-mcf-exist-unique}.  The uniqueness of such solutions will
then be shown in Theorem \ref{thm-uniqueness} in section~\ref{ss-uniq}.

\subsection{Definition and basic properties of Ilmanen's evolution by MCF}
\label{ss-defn}

\begin{definition}[Ilmanen convex evolution by MCF]
\label{def:emcf-solution}
If $C_0\subset \R^{n+1}$ is a closed convex set with nonempty interior then we define
its \emph{Ilmanen evolution by MCF} to be
\[
C_t^{\rm Ilm} := \overline{\bigcup_D D_t}
\]
where the union is over all smooth compact solutions $\{D_s:0\leq s <S\}$ to MCF with
$D_0\subset \interior C_0$.
\end{definition}

We now prove some basic properties of the solutions $C_t^{\rm Ilm}$ defined above.

\begin{prop}\label{prop-solution-from-sequences}
Let $C_0\subset\R^{n+1}$ be closed convex with nonempty interior.  For any sequence
of compact smooth bounded convex subsets $D_m\subset\interior C_0$ with
\[
D_m\subset\interior D_{m+1}\text{ and } C_0=\overline{\cup_{m=1}^\infty D_m}
\]
the Ilmanen flow $C^{\rm Ilm}_t$ of $C_0$ is given by
$\overline{\cup_1^\infty D_{m,t}}$ where $D_{m,t}$ is the smooth MCF starting from
$D_m$.
\end{prop}
\begin{proof}
By definition $\overline{\cup_{m=1}^\infty D_{m,t}} \subset C_t$.  On the other hand,
if $D\subset C_0$ is compact, convex, and smoothly bounded, then
$\{\interior D_m\}_{m\in\N}$ is an open covering of $D$, so that for some $m\in\N$
one has $D\subset D_m$.  The evolution $D_t$ by MCF of $D$ therefore satisfies
$D_t\subset D_{m,t}$.  This implies
$C_t \subset \overline{\cup_{m=1}^\infty D_{m,t}}$.
\end{proof}

\begin{prop}
If $\{C_t \mid 0\leq t<T\}$ is an Ilmanen evolution of $C_0$, and if $a\in(0,T)$,
then $\{C_{t+a} \mid 0\leq t<T-a\}$ is the Ilmanen evolution of $C_a$.
\end{prop}
\begin{proof}
Let $D_m$ be a smooth compact and convex sequence of sets with
$\overline{\cup_m D_m} = C_0$, and let $\{D_{m,t} \mid 0\leq t<T_m\}$ be the
corresponding smooth MCF solutions.  Then $D_{m,a}$ are smooth, compact, convex, and
by Proposition~\ref{prop-solution-from-sequences}, $C_a=\overline{\cup_m D_{m,a}}$.
Since $\{D_{m, a+t}\mid 0\leq t<T_m-a\}$ are the smooth MCFs starting from
$D_{m, a}$, Proposition~\ref{prop-solution-from-sequences} implies that the Ilmanen
evolution starting from $C_a$ is indeed $\overline{\cup_m D_{m, a+t}} = C_{a+t}$.
\end{proof}

\begin{prop}\label{prop-interior-empties}
If $\{C_t \mid 0\leq t < T\}$ is an Ilmanen evolution then $C_t\supset C_s$ for all
$0\leq t\leq s<T$.  Moreover, if $T<\infty$, then the set $\bigcap_{t\in[0, T)} C_t$
is closed and has empty interior.
\end{prop}
\begin{proof}
All approximating compact smooth convex mean curvature flows $\{D_t\}$ in the
definition \ref{ss-defn} of the Ilmanen flow $\{C_t\}$ are shrinking, so by that
definition the sets $C_t$ also shrink in time.

Each $C_t$ is closed, so their intersection is also closed.

Suppose that the interior of $\bigcap_{0\leq t<T} C_t$ is nonempty, i.e.~suppose that
there exist a point $p\in \R^{n+1}$ and $r>0$ such that $B_{2r}(p) \subset C_t$ for
all $t<T$.  Choose $t_1$ with $T-\frac{r^2}{2n} <t_1 <T$.  By definition of the
Ilmanen flow there is a smooth compact MCF $D_t$ that is defined for $0\leq t < t_2$
for some $t_2\in (t_1, T)$, and that satisfies $D_0\subset C_0$ and
$B_r(p)\subset D_{t_1}$.  The maximum principle implies that $D_t$ contains
$B_{\sqrt{r^2-2n(t-t_1)}}(p)$ for all $t\leq t_1+\frac{r^2}{2n}$.  In particular,
$D_t$ is nonempty and contains the point $p$ for all $t<t_1+\frac{r^2}{2n}$.  Since
$t_1+\frac{r^2}{2n} > T$ it follows that the Ilmanen flow is defined at least for all
$t<t_1+\frac{r^2}{2n}$, and in particular is defined and nonempty beyond $t=T$.

This contradicts the assumption that $\{C_t\mid 0\leq t< T\}$ is an Ilmanen flow, and
thus we have shown that $\bigcap_{t<T}C_t$ has no interior.
\end{proof}
\begin{definition} \label{def-epsilon-core} For any convex set $C$ we define the
\emph{$\varepsilon$-core of $C$} to be
\[
K_\varepsilon C := \{ x\mid \bar B_\varepsilon(x)\subset C\}.
\]
\end{definition}

\begin{lemma}\label{lemma-Ilm-is-Hu}
If $C_0$ is compact and has smooth boundary, then $C_t^{\rm Ilm} = C_t$, where
$C_t^{\rm Ilm}$ is the Ilmanen evolution of $C_0$, and $C_t$ is the smooth classical
solution that appears in Huisken's theorem.
\end{lemma}
\begin{proof}
By the maximum principle for compact smooth solutions of MCF, any solution
$D_t^{\rm Ilm}$ of MCF with $D_0\subset C_0$ is contained in the classical solution
$C_t$.  Thus $C_t^{\rm Ilm} \subset C_t$.

On the other hand, the sets $K_\varepsilon C_0 $ are compact, convex, and smoothly
bounded, at least if $\varepsilon>0$ is small enough.  It follows from the regularity
theory for smooth compact solutions of MCF that the classical solutions
$C_{\varepsilon, t}$ to MCF starting from $K_\varepsilon C_0$ converge smoothly to
$C_t$ on any compact time interval $[0, T']\subset[0, T)$.  This implies
$C_t\subset C_t^{\rm Ilm}$.
\end{proof}

\begin{prop}[Continuity in time of MCF] The Ilmanen flow
$\{C_t^{\rm Ilm} \mid 0\leq t<T\}$ of any closed convex $C\subset \R^{n+1}$ with
nonempty interior satisfies
\[
K_{\varepsilon}C_t^{\rm Ilm} \subset C_s^{\rm Ilm} \subset C_t^{\rm Ilm} \quad \mbox{
  whenever } \quad 0\leq t\leq s\leq t+\varepsilon^2/2n\leq T.
\]

\end{prop}
\begin{proof}
We first consider the case $t=0$.  If $x\in K_\varepsilon C_0^{\rm Ilm}$ then
$\bar B_\varepsilon(x)\subset C_0^{\rm Ilm}$.  The MCF at time $s$ starting from
$\bar B_\varepsilon(x)$ is the ball $\bar B_{\varepsilon(s)}(x)$ with radius
$\varepsilon(s) = \sqrt{\varepsilon^2-2ns}$.  The definition of the Ilmanen flow
implies $\bar B_{\varepsilon(s)}(x) \subset C_s^{\rm Ilm}$ for
$s < \varepsilon^2/2n$, and hence $x\in C_s^{\rm Ilm}$ for
$0\leq s\leq \varepsilon^2/2n$.  Thus we have shown
\[
K_\varepsilon C_0^{\rm Ilm} \subset \hat C_s^{\rm Ilm} \subset C_0^{\rm Ilm} \text{
  for }0\leq s\leq {\varepsilon^2 / 2n}
\]
The general case follows by considering $\hat C_s^{\rm Ilm} = C_t^{\rm Ilm}+s$ for
$0\leq s<T-t$ and applying the same reasoning.
\end{proof}

\begin{lemma}[Comparison with compact smooth solutions of MCF] Assume that
$\{C_t^{\rm Ilm} \mid 0\leq t <T\}$ is an Ilmanen evolution by MCF and let
$\{K_t \mid t_0\leq t <t_1\}$ be a family of compact smoothly bounded domains whose
boundary evolves by classical MCF.
\begin{enumerate}[{\em (a)}]
\item If $K_{t_0} \subset \interior C_{t_0}^{\rm Ilm} $ then
$K_t\subset C_t^{\rm Ilm} $ for all $t\in [t_0, t_1)$.
\item If $K_{t_0}$ and $C_{t_0}^{\rm Ilm} $ are disjoint then $K_t$ and
$C_t^{\rm Ilm} $ are disjoint for all $t\in [t_0, t_1).$
\end{enumerate}
\end{lemma}

\begin{proof}
(a) The compactness of $K_{t_0}\subset\interior C_{t_0}^{\rm Ilm}$ implies that there
is a compact smooth convex $D_0\subset \interior C_0$ with
$K_{t_0}\subset D_{t_0}^{\rm Ilm} $.  The comparison principle for smooth compact
solutions of MCF then implies that $K_t \subset D_t \subset C_t^{\rm Ilm} $ for all
$t\in (t_0, t_1)$.

(b) The statement is true for any smooth compact solution $\{D_t \mid 0\leq t<T\}$
contained in $C_t^{\rm Ilm}$, and therefore also holds for $C_t^{\rm Ilm}$.
\end{proof}

\begin{lemma}\label{lemma-spacetime-trace-closed}
The space-time trace of an Ilmanen evolution by MCF is closed, i.e.~if
$\{C_t^{\rm Ilm} \mid 0\leq t <T\}$ is an Ilmanen evolution by MCF, then the set
\[ [C]^{\rm Ilm} \stackrel{\rm def}= \{(t, x)\in[0, T)\times \R^{n+1} \mid x\in
C_t^{\rm Ilm}\}
\]
is a closed subset of $[0, T)\times\R^{n+1}$.
\end{lemma}
\begin{proof}
Consider a sequence of points $(t_k, q_k)\in [C]^{\rm Ilm}$ that converges to
$(\bar t, \bar q)$.  We will show that $(\bar t, \bar q)\in[C]^{\rm Ilm}$.

If there are infinitely many $k$ with $t_k\geq \bar t$ then we can pass to a
subsequence and assume that $t_k\geq \bar t$ for all $k$.  It follows from
$C_{t_k}\subset C_{\bar t}$ that $q_k\in C_{\bar t}$ for all $k$.  Since $C_{\bar t}$
is closed this implies $\bar q\in C_{\bar t}$.

If only finitely many $k$ satisfy $t_k\geq \bar t$ then we may assume after passing
to a subsequence that $t_k<\bar t$ for all $k$.  In this case we choose a point
$p\in \interior C_{\bar t}^{\rm Ilm}$, and let $\delta>0$ be small enough to ensure
$B_\delta(p)\subset C_{\bar t}^{\rm Ilm}$.  Since the sets $C_t^{\rm Ilm}$ shrink, it
follows that $B_{\delta}(p)\subset C_t^{\rm Ilm}$ for all $t\leq \bar t$.

Since $C_{t_k}^{\rm Ilm}$ is convex it contains the convex hull of $B_\delta(p)$ and
$q_k$.  If $z$ is an interior point of the convex hull of $\bar q$ and $B_\delta(p)$,
then $z$ also is an interior point for the convex hull of
$\{q_k\}\cup B_{\delta}(p)$, if $k$ is large enough.  It follows that there is an
$\varepsilon>0$ such that $B_\varepsilon(z)\subset C_{t_k}^{\rm Ilm}$ for all large
$k$.  Choose $k$ so large that $2n(\bar t-t_k)<\varepsilon^2$.  Then, by comparison,
we find that $B_r(z)\subset C_{\bar t}$ if $r=\sqrt{\varepsilon^2-2n(\bar t-t_k)}$.
In particular, $z\in C_{\bar t}^{\rm Ilm}$.

Thus we have shown that $C_{\bar t}^{\rm Ilm}$ contains the interior of the convex
hull of $\{\bar q\} \cup B_\delta(p)$.  Since $C_{\bar t}^{\rm Ilm}$ is by definition
closed, this implies $\bar q\in C_{\bar t}^{\rm Ilm}$.
\end{proof}

\medskip
\subsection{Regularity of Ilmanen's evolution by MCF}
\label{ss-reg}
\begin{lemma}\label{lemma-reg} Let $C_0\subset \R^{n+1}$ be a closed convex set
with nonempty interior, and let $\{C_t^{\rm Ilm} \mid 0\leq t <T\}$ be its Ilmanen
evolution by MCF.  Then $\partial C_t^{\rm Ilm}$ is smooth for each $t>0$, and the
family of hypersurfaces $\{\partial C_t^{\rm Ilm} \mid 0<t<T\}$ is a smooth solution
of MCF.
\end{lemma}
\begin{proof}
Let $\{C_t^{\rm Ilm} \mid 0\leq t <T \}$ be an Ilmanen evolution by MCF, and choose a
sequence $D_{k, t}\subset D_{k+1, t}$ of compact smooth convex solutions of MCF whose
union is $C_t^{\rm Ilm}$.

For a given time $t_0\in (0, T)$ and any point $z\in \partial C_{t_0}^{\rm Ilm}$
choose a line segment $pq$ whose midpoint is $z$ and whose two endpoints $p, q$
satisfy $p\in \interior C_{t_0}^{\rm Ilm}$, $q\not\in C_{t_0}^{\rm Ilm}$.  Since the
space time track of the evolution $\{C_t^{\rm Ilm}\}$ is closed it follows that there
is a small $\tau>0$ such that $q\not\in C_t^{\rm Ilm}$ for
$t\in (t_0-\tau, t_0+\tau)$.  Comparison with shrinking spheres shows that we may
also assume that $p\in \interior C_t^{\rm Ilm} $ for the same range of times $t$.  In
fact for sufficiently small $\delta>0$ we may assume that
$B_\delta(p)\subset C_t^{\rm Ilm}$ and $B_\delta(q)\cap C_t^{\rm Ilm}=\varnothing$
for $|t-t_0|<\tau$.

Since the sequence of approximating smooth compact sets $D_{k, t}$ increases to
$C_t^{\rm Ilm}$ we may even assume that $B_\delta(p)\subset D_{k, t}$ and
$B_\delta(q)\cap D_{k, t}=\varnothing$ for $|t-t_0|<\tau$.

We are now in the situation where we can apply Lemmas \ref{lem:convex-is-Lipschitz}
and \ref{lem:convex-in-cylinder-is-graph} to conclude that the portion of
$\partial D_{k, t}$ inside the convex hull of $B_{\delta/2}(p)\cup B_{\delta/2}(q)$
is a uniformly Lipschitz graph.  The Ecker--Huisken estimates \cite{EH2} or,
alternatively, the interior estimates for non-divergence form quasilinear parabolic
equations in Ladyzhenskaya-Solonikov-Uralceva \cite[Theorem 1.1, Ch VI, \S1]{LSU}
imply that the curvatures of $\partial D_{k, t}$ and all their higher derivatives are
locally uniformly bounded.  Since the boundaries converge monotonically, they
converge uniformly to $\partial C_t^{\rm Ilm}$ and we conclude that
$\partial C_t^{\rm Ilm}$ is indeed smooth.
\end{proof}

\subsection{The shadow of a MCF}\label{ss-shadow}

Here we derive some properties of MCF solutions $\partial C_t$, $t \in [0,T)$ for
which $C_t$ are non-compact, closed with nonempty interior, and contain no infinite
line.  It follows that each $C_t$ contains a half line whose direction we may assume
is $e_{n+1}$.  If $\pi:\R^{n+1}\to\R^n$ is the orthogonal projection along the
$x_{n+1}$-axis then we call $D_t:=\pi(C_t)\subset\R^n$ the \emph{shadow} of $C_t$.

\begin{lemma}[Evolution of the shadow $D_t$]
\label{lemma-shadow-moves-by-mcf}
Let $C_t$, $t \in (0, T)$ be any smooth MCF solution starting at a closed convex set
$C_0$ with nonempty interior.  Then the shadow $D_t = \pi (C_t)$ of $C_t$ is a smooth
MCF evolution starting at $D_0 = \pi (C_0)$.
\end{lemma}

\begin{proof}
For any $k >0$, let $C^k_t := C_t - k \, e_{n+1}$ be the translation of the
hypersurface $C_t$ in the $-e_{n+1}$ direction by $k$.  Denote by $\cC_t$ the
cylinder
\[
\cC_t := \{ (x', x_{n+1} ): x' \in \partial D_t, x_{n+1} \in \R \}.
\]
Since $ C_t $ is asymptotic to $\cC_t$ at infinity, for every $t >0$, the convexity
of $C_t$ implies that $C^k_t \to \cC_t$, as $k \to +\infty$, uniformly on compact
subsets of $\R^{n+1} \times (0, T)$.  The uniform convergence combined again with the
convexity of $C^k_t$ imply that the sequence $\{ C^k_t \}$ is uniformly Lipschitz on
compact subsets of $\R^{n+1} \times (0, T)$ and hence $C_t^k \to \cC_t$ smoothly by
the well known results in \cite{EH2}.  This implies that $\cC_t$ is a MCF solution
and in turn gives that $\partial D_t$ evolves by MCF as well.
\end{proof}

For every $t \in [0,T)$ the hypersurface $\partial C_t$ is a graph
$u(\cdot, t): D_t \to \R$ over the shadow $D_t$ of the set $C_t$.  At time $t=0$ the
shadow need not be an open set as $\partial C_0$ could contain vertical half lines
--- e.g.~consider $C_0=\{(x, y)\in\R^2 : |x|\leq 1, y\geq 0\}$.  In the following
Lemma we use the strong maximum principle to show that this is not the case for
$t>0$.

\begin{lemma}
\label{lemma-shadow-separation}
If $\partial D_t \neq \varnothing$ then the cylinder
$\cC_t := \{ (x', x_{n+1} ): x' \in \partial D_t, x_{n+1} \in \R \}$ does not touch
$\partial C_t$ and therefore we have
\begin{equation}
\label{eq-boundary-beh}
\lim_{x \to \partial D_t} u(x,t) = +\infty, \qquad \mbox{for}\,\, t \in (0,T).
\end{equation}

\end{lemma}
\begin{proof} The convexity of $\hM_t$ implies that $\hM_t \subset D_t \times \R$.
At time $t=0$ it may be that $\partial C_0 \cap \cC_0 \neq \varnothing$.  However,
the strong maximum principle guarantees that the MCF solution $\partial C_t$ and
$\cC_t$ do not touch for $t >0$.  Since $\partial C_t$ can be represented as a graph
$x_{n+1} = u(x,t)$ over the interior of $D_t$ and it is asymptotic to the cylinder
$\partial D_t \times \R$ at infinity, we conclude that \eqref{eq-boundary-beh} holds.
\end{proof}

\subsection{Existence of mean curvature flow starting at any convex initial data }
\label{ss-exist}

The discussion in subsections \ref{ss-defn} and \ref{ss-reg} together with known
results about the structure of convex sets lead to the following existence result for
smooth mean curvature flow of any closed convex set with nonempty interior $C_0$,
bounded or unbounded.  Similarly to the weak case in Definition
\ref{def:emcf-solution} we will say that $C_t$ is a smooth MCF evolution starting at
$C_0$ if $\partial C_t$ is a classical solution to MCF.

\begin{theorem}[Existence of smooth MCF]
\label{prop-existence} Let $C_0\subset \R^{n+1}$ be a closed convex set with nonempty
interior.  Then, there is $T>0$ and a smooth MCF solution $C_t$, $t \in (0, T)$
starting at $C_0$.  In addition, we have
\begin{enumerate}[{\em (a)}]
\item if the initial set $C_0$ is compact, then $C_t$ is compact and $T < \infty$.
\item if $C_0=\R^k \times D_0$, for some $0 < k <n$, then $C_t=\R^k \times D_t$,
where $D_t\subset\R^{n+1-k}$ is the mean curvature flow evolution starting from
$D_0$.
\item if $\partial C_0$ is complete, non-compact, and contains no infinite line, then
$\partial C_t$ is complete, non-compact, and contains no infinite line.
\end{enumerate}
In the latter case, the shadow $D_t = \pi (C_t)$ is a smooth mean curvature flow
evolution starting at $D_0 := \pi(C_0)$.

\end{theorem}

\begin{proof} Given any closed convex set with nonempty interior $C_0$, let
$C^{\rm Ilm}_t$ be its Ilmanen MCF evolution according to the Definition
\ref{def:emcf-solution}.  Lemma \ref{lemma-reg} shows that $C^{\rm Ilm}_t$ is
$C^\infty$ smooth for all $t >0$ and therefore it is a smooth classical MCF solution.
It is clear that (a), (b) and (c) hold by the definition of $C^{\rm Ilm}_t$ and
convexity.  Finally, the fact that that the shadow $D_t = \pi (C_t)$ evolves by MCF,
is shown in Lemma \ref{lemma-shadow-moves-by-mcf} (which holds for any MCF evolution,
not necessarily the one constructed here).

\end{proof}

\subsection{Uniqueness of convex MCF}
\label{ss-uniq}

It is well known that solutions to MCF that are defined using level set methods and
viscosity solutions can ``fatten.''  From the viewpoint of classical solutions
fattening occurs when one given initial hypersurface $\partial C$ admits more than
one smooth evolution by MCF.  For the Ilmanen flow as we have described them here,
fattening would correspond to a failure of the solution to depend continuously on its
initial data: fattening would occur if there were a decreasing family of Ilmanen
flows $C_{j, t}\supset C_{j+1, t}$ such that $\cap_{j=1}^\infty C_{j,t}$ is strictly
larger than the Ilmanen flow of $\cap_{j=1}^\infty C_{j,0}$.

In 1990 Ilmanen gave an example of fattening \cite{IlmanenProcSymp} in which
$\partial C_0$ was smooth, but not compact\footnote{~Ilmanen's example was
  $C=\{(x, y)\in\R^2 \mid 0\leq y\leq (1+x^2)^{-1}, x\in\R\}$.  One forward evolution
  is $C_t=\{(x, y)\mid 0\leq y\leq u(x, t), x\in\R\}$ where $u$ is the solutions to
  $u_t=\frac{u_{xx}}{1+u_x^2}$, $u(x, 0) = (1+x^2)^{-1}$.  In the other forward
  solution we imagine the two ends of the region $C$ to be connected ``at
  $(\pm\infty, 0)$'' and construct a closed curve $\gamma_t$ that evolves by Curve
  Shortening and that converges to $C$ as $t\searrow0$.  This is made possible by the
  fact that the area of $C$ is finite.  The second solution $\bar C_t$ is the region
  enclosed by $\gamma_t$ and is strictly contained in the first solution $C_t$.}.
Many other examples were discovered afterwards.  On the other hand, it is also well
known that this phenomenon does not appear when the initial hypersurface
$\partial C_0$ is smooth, mean convex, and \emph{compact}.  In this section we prove
the uniqueness of smooth convex solutions starting at convex hypersurfaces that are
not necessarily smooth or compact.  Note also that in view of Lemma~\ref{lemma-reg}
the Ilmanen evolution starting at any convex initial data is smooth and convex and
consequently our uniqueness result shows that uniqueness holds in this class as well.

\begin{theorem} \label{thm-uniqueness} Let $C_0\subset \R^{n+1}$ be a closed convex
set with nonempty interior.  Assume that
$M_t^1=\partial C^1_t, M_t^2=\partial C_t^2 \in \R^{n+1}$, $t \in (0, T]$, are two
smooth convex MCF solutions with the same initial data $M_0=\partial C_0$ in the
sense that
\[
\lim_{t \to 0} \partial C_t^1 = \lim_{t \to 0} \partial C_t^2=\partial C_0,
\]
where convergence is locally uniform.  Then,
\begin{equation}
\label{eqn-zero}
C_t^1= C_t^2, \qquad \forall t \in (0, T].
\end{equation}
\end{theorem}

\begin{proof}
According to Proposition \ref{prop-classification} our initial convex set $C_0$ can
be expressed as $C_0 = \R^k \times \hat C_0$, $0 \leq k \leq n$, where
$\hat C_0 \subset \R^{n+1-k}$ is either compact with non-empty interior, or closed
non-compact with non-empty interior and contains no infinite line.

In the first case, the initial set contains the affine subspace $\R^k\times\{p\}$ for
any $p\in \hat C_0$.  If we choose $p\in\interior\hat C_0$ then
$B^{n+1-k}_\epsilon(p)\subset \hat C_0$, and thus $B^{n+1}_\epsilon(q) \subset C_0$
for any $q\in \R^k\times \{p\}$.  By the maximum principle it follows that if
$\partial C_t$ evolves smoothly by MCF, then
$B^{n+1}_{\sqrt{\epsilon^2-2nt}}(q) \subset C_t$, for $t<\epsilon^2/(2n)$.  Thus the
convex set $C_t$ contains the affine subspace $\R^k\times\{p\}$ for
$t<\epsilon^2/(2n)$ and thus $C_t$ must be a product $C_t = \R^k\times\hat C_t$ for
some smooth compact convex $\hat C_t\subset\R^{n+1-k}$ whose boundary evolves by MCF.
Since the smooth evolution of MCF starting from $\hat C_0$ is unique it follows that
in this case $C_t^1=C_t^2$ for $0\leq t<\epsilon^2/(2n)$.  Repeated application of
this argument then shows that $C_t^1=C_t^2$ for all $t\geq 0$ at which either
solution is defined.

In the second case where $\hat M_0 =\partial \hat C_0$ for some
$\hat C_0 \subset \R^{n+1-k}$ which is closed non-compact with non-empty interior and
contains no infinity line, one can easily deduce that for both solutions can be
expressed as $M^1_t=\partial C^1_t$ and $M_t^2= \partial C_t^2$ where
$C_t^1, C_t^2 \subset \R^{n+1}$ are closed convex non-compact and contain no infinite
line.  If $M_0$ can be written as an entire graph over some hyperplane, one can
choose this hyperplane so that $M_0$ is a proper entire graph.  Then one can easily
deduce that the same holds for any solution, that is both $M^1_t$, $M_t^2$, are
proper entire graphs and the proof follows from the results in \cite{DS}.

Hence, we will concentrate here in the case where $M_0=\partial C_0$ is not an entire
graph.  We will use induction in the dimension $n$ of our solutions and we will see
that in this case as well the proof follows from the ideas in \cite{DS}.  For the
reader's convenience we will provide here a detailed proof despite the fact that some
of the arguments are similar to those in \cite{DS}.

\smallskip When $n=1$, that is, we are in the case of convex complete non-compact
solutions of curve shortening flow, uniqueness was shown by K.S.  Chou and X.P.  Zhu
in \cite{CZ}.  Assuming that the theorem holds for complete, non-compact, convex MCF
solutions of dimension $d \leq n-1$, we will now show that the same also holds in
dimension $d=n$.

\smallskip

To this end, we recall that from our discussion in section \ref{ss-shadow} the
solutions $M_t^i$ can be written as graphs $u_i: D^i_t \to \R$ over their shadows
$D_t^i$.  Furthermore, by Lemma \ref{lemma-shadow-moves-by-mcf} we have that both
$\partial D_t^i$ are MCF evolutions starting at $\partial D_0$ and hence by our
induction assumption, we have that $\partial D_t^1=\partial D_t^2= \partial D_t$.  By
convexity, we also have $D_t^1=D_t^2=D_t$.

\smallskip

We will now apply the arguments in \cite{DS} (Theorem 1.3 and Theorem 1.4) to
conclude that
\[u_1(x,t) = u_2(x,t) \qquad \mbox{for all } \,\, (x,t) \in \cup_{t \in (0,T)} D_t
\times \{t\}.\] This readily implies that $M^1_t = M^2_t$, $t \in (0, T)$, finishing
the proof of the Theorem.  To show that $u_1 = u_2$, it is sufficient to prove that
$u_1 \leq u_2$, since the same argument will also imply that $u_2 \leq u_1$, thus
showing that $u_1=u_2$.

\smallskip {\em We may assume without loss of generality that $u_0 \geq 0$ which also
  implies that $u_1, u_2 \geq 0$.  }  For a small number $\e >0$, define
\[u(x,t) =u_1(x,t)+1 \qquad \mbox{and} \qquad \bu(x,t) =u_2(x,t+\e)+1\] Then $u, \bu$
satisfy equations
\begin{equation}
\label{eqn-ubu}
u_t = \left(\delta^{ij}-\frac{D_i u D_j u}{1+|Du|^2}\right) D_{ij}u, \qquad \bu_t = \left(\delta^{ij}-\frac{D_i \bu D_j \bu}{1+|D\bu|^2}\right) D_{ij}\bu \end{equation}
and furthermore $u, \bu \geq 1$.  To simplify the notation below we set
\[
a_{ij} = \delta^{ij}-\frac{D_i u D_j u}{1+|Du|^2}, \qquad \ba_{ij} =
\delta^{ij}-\frac{D_i \bu D_j \bu}{1+|D\bu|^2}.
\]

\smallskip

Define
\[
w:= u - \bar u.
\]
For any fixed $t \in (0, T-\e]$, $u(\cdot, t)$ is defined for $x \in D_t$ and
$\bar u(\cdot, t) $ is defined for $x \in D_{t+\e}$.  By convexity we have
$D_{t+\e} \subset D_t.$ This means that $w(x,t)$ is defined for
$(x, t) \in \cup_{ t \in (0, T-\e]} D_{t+\e} \times \{t \}$.  Subtracting equations
\eqref{eqn-ubu}, we find that the function $w$ satisfies the equation
\begin{equation}
\label{eqn-w}
w_t - a_{ij} D_{ij} w = ( a_{ij} - \bar a_{ij} )\, D_{ij} \bu 
\end{equation}
on $\cup_{ t \in (0, T-\e]} D_{t+\e} \times \{t \}$.

\smallskip

Following \cite{DS} we compare $w$ from above with
\[
\zeta(x,t) := \e \, (t+\e)\, u^2(x,t).
\]
First, we use $u_t - a_{ij} D_{ij} u =0$ and find that $\zeta$ satisfies
\[
\zeta_t - a_{ij} D_{ij} \zeta = - 2 \e \, (t+\e) \, a_{ij} D_j u D_i u + \e u^2,
\]
where by direct calculation
\[
\begin{split}
a_{ij} D_j u D_i u = \frac{|Du|^2}{1+|Du|^2}.
\end{split}
\]
Combining the above gives
\[
\zeta_t - a_{ij} D_{ij} \zeta = - 2 \e \, (t+\e) \, \frac{|Du|^2}{1+|Du|^2} + \e u^2
\geq \e \, ( u^2 - 2\, (t + \e)).
\]
Since $u \geq 1$, we conclude that for $t \leq 1/8$ and $\e < 1/10$, we have
\begin{equation}
\label{eqn-zeta}
\zeta_t - a_{ij} D_{ij} \zeta > \frac \e2 u^2. 
\end{equation}

\smallskip

Set next
\[
W := w - \zeta=u-\bu - \e \, (t +\e) \, u^2
\]
and
\[
T^*= \min \Big ( T, \frac 18\Big )-\e.
\]
By \eqref{eqn-w} and \eqref{eqn-zeta} we find that $W$ satisfies
\begin{equation}
\label{eqn-W}
W_t - a_{ij} D_{ij} W < ( a_{ij} - \ba_{ij} )\, D_{ij} \bu - \frac \e 2 u^2 
\end{equation}
on the set $\cD_{T^*}:= \cup_{t \in (0,T^*]} D_{t+\e} \times \{t\}$.  Furthermore,
for any $K \subset \subset D_\e$ compact we have $\lim_{t \to 0} u(x,t) = u(x,0)$,
$\lim_{t \to 0} \bu(x,t) = \bu(x,0)$ uniformly on $K$, which implies
\begin{equation}
\label{eqn-W10} \lim_{t \to 0} W(x,t) = u_0(x) - u_2(x,\epsilon) - \e^2 u^2(x,0) \leq - \e^2<0 
\end{equation}
uniformly on $K$.  Here we used that $u(x,0)=u_0(x)+1$,
$\bu(x,0) = u_2(x,\e) + 1 \geq u_0(x) + 1$ and $u(x,0) \geq 1$ since $u_0 \geq 0$.

\smallskip

We will use \eqref{eqn-W} -\eqref{eqn-W10} and the maximum principle to conclude that
\[
W \leq 0 \qquad \mbox{on}\,\, \cD_{T^*}:= \cup_{t \in (0,T^*]} D_{t+\e} \times \{t\}.
\]
To this end, observe first that $u, \bu \geq 1$ implies that for every fixed $\e >0$
and for all $t \in (0,T)$,
\begin{equation}
\label{eqn-WWW}
m^*:=\sup_{(x,t) \in \cD_{T^*} } W(x,t) \leq \e^{-2}.  \end{equation}
Indeed, notice that if there is a point $(x,t) \in \R^n \times (0,T^*]$ where $W(x,t) \geq 0$, then since $\bu \geq 1$, at such point we have $u \geq \bu + \e (t + \e) \, u^2 \geq \e^2 \, u^2$, that is $u(x,t) \leq \e^{-2}$.  Hence, we have 
\begin{equation}
\label{eqn-Wpos}
W(x,t) > 0 \implies W(x,t) \leq u(x,t) \leq \e^{-2} \end{equation}
and therefore the same holds for the supremum $m^*$.

\smallskip

\begin{claim}
\label{claim-1}
If $\e$ is sufficiently small then
\[
m^*:=\sup_{(x,t)\in \cD_{T^*}} W(x,t) \leq 0.
\]

\end{claim}

Once this claim is shown, the theorem will follow by simply letting $\e \to 0$ to
show that $u \leq \bu$ and then switching the roles of $u$ and $\bu$.

\begin{proof}[Proof of Claim \ref{claim-1}]

To prove the claim, we assume by contradiction, that
\[
m^* >0.
\]
We treat separately the cases where: (i) the supremum is attained at an interior
point $(x_0, t_0) \in \cD_{T^*}$, (ii) the supremum is attained as $|x| \to +\infty$,
(iii) the supremum is attained as the limit $m^*=\lim_{k \to \infty} W(x_k, t_k)$,
for a sequence of points $(x_k, t_k) $ such that $t_k \to t_0 \in (0, T^*]$ and
$x_k \to x_0 \in \partial D_{t_0+\e}$, and (iv) the supremum is attained as the limit
$m^*=\lim_{k \to \infty} W(x_k, t_k)$, for a sequence of points $(x_k, t_k) $ such
that $t_k \searrow 0$ and $x_k \to x_0 \in \bar D_{\e}$.

\smallskip

\noindent{\em Case 1.} We have $m^* = W(x_0, t_0)$ for some point $t_0 \in (0, T^*]$
and $x_0 \in D_{t_0+\e}$.  Then at such point
\begin{equation}
\label{eqn-uuu}
\big (1 - \e (t_0+\e) u \big ) \, u = \bu + m^* \qquad \mbox{and} \qquad \big (1-2\e (t_0+\e) u \big ) D_i u = D_i \bu \end{equation}
Note that the first equality, $m^* >0$ and $ \bu \geq 1$ imply that 
\begin{equation}
\label{eqn-u12}
1- \e (t_0+\e) u >0 \end{equation}
at the maximum point, which will be used below.  We will now use the second equality in \eqref{eqn-uuu} to evaluate the right hand side of \eqref{eqn-W} at the maximum point.  First, we have
\begin{equation*}
\begin{split}
a_{ij} - \ba_{ij} &= \frac{D_i \bu D_j \bu}{1+|D\bu|^2} - \frac{D_i u D_j
  u}{1+|Du|^2} = \frac{D_i \bu D_j \bu}{1+|D\bu|^2} - (1- 2\e (t_0+\e) u)^{-2}
\, \frac{D_i \bu D_j \bu}{1+|Du|^2}  \\
&= \frac {(1- 2\e \, (t_0+\e)\,u )^2\, (1+|Du|^2)-  (1+ |D\bu|^2) } {(1- 2\e (t_0+\e)\,u )^2} \cdot   \frac{D_i \bu D_j \bu}{(1+|Du|^2)(1+|D\bu|^2)} \\
&=- \frac{ 4\e (t_0+\e) \,u(1- \e (t_0+\e)\, u)}{(1- 2\e (t_0+\e) \, u)^2} \cdot
\frac{D_i \bu D_j \bu}{(1+|Du|^2)(1+|D\bu|^2)}.
\end{split}
\end{equation*}
To derive the last equality we used $(1- 2\e (t_0+\e)\,u )^2\,|Du|^2 = |D\bu|^2$
which gave us
\[(1- 2\e (t_0+\e) u )^2\, (1+|Du|^2)- (1+ |D\bu|^2) = (1- 2\e (t_0+\e) u )^2 -1 = -
4\e (t_0+\e) u \, (1- \e (t_0+\e) u).\]

Combining the above with \eqref{eqn-W} we find that at the point
$(x_{\max}(t_0),t_0)$ we have
\[
0 \leq W_t - a_{ij} D_{ij} W < - \frac{ 4\e (t_0+\e) u\, (1- \e (t_0+\e)\, u)}{(1-2\e
  (t_0+\e)\, u)^2} \, \frac{D_{ij} \bu D_i \bu D_j \bu}{(1+|Du|^2)(1+|D\bu|^2)} -
\frac \e 2 u^2.
\]

\smallskip

The convexity of $\bu$ implies that $D_{ij} \bu D_i \bu D_j \bu \geq 0$.  Therefore,
the last differential inequality and \eqref{eqn-u12} yield
\begin{equation}
\label{eqn-W5}
\begin{split}
0 \leq W_t - a_{ij} D_{ij} W &< - \frac \e 2 \, u^2 <0
\end{split}
\end{equation}
leading to a contradiction.  This proves our assertion in this case.

\smallskip

\noindent{\em Case 2.} The supremum $m^*$ is attained at infinity, that is there
exists a sequence of times $t_k \in (0, T^*]$, $t_k \to t_0 \in [0, T^*]$, and a
sequence of points $x_k \in D_{t_k+\e}$, $|x_k| \to \infty$, such that
$W(x_k, t_k) > \frac{m^*}2 >0.$

\smallskip To derive a contradiction, we simply observe that on the one hand
$|x_k| \to +\infty$ and the convexity of our initial data imply that
$u_0(x_k, 0) \to +\infty$.  In addition by the convexity of $u(\cdot, t)$ we have
$u(x_k, t_k) \geq u_0(x_k, 0)$, hence $u(x_k, t_k) \to +\infty$.  On the other hand,
\eqref{eqn-Wpos} shows that $W(x_k, t_k)>0$ implies $u (x_k, t_k) < \e^{-2}$,
contradicting $u(x_k, t_k) \to +\infty$.  This proves our assertion in this case as
well.

\smallskip
\noindent{\em Case 3.} The supremum $m^*$ is attained at a point $(x_0, t_0) $,
$x_0 \in \partial D_{t_0+\e}$, $t_0 \in (0, T^*]$, that is, there is a sequence of
times $t_k \in (0, T^*]$, $t_k \to t_0 >0$ and a sequence of points
$x_k \in D_{t_k+\e}$, $x_k \to x_0 \in \bar D_{\e+t_k}$, such that
$W(x_k, t_k) \to m^* >0$.  We may choose the sequence $(x_k, t_k)$ such that
$W(x_k, t_k) := \max_{(x,t) \in \cD_{t_k}} W(x,t)$, where
$\cD_{t_k} := \cup_{t \in (0, t_k]} D_{t+\e} \times \{ t\}$.  Then, the same argument
as in case 1, leads to a contradiction.

\smallskip

\noindent{\em Case 4.}  The supremum $m^*$ is attained near $t_0=0$, namely there is
a sequence of times $t_k \in (0, T^*]$, $t_k \searrow 0$, and a sequence of points
$x_k \in D_{t_k+\e}$, $x_k \to x_0 \in \bar D_{\e}$, such that
$W(x_k, t_k) > \frac{m^*}2 >0.$ First note that $x_0$ cannot belong in $D_\e$ since
$\lim_{t\to 0} W(\cdot, t) <0$, on any compact set $K \subset \subset D_\e$.  Hence
$x_0 \in \partial D_{\e}$.  Next note again that by \eqref{eqn-Wpos},
$W(x_k, t_k) >0$ implies that $u (x_k, t_k) < \e^{-2}$.  On the other hand,
$W(x_k, t_k) >0$ implies that
$u(x_k, t_k) > \bu (x_k, t_k) \geq \bu (x_k, 0)=u_2 (x_k, \e) $ and
$u_2 (x_k, \e) \to +\infty$, since $x_k \to x_0 \in \partial D_\e$.  This
contradicts, $u (x_k, t_k) < \e^{-2}$, showing that this case is not possible.

\smallskip

Combining the four cases above shows that $m^* \leq 0$, finishing the proof of the
claim.
\end{proof}

We have just seen that $W:= u - \bu - \e ( t+\e) u^2 \leq 0$ on
$\R^n \times (0, T^*]$, where $u=u_1$ and $\bu(\cdot, t) = u_2(\cdot, t+\e)$.  Let
$\e \to 0$ to obtain that $u_1 \leq u_2$ on $\cD_{T^*}$.  Similarly, $\bu \leq u$ on
$\cD_{T^*}$ which means that $u=\bu$ on $\cD_{T^*}$.  By repeating the same proof
starting at $t=T^*$ we conclude after finite many steps that $u \equiv \bu$ on
$\cup_{t \in (0, T]} D_t \times \{ t \}$, finishing the proof of Proposition
\ref{thm-uniqueness}.
\end{proof}

\subsection{Proof of Theorem \ref{thm-mcf-exist-unique}} 
The existence of a smooth solution $C_t$, $t \in (0,T)$, starting at $C_0$ follows
from Theorem \ref{prop-existence}.  The uniqueness of any smooth convex solution
starting at $C_0$ follows from Theorem \ref{thm-uniqueness}.  The proof of Theorem
\ref{thm-mcf-exist-unique} is a direct consequence of these two results.

\begin{remark} While the convexity of smooth compact solutions is known to be
preserved by MCF $($see in \cite{Hu}$)$, in the non-compact case the convexity of
smooth solutions is not yet known to be preserved.

\end{remark}

\section{The space of convex sets in $\R^{n+1}$ and its topology}
\label{sec-convex-topology}

\subsection{The set of convex subsets}\label{sec-X-barX-introduced}
Let
\[
\bar X = \{C \subset \R^{n+1} \mid C\mbox{ is closed and convex}\}.
\]
The space $\bar X$ also contains convex sets $C$ whose interior is empty.  Such sets
are contained in a hyperplane.  If they do not occupy the entire hyperplane their
flow by Mean Curvature is not defined, and thus we will exclude convex sets with
empty interior.  We define
\[
X = \{C\in \bar X \mid C\text{ has nonempty interior}\}.
\]

\subsection{The Huisken measure}
\label{sec:Huisken-measure}
For each convex subset $C\in X$ we define the \emph{Huisken measure}
\[
\mu_C = \bigl(4\pi\bigr)^{-n/2} e^{-\|x\|^2/4} H^n_C,
\]
where $H^n_C$ is $n$-dimensional Hausdorff-measure on $\partial C$.  This measure is
just the integrand in the definition of Huisken's energy:
\begin{equation}\label{eqn-HE}
\mc H(C) = \frac{1}{(4\pi)^{n/2}}\int_{\partial C} e^{-\|x\|^2/4} dH^n_C = \mu_C(\R^{n+1}).
\end{equation}
The normalizing factor $(4\pi)^{n/2}$ is chosen so that
\[
(4\pi)^{-n/2} \int_Pe^{-\|x\|^2/4} dH^n = 1
\]
for any hyperplane $P\subset\R^{n+1}$ containing the origin.  This implies
\begin{equation}
\label{eq-Huisken-of-cylinders}
\mc H(C\times\R^k) = \mc H(C)
\end{equation}
for any convex set $C\subset\R^{n+1-k}$.

If $C$ has empty interior, so that it is a convex subset of some hyperplane of
$\R^{n+1}$, then we change the above definition by doubling the measure:
\[
\mu_C = 2\,\bigl(4\pi\bigr)^{-n/2} e^{-\|x\|^2/4} \;H^n_C.
\]
In either case (empty or nonempty interior) the measure $\mu_C$ is given by
\[
\forall f\in C^0_c(\R^{n+1}):\quad \langle\mu_C , f\rangle =
\lim_{\varepsilon\searrow 0}\frac1\varepsilon \int_{B_\varepsilon(C)\setminus
  C}e^{-\|x\|^2/4}f(x) \, dx,
\]
where $dx$ is Lebesgue measure on $\R^{n+1}$ and $B_\varepsilon(C)$ is the
$\varepsilon$~neighborhood of $C$.

\medskip

The following lemma shows that most of the Huisken measure of any convex set is
located in a large enough ball $B_R\subset\R^{n+1}$.

\begin{lemma}\label{lemma-Huisken-measures-tight}
For every $n$ there is a $c_n\in\R$ such that for any closed convex set
$C\subset\R^{n+1}$ and $R\geq 0$ one has
\[
\int_{\partial C\setminus B_R(0)} e^{-\|x\|^2/4}dH^n \leq c_n R^n e^{-R^2/4}.
\]
\end{lemma}
\begin{proof}
For $i=1, \dots, n+1$ consider the sets
$E_i^\pm \subset\partial C \setminus B_R^{n+1}$ on which the unit normal $N$ to $C$
satisfies $\pm \langle N, e_i\rangle \geq \tfrac 12$.  The sets $E_i^\pm$ cover
$\partial C \setminus B_R^{n+1}$.

We estimate the contribution of $E_{n+1}^+$ to the Huisken integral.  Let
$\pi:\R^{n+1}\to\R^n$ be the orthogonal projection along the $x_{n+1}$ axis.  Then
the set $E_{n+1}^+$ is the graph of a function $x_{n+1}= h(x_1, \dots, x_n)$ which is
defined on $\Omega = \pi(E_{n+1}^+)$.  The unit normal on $E_{n+1}^+$ is
$N=(-\nabla h(x), 1)/\sqrt{1+|\nabla h|^2}$, so
$\langle N, e_{n+1}\rangle \geq \frac 12$ implies that $|\nabla h|\leq 1$ everywhere
on $\Omega$.  Using $\|x\|^2+h(x)^2\geq R^2$ everywhere on $E_{n+1}^+$ we find
\begin{align*}
  \int_{E_{n+1}^+}e^{-\|x\|^2/4}dH^n
  &= \int_\Omega e^{-(\|x\|^2+h(x)^2)/4} \sqrt{1+|\nabla h|^2}\,dx\\
  &\lesssim  \int_{\Omega\cap B_R^n} e^{-(\|x\|^2+h(x)^2)/4}\,dx
    +  \int_{\Omega\setminus B_R^n} e^{-(\|x\|^2+h(x)^2)/4}\,dx\\
  &\lesssim  e^{-R^2/4}|B_R^n| +  \int_{\R^n\setminus B_R^n}e^{-\|x\|^2/4}dx\\
  &\lesssim R^ne^{-R^2/4}.
\end{align*}
The same argument applies to all regions $E_i^\pm$ so the lemma follows.
\end{proof}

\subsection{The Gaussian mass}

For any $C\in\bar X$ we set
\begin{equation}
\label{eq:gaussian-vol}
\gvol(C) = \int_{C} e^{-\|x\|^2/4} \frac{dx}{(4\pi)^{n/2}}
\end{equation}
where $dx$ is again Lebesgue measure on $\R^{n+1}$.

A closed convex set $C$ has nonempty interior if and only if $\gvol(C)>0$.

\subsection{A topology on $\bar X$}
For every closed convex set $C\subset \R^{n+1}$ we consider the distance function
from a convex set $C$, $d_C(x)$, defined as in \eqref{eq-dist-fun}.

\begin{prop}\label{prop-dC-convex-and-Lipschitz}
The distance function $d_C$ is convex and Lipschitz continuous with Lipschitz
constant~$1$.
\end{prop}
\begin{proof}
For each $y\in C$ the function $x\mapsto \|x-y\|$ has these properties.  The
proposition then follows from the definition $d_C(x) = \min_{y\in C}\|x-y\|$.
\end{proof}

The Hausdorff distance between closed subsets of a metric space is given by
\[
d_H(C_1, C_2) = \sup_{x}|d_{C_1}(x)-d_{C_2}(x)|,
\]
and a sequence of sets $C_k$ converges in the Hausdorff distance if the distance
functions $d_{C_k}$ converge uniformly.  This notion of convergence is too strong for
our purposes since the Hausdorff distance between a bounded and an unbounded set is
infinite.  Instead we consider the metric
\[
\rho(C,C') := \sup_{x\in\R^{n+1}} \frac{|d_C(x)-d_{C'}(x)|}{1+\|x\|^2}
\]
and the topology it defines on $\bar X$.  We will use the following notation:
\[
C_k\rto C \stackrel{\rm def}\iff \rho(C_k, C)\to 0.
\]

The following observation will be used in the next section.

\begin{prop}\label{prop-limit-of-dist-funcs}
If $C_k$ is a sequence of closed convex sets for which
$f(x)=\lim_{k\to\infty}d_{C_k}(x)$ exists pointwise, then $d_{C_k}$ converges
uniformly on all bounded sets and there is a closed convex set $C\subset\R^{n+1}$
such that $f=d_C$, and such that $C_k\rto C$.
\end{prop}
\begin{proof}
By Proposition~\ref{prop-dC-convex-and-Lipschitz} each $d_{C_k}$ is convex and
Lipschitz with constant~$1$.  Therefore the limit function $f$ is also convex and
Lipschitz continuous with constant~$1$.  Its zero set $C:=f^{-1}(0)$ is a closed
convex set.  Lipschitz continuity of $f$ with constant~$1$ implies that
$f(x)\leq d_C(x)$ for all $x\in \R^{n+1}$.

To prove the reverse inequality, we fix $x_0 \in\R^{n+1}$ and, for each $k\in\N$, let
$y_k\in C_k$ be the nearest point to $x_0$.  After passing to a subsequence we may
assume that $y_k\to y_0$ for some $y_0\in\R^{n+1}$.  Since $d_{C_k}\to f$ locally
uniformly, we have $f(y_0) = \lim_{k \to \infty} d_{C_k}(y_k)$=0.  Hence
\[
d_C(x_0) \leq \|x_0 -y_0\|\leq \liminf_{k\to\infty} \|x_0-y_k\|\leq
\liminf_{k\to\infty} d_{C_k}(x_0)=f(x_0).
\]
It follows that $f=d_C$.  Since we have shown that $d_{C_k}$ converges locally
uniformly to $d_C$ it follows that $C_k\rto C$.
\end{proof}

\subsection{Properties of the topology on $\bar X$ and $X$}

We will now gather various compactness properties of the topology on $\bar X$ and the
subset $X$ we have just defined in the previous section.  These properties will be
used in the following sections where we will study ancient convex solutions.
\begin{prop}\label{prop-convergence-of-boundary}
Let $C_k\in X$ be a sequence with $C_k\rto C$.

If $K\subset \interior(C)$ is compact, then $K\subset \interior(C_k)$ for
sufficiently large $k$.

If $\bar K$ is a compact set with $\bar K\cap C=\varnothing$, then
$\bar K\cap C_k=\varnothing$ for large enough $k$.
\end{prop}
\begin{proof}
If $K\not\subset C_k$ for infinitely many $k\in\N$, then there is a sequence
$k_i\in\N$ and corresponding points $p_{k_i}\in K\setminus C_{k_i}$.  Since $C_{k_i}$
is closed and convex, we can separate $p_{k_i}$ and $C_{k_i}$ with a hyperplane,
i.e.~there exist unit vectors $a_{k_i}\in\R^{n+1}$ with
$\langle a_{k_i}, p_{k_i}-q \rangle >0$ for all $q\in C_{k_i}$.

It follows from $K\subset\mathrm{int}(C)$ that there is a $\delta>0$ so that
$B_\delta(p)\subset C$ for all $p\in K$.  In particular, $B_\delta(p_{k_i})\subset C$
and $r_{k_i}:=p_{k_i}+\delta a_{k_i}\in C$ for all $i$.  On the other hand
$\langle a_{k_i}, p_{k_i}-q\rangle>0$ implies
\[
\|r_{k_i}-q\|\geq \langle a_{k_i}, r_{k_i}-q\rangle = \langle a_{k_i}, p_{k_i}+\delta
a_{k_i}-q\rangle =\langle a_{k_i}, p_{k_i}-q\rangle + \delta > \delta.
\]
This holds for all $q\in C_{k_i}$ so $d_{C_{k_i}}(r_{k_i})\geq \delta$ for all $i$.
In view of $r_{k_i}\in C$ we have $d_C(r_{k_i})=0$ and
$d_{C_{k_i}}(r_{k_i}) -d_{C}(r_{k_i})>\delta$ for all $i$, contradicting the
assumption that $d_{C_{k_i}}\to d_C$ locally uniformly.  This shows that
$K\subset C_k$ for large enough $k$.  To prove that $K\subset\interior(C_k)$ for
large $k$ we note that for some small $\epsilon>0$ the closed $\epsilon$-neighborhood
$K_\epsilon=\{x\mid d_K(x)\leq \epsilon\}$ is contained in $C$, and then apply the
previous arguments to $K_\epsilon$ instead of $K$.

To prove the statement about $\bar K$ we note that since $\bar K$ is compact, it
follows from $C_k\rto C$ that $d_{C_k}$ converges uniformly to $d_C$ on $\bar K$.
Since $\min_{\bar K}d_C >0$, this implies that $\min_{\bar K}d_{C_k}>0$ for large
enough $k$.  Hence $\bar K\cap C_k=\varnothing$ for large enough $k$.
\end{proof}

\begin{lemma}\label{lemma-rho-implies-weak-convergence}
Let $C_k\in X$ be a sequence with $C_k\rto C$ and $\gvol(C)>0$.  Then the surface
measures $H^n|{\partial C_k}$ converge weakly, locally, to $H^n|_{\partial C}$, i.e.
\[
\lim_{k\to\infty} \int \varphi(x) dH^n|_{\partial C_k} = \int \varphi(x)
dH^n|_{\partial C}
\]
for all compactly supported continuous $\varphi:\R^{n+1}\to\R$.
\end{lemma}

\begin{proof}
If $\gvol(C)>0$ then $C$ has nonempty interior.  Let $p\in \partial C$ be given.  By
assumption $\interior(C)\neq\varnothing$, so we can rotate and translate our
coordinates to place $p$ at the origin and so that in some small neighborhood
$\mc N= B_r^n(0)\times[-2a, 2a]$ of $p$ the convex set $C$ is given by
\[
C\cap \mc N = \{(y, z)\in\R^n\times\R \mid \|y\|\leq r, -2a\leq z \leq h(y)\}
\]
for some Lipschitz continuous function $h:B_r^n(0)\to[-a, a]$.  Since $|h(x)|\leq a$
the boundary $\partial C\cap\mc N$ never intersects the top and bottom
$B_r^n(0)\times\{\pm 2a\}$ of $\mc N$.

For any $\varepsilon>0$ the set
$K=\{(y, z)\mid \|y\|\leq r, -2a\leq z\leq h(y)-\varepsilon\}$ is a compact subset of
$\interior(C)$.  Therefore $K\subset \interior(C_k)$ for large enough $k$.

Similarly, the set $\bar K=\{(y,z)\mid \|y\|\leq r, h(y)+\varepsilon \leq z\leq 2a\}$
is disjoint from $C$, so for large enough $k$ it also does not intersect $C_k$.

For sufficiently large $k$ it follows that $C_k\cap \mc N$ is also given by
\[
C_k\cap \mc N = \{(y, z)\in\R^n\times\R \mid \|y\|\leq r, -2a\leq z \leq h_k(y)\}
\]
for some function $h_k:B_r^n(0)\to [-2a, 2a]$ that satisfies
$|h_k(y)-h(y)|\leq \varepsilon$.

We have shown that $\partial C_k\cap \mc N$ and $\partial C\cap \mc N$ are graphs of
concave functions $h_k, h:B_r^n(0)\to[-2a,2a]$, and that $h_k$ converges uniformly to
$h$ as $k\to\infty$.  Using concavity of $h_k$ this implies that $\nabla h_k$
converges in measure to $\nabla h$, and therefore that the surface area measures
$H^n|_{\partial C_k\cap\mc N}$ converge weakly to $H^n|_{\partial C\cap \mc N}$.
\end{proof}

\begin{lemma}\label{lemma-huisken-cont}
The Huisken functional $\hu:X\to\R$ is continuous.  More precisely, if $C_k\in X$ is
a sequence with $C_k\rto C$ for $C\in X$, then $\hu(C_k)\to\hu(C)$.  Moreover, the
Huisken measures $\mu_{C_k}$ converge in the weak$^*$ topology to $\mu_C$, in the
sense that $\langle\mu_{C_k}, \varphi\rangle \to \langle\mu_C, \varphi\rangle$ for
every bounded and continuous $\varphi:\R^{n+1}\to\R$.
\end{lemma}
\begin{proof}
The previous Lemma implies that for any compactly supported continuous
$\varphi:\R^{n+1}\to\R$ one has
$\langle \mu_{C_k}, \varphi\rangle \to \langle\mu_C, \varphi\rangle$.  We will obtain
the same conclusion for any bounded continuous $\varphi:\R^{n+1}\to\R$.  The Lemma
then follows by choosing $\varphi(x) = (4\pi)^{-n/2}e^{-\|x\|^2/4}$.

Choose a nondecreasing continuous cut-off function $\sigma:\R\to\R$ with
$\sigma(r)=0$ for $r\leq \frac{1}{2}$, $\sigma(r)=1$ for $r\geq 1$ and consider
$\psi_R(x) = \sigma(\|x\|/R)\varphi(x)$ and $\varphi_R(x) = \varphi(x) - \psi_R(x)$.

Then $\varphi_R$ is compactly supported, so
$\langle\mu_{C_k}, \varphi_R\rangle \to \langle\mu_C, \varphi_R\rangle$.
Furthermore, $\psi_R$ is supported outside $B_R$, so
Lemma~\ref{lemma-Huisken-measures-tight} implies that
$|\langle\mu_{C_k}, \psi_R \rangle| \leq \|\varphi\|_\infty c_nR^ne^{-R^2/4}$.

For any given $\epsilon>0$ we choose $R$ so large that
$\|\varphi\|_\infty c_nR^ne^{-R^2/4} < \epsilon$.  Then we have shown
\begin{align*}
  \limsup_{k\to\infty}\langle\mu_{C_k}, \varphi\rangle
  &=\limsup_{k\to\infty}\Big(\langle\mu_{C_k}, \varphi_R\rangle + \langle\mu_{C_k}, \psi_R\rangle\Big) \\
  &\leq \langle\mu_C, \varphi_R\rangle + \epsilon\\
  &=\langle\mu_C, \varphi\rangle - \langle\mu_C, \psi_R\rangle + \epsilon\\
  &\leq \langle\mu_C, \varphi\rangle + 2\epsilon.
\end{align*}
This holds for every $\epsilon>0$, so
$\limsup_{k\to\infty}\langle\mu_{C_k}, \varphi\rangle \leq \langle\mu_C,
\varphi\rangle$.

By applying the same argument to $-\varphi$ instead of $\varphi$ we also find
$\liminf_{k\to\infty}\langle \mu_{C_k}, \varphi\rangle \geq \langle\mu_C,
\varphi\rangle$, and hence
$\langle\mu_{C_k}, \varphi\rangle \to \langle\mu_C, \varphi\rangle$, as claimed.
\end{proof}

\begin{lemma}[First Compactness Lemma]
\label{lemma-compactness-1}

For each $R>0$ the set
\[
\bar X_R := \{C\in\bar X \mid C\cap B_R(0)\neq\varnothing\}
\]
is compact.
\end{lemma}
\begin{proof}
Let $C_k$ be a sequence of closed convex sets for which there exist
$p_k\in C_k\cap B_R(0)$.  Then
\[
d_{C_k}(x)\leq d_{C_k}(p_k) + \|x-p_k\|\leq \|x\|+R
\]
for all $x\in\R^{n+1}$ and $k\in\N$.

The distance functions $d_{C_k}$ are therefore uniformly bounded on compact sets.
They are also uniformly Lipschitz continuous, so by Ascoli's theorem we may assume
after passing to a subsequence that $d_{C_k}$ converges uniformly on compact sets.
By Proposition~\ref{prop-limit-of-dist-funcs} there is a $C\in \bar X$ with
$C_k\rto C$.

If $C\cap \bar B_R(0)$ were empty, then
Proposition~\ref{prop-convergence-of-boundary} implies
$C_k\cap \bar B_R(0)=\varnothing$ for large $k$, a contradiction with
$C_k\in \bar X_R$ for all $k$.
\end{proof}

Recall the definition \eqref{eq:gaussian-vol} of the enclosed Gaussian volume
$\gvol(C)$ of any convex set $C \in \bar X$.
\begin{lemma}\label{lemma-gvol-continuous}
The enclosed Gaussian volume $\gvol:\bar X\to\R$ is continuous, i.e.~if $C_k\to C$
then $\gvol(C_k)\to \gvol(C)$.
\end{lemma}
\begin{proof}
Let $\epsilon>0$ be given and consider the open neighborhood
$C^\epsilon=\bigcup_{x\in C}B_\epsilon(x)$ of $C$.  There is a $k_\epsilon\in\N$ such
that $|d_{C_k}(x)-d_C(x)|\leq \epsilon/2$ for all $x\in B_R(0)$ and
$k\geq k_\epsilon$.  This implies that $d_C(x)\leq \epsilon/2$ for all
$x\in C_k\cap B_R(0)$, i.e.~$C_k\cap B_R(0)\subset C^\epsilon$.  Therefore, if
$k\geq k_\epsilon$, then
\[
\gvol(C_k)\leq \gvol(C_k\cap B_R)+\gvol(\R^{n+1}\setminus B_R) \leq \gvol(C^\epsilon)
+ \gvol(\R^{n+1}\setminus B_R)
\]
holds for all $R>0$.  Letting $R\to\infty$ we see that
$\gvol(C_k)\leq \gvol(C^\epsilon)$ for all $k\geq k_\epsilon$.  Hence
$\limsup_{k\to\infty} \gvol(C_k)\leq \gvol(C^\epsilon)$ for all $\epsilon>0$.
Monotone convergence implies $\gvol(C^\epsilon)\to\gvol(C)$ as $\epsilon\to0$, so we
have shown $\limsup_{k\to\infty}\gvol(C_k) \leq \gvol(C)$.

We still have to show that $\liminf_{k\to\infty}\gvol(C_k)\geq \gvol(C)$.  If
$\gvol(C)=0$ then there is nothing to prove, so assume $\gvol(C)>0$.

Let $\epsilon>0$ be given again.  Then, since $\gvol(\partial C)=0$, there is a
compact $K\subset \interior (C)$ with $\gvol(K)>\gvol(C)-\epsilon$.  By
Proposition~\ref{prop-convergence-of-boundary} there is a $k_\epsilon\in\N$ such that
$K\subset C_k$ for all $k\geq k_\epsilon$.  Hence
$\gvol(C_k)\geq \gvol(K) \geq \gvol(C)-\epsilon$ for all $k\geq k_\epsilon$.  This
implies $\liminf_{k\to\infty}\gvol(C_k)\geq \gvol(C)-\epsilon$.  Since we have shown
this for any $\epsilon>0$ we get $\liminf_{k\to\infty}\gvol(C_k)\geq \gvol(C)$.
\end{proof}

\begin{lemma}[Second Compactness Lemma]
\label{lemma-compactness-2}
For any $\gvol_0>0$ the set
$X_{\gvol\geq \gvol_0} = \{C\in X \mid \gvol(C)\geq \gvol_0\}$ is compact.
\end{lemma}
\begin{proof}
Choose $R>0$ so large that $\gvol(\R^{n+1}\setminus B_R) < \gvol_0$.  Then every
$C\in X_{\gvol\geq \gvol_0}$ must intersect $B_R$,
i.e.~$X_{\gvol\geq \gvol_0}\subset \bar X_R := \{C\in \bar X \mid C\cap B_R\neq
\varnothing\}$.  We have shown that $\bar X_R$ is compact, and continuity of $\gvol$
implies that $X_{\gvol\geq \gvol_0}$ is a closed subset of $\bar X_R$.  Therefore
$X_{\gvol\geq \gvol_0}$ is compact.
\end{proof}

\section{The semiflow defined by RMCF}
\label{sec-semiflow}

We will now study some basic properties of the semiflow defined by RMCF.

\subsection{The semiflow}
If $C\in X$, then $C$ is closed, convex and has nonempty interior.  Hence, the
existence and uniqueness Theorem \ref{thm-mcf-exist-unique} provides a unique maximal
solution $\hat C_t$ to MCF with initial condition $\hat C_0=C$ that is defined on a
time interval $0\leq t< \Tx (C)$.

The rescaling
\begin{equation}\label{eq-rescaling}
C_\tau \stackrel{\rm def}= e^{\tau/2} \hat C_{1-e^{-\tau}},  \qquad \tau:=  -\ln(1-t) \geq 0
\end{equation}
turns the solution $\hat C_t$ of MCF into a solution $C_\tau$ of RMCF, whose lifespan
is
\[
\taum(C) =
\begin{cases} +\infty \quad & \text { if } \,\, \Tx(C) \geq 1, \\ -\ln(1-\Tx(C))
\quad & \text{ if } \,\, \Tx(C)<1.
\end{cases}
\]
If we denote this solution by
\[
\phi(C, \tau) = \phi^\tau(C) = C_\tau
\]
then we have defined a map
\[
\phi:\mf D\to X
\]
whose domain is
\begin{equation}
\label{eq-D}
\mf D := \{(C, \tau)\in X\times[0,\infty) \mid C\in X, 0\leq \tau<\taum(C)\}.
\end{equation}
It is easy to see that the semigroup properties hold:
\begin{enumerate}[(1)~]
\item $\phi^0(C)=C$ for all $C\in X$
\item if $0\leq \tau<\taum(C)$ and $0\leq \tau'<\taum\bigl(\phi^\tau(C)\bigr)$ then
$\tau+\tau'<\taum(C)$ and
\[
\phi^{\tau+\tau'}(C) = \phi^{\tau'}\bigl(\phi^\tau(C)\bigr).
\]
\end{enumerate}

\subsection{Continuity of the semiflow}
\label{ss-semiflow-continuous}
Here we collect three lemmas which together guarantee that $\phi^\tau$ is a
continuous local semiflow on $X$.  We begin with a convenient description of what
must happen if a solution to RMCF becomes singular in finite time.

\begin{lemma}\label{lemma-huisken-to-zero}
If $\taum(C) < +\infty$, then $\hu(\phi^\tau(C)) \to 0$ as $\tau\nearrow \taum(C)$.
\end{lemma}

\begin{proof}
Since the set $C$ is fixed here, to simplify the notation we denote $ \taum(C)$
simply by $\taum$.  Assume by contradiction that $\hu(C_\tau)\geq \delta$ for all
$\tau\in(0,\taum)$.  Then there is an $R>0$ such that
$C_\tau\cap B_R(0)\neq \varnothing$ for all $\tau< \taum$.

Consider the unrescaled MCF $\hat C_t$ corresponding to $C_\tau$, i.e.
\[
\hat C_t := \sqrt{1-t}\,C_{-\ln(1-t)}\qquad\text{ where } 0<t <\Tx := 1-e^{-\taum} .
\]
For $t\in (0, \Tx)$, $\partial\hat C_t$ is a smooth convex MCF that intersects
$B_{R\sqrt{1-t}}$.  It follows that
\[
\hat C_{\Tx} := \bigcap_{t<\Tx}\hat C_t
\]
is nonempty.

The RMCF, $C_\tau = e^{\tau/2}\hat C_{1-e^{-\tau}}$ satisfies
\[
C_{\taum} :=\lim_{\tau\nearrow \taum}C_\tau = e^{\Tx/2}\hat C_{\Tx}
\]
If $C_{\taum}$ has nonempty interior, then $\hat C_{\Tx}$ also has nonempty interior
and then by Proposition \ref{prop-interior-empties} the Ilmanen flow $\hat C_t$ could
be extended beyond $t=\Tx$.  This contradicts the assumption that $C_\tau$ does not
extend beyond $\tau=\taum$.  Thus $C_{\taum}$ has no interior, which implies that
$C_{\Tx}$ is contained in a hyperplane $L\subset \R^{n+1}$.  Since
$\hu(C_\tau)\geq \delta$ it follows that $C_{\taum}$ must contain a relatively open
subset of the hyperplane $L$.  Going back again to the unrescaled flow we see that
$\hat C_{\Tx} = e^{-\taum/2} \, C_{\taum}$ also contains a relatively open subset of
the hyperplane $\hat L := e^{-T/2}L$.  Thus $\hat C_t$ shrinks to an open subset of a
hyperplane $\hat L$.  This contradicts the strong maximum principle, so that the
assumption that $\hu(C_\tau)\geq \delta$ for all $\tau$ cannot hold.
\end{proof}

\begin{lemma}\label{lemma-D-open}
The domain $\mf D$ of the semiflow is an open subset of $X\times[0,\infty)$.
\end{lemma}

\begin{proof}
We show that the time of existence $\taum(C)$ of a RMCF with initial data $C\in X$ is
a lower semicontinuous function of $C\in X$, i.e.~for any $C\in X$, any $T<\taum(C)$,
and any sequence $C_m\in X$ with $C_m\rto C$ one has $\taum(C_m)>T$ for large enough
$m$.

To prove this we work with the unrescaled MCF $\{\hat C_t \mid 0\leq t <\Tx(C)\}$
starting from~$C$.  Let $\tau_0$ be given and define $t_0 = 1-e^{-\tau_0}$, so that
$t_0 < \Tx(C) $.  By definition~\ref{def:emcf-solution} there exists a smooth compact
MCF $D_t$ that is defined for all $t\in [0, t_0]$ and for which
$D_0\subset \interior C$.  Since $C_m\rto C$ it follows that
$D_0\subset\interior C_m$ for all large $m$, and by definition the MCF
$\hat C_{m, t}$ starting from $C_m$ contains $D_t$ for all $t\in [0, t_0]$.  In
particular, $\Tx(C_m) > t_0$, which then implies $\taum(C_m) > \tau_0$ for all large
$m$.
\end{proof}

\begin{prop}\label{lemma-semiflow-continuous}
The semiflow $\phi:\mf D\to X$ is continuous.
\end{prop}
\begin{proof}
Let $C_m\in X$ be a sequence with $C_m\rto C_*\in X$, and let
$\tau_0 \in(0, \taum(C_*))$ be given.  By Lemma~\ref{lemma-D-open} the semiflow
$\phi^\tau(C_m)$ is defined for all $\tau \in [0,\tau_0]$ if $m$ is large enough.
Suppose that for some $\bar \tau\in [0, \tau_0]$ the sequence $C_m(\bar \tau)$ does
\emph{not} converge to~ $C_*(\bar \tau)$.  Then we may assume, after passing to a
subsequence, that $\phi^\tau(C_m)\rto D(\tau)$ uniformly for $\tau\in[0,\tau_0]$, and
that $D(\bar\tau)\neq \phi^{\bar\tau}(C_*)$.  The limit $D(\tau)$ is itself a
solution to RMCF, and its initial value is
$D(0)=\lim_{m\to\infty} \phi^0(C_m)=\lim_{m\to\infty} C_m = C_*$.  This leaves us
with two solutions $D(\tau)$ and $\phi^\tau(C_*)$ to RMCF, both of which start at
$C_*$, but for which $D(\bar\tau) \neq \phi^{\bar\tau}(C_*)$.  This contradicts the
uniqueness of solutions to MCF (Theorem \ref{thm-uniqueness})
\end{proof}

\begin{lemma}\label{lemma-compactness-3-setup}
For any $\delta>0$ there exist $\gvol_\delta>0$ and $\tau_\delta$ such that for all
$C\in X$ with $\gvol(C)\leq \gvol_\delta$ and $\delta\leq \hu(C)\leq 2-\delta$, the
solution $C_\tau$ to RMCF starting at $C$ becomes singular before $\tau=\tau_\delta$,
or else satisfies $\hu\bigl(\phi^{\tau_\delta} (C)\bigr)<\delta$.
\end{lemma}

\begin{proof}
We argue by contradiction and assume that a sequence $C_m\in X$ exists with
$\gvol(C_m)\to 0$, and a sequence $\tau_m\to\infty$ such that
$\delta\leq \hu\bigl(\phi^\tau(C_m)\bigr)\leq 2-\delta$ holds for all $m$ and all
$\tau\in[0, \tau_m]$.

Since $\hu(C_m)\geq \delta$ it follows from Lemma~\ref{lemma-Huisken-measures-tight}
that there is an $R_\delta>0$ such that $C_m\cap B_{R_\delta}(0) \neq \varnothing$
for all $m$.  The first compactness Lemma~\ref{lemma-compactness-1} implies that we
may assume that $C_m\rto C_*$ for some $C_*\in\bar X$.  Since $\gvol(C_m)\to 0$,
Lemma~\ref{lemma-gvol-continuous} implies $\gvol(C_*) = 0$, and hence the interior of
$C_*$ must be empty.  It follows that $C_*$ is contained in some hyperplane
$L\subset \R^{n+1}$.  We may assume that $L$ is parallel to
$\R^n\times\{0\} \subset\R^{n+1}$.

We claim the hyperplane $L$ contains the origin.  If this were not true then let the
distance from $L$ to the origin be $2\delta$, so that $L=\R^n\times\{2\delta\}$.
Thus $C_*$ is disjoint from $B^n_r\times[-\delta,\delta]$, and for large enough $m$,
$C_m$ will also not intersect $B^n_r\times[-\delta,\delta]$.  By comparison with a
suitably placed BLT-pancake (see Lemma~\ref{lemma-clearing-out-with-pancakes}) there
exist $\varrho,T>0$ such that $\phi^T(C_m)$ is disjoint from
$B^{n+1}_{\varrho/\delta}$.  Lemma \ref{lemma-Huisken-measures-tight} tells us that
\[
\hu(\phi^T(C_m)) \leq c_n (\varrho/\delta)^n e^{-\varrho^2/(4\delta)}.
\]
If $\delta$ is small enough then this implies $\hu(\phi^T(C_m)) < \delta$.  For large
enough $m$ this contradicts the assumption that $\hu(\phi^\tau(C_m))\geq \delta$ for
all $\tau\in[0, \tau_m]$.

We may assume from here on that $L=\R^n\times\{0\}$.  We will use
Lemma~\ref{lemma-clearing-out-with-expanders} to obtain a contradiction in this case.

The assumption $\hu(C_m)\leq 2-\delta$ implies $\hu(C_*)\leq 2-\delta$.  Choose
$R_\delta$ so that
\[
(4\pi)^{n/2}\int_{B^n_{R_\delta}\times\{0\}} e^{-\|X\|^2/4}dH^n > \frac{2-\delta}{2}.
\]
The integral is exactly the Huisken energy of the two-sided ball
$B^n_{R_\delta}\times\{\pm0\}\subset\R^{n+1}$, so it follows from
$\hu(C_*)=\lim\hu(C_m)$ that $C_*$ cannot contain the entire ball $B^n_{R_\delta}$.
Hence $B^n_{R_\delta}$ contains a point $p$ that does not lie in $C_*$.

Let $q\in C_*$ and $q_m\in C_m$ be the nearest points to $p$ in $C_*$ and $C_m$,
respectively.  Convexity of $C_m$ and $C_*$ together with $C_m\rto C_*$ implies that
$q_m\to q$.

Choose $R>0$ be so large that any $\tilde C\in X$ with
$\tilde C\cap B^{n+1}_R=\varnothing$ satisfies $\hu(\tilde C)<\delta$ (such an $R$
exists according to Lemma~\ref{lemma-Huisken-measures-tight}).  Consider the line
segment $q\bar q$ of length $R$ starting at $q$, in the direction of $p$, and let the
line segment be so long that $\|\bar q\|= R$ (increase $R$ if necessary).  Let $A$ be
the slope from Lemma~\ref{lemma-clearing-out-with-expanders} corresponding to $R$.

\textit{Claim: } for sufficiently large $m$ all $(x, y)\in C_m$ satisfy
$|y|\leq A\|x-\bar q\|$.  If this were not true then along some subsequence we would
have points $(x_m, y_m)\in C_m$ with $|y_m| > A \|x_m-\bar q\|$.

If the sequence $|y_m|$ is bounded, then $\|x_m-\bar q\|\leq A^{-1}|y_m|$ is also
bounded, so we can extract a convergent subsequence whose limit $(x_*, y_*)$ is a
point in $C_*$ for which $|y_*|\geq A\, \|x_*-\bar q\|$.  Since $\bar q\not\in C_*$
we have $\|x_*-\bar q\|>0$ and hence $|y_*|>0$.  This contradicts
$(x_*,y_*)\in C_*\subset\R^n\times\{0\}$, i.e.~$y_*=0$.

Next consider the case where $|y_m|$ is unbounded.  Recall that $q_m\in C_m$ is the
nearest point in $C_m$ to $p$, and that $q_m\to q$.  For large $m$ the line segment
connecting $(x_m, y_m)$ and the point $q_m$ contains a point
$(\tilde x_m, \tilde y_m)$ with $|\tilde y_m|=1$.  By convexity
$(\tilde x_m, \tilde y_m)\in C_m$.  Using $A\|x_m-\bar q\|<|y_m|$ one verifies that
$\|\tilde x_m\|$ is bounded.  Indeed, let $q_m=(a_m, b_m)\in \R^n\times\R$.  Then
$a_m\to q$ and $b_m\to 0$.  By definition of $\tilde{x}_m, \tilde{y}_m$
\[
(\tilde x_m, \tilde y_m)= (1-\theta_m)q_m + \theta_m(x_m, y_m)
\]
where $\theta_m$ is determined by $|(1-\theta_m)b_m+\theta_m y_m| = 1$, i.e.
$\theta_m = \frac{\pm 1-b_m}{y_m-b_m} = \mc O(|y_m|^{-1})$.  Hence
\[
\tilde x_m - q_m = \theta_m(x_m-q_m) = \theta_m(x_m-q_m),
\]
and therefore $\|x_m-q_m\|\leq \|x_m-q\|+\|q-q_m\|\leq A|y_m|+o(1)\leq 2A|y_m|$ if
$m$ is large, because $q_m\to q$ and $|y_m|\to\infty$.  This implies
$\|\tilde x_m-q_m\| \leq 2A + o(|y_m|^{-1})$.  Since $q_m\to q$ it follows that
$\tilde x_m$ is bounded.

Thus we can extract a convergent subsequence of $(\tilde x_m, \tilde y_m)$ whose
limit $(\tilde x_*, \tilde y_*)$ both belongs to $C_*$ and satisfies
$|\tilde y_*|=1$, again a contradiction.

We can now complete our proof by invoking
Lemma~\ref{lemma-clearing-out-with-expanders}.  For large $m$ the set $C_m$ is
contained in the region outside the cone $|y| = A\|x-\bar q\|$ where $\|\bar q\|=R$.
It follows that $\phi^2(C_m)$ lies in the region $\|x\|\geq R$, and hence
$\hu\bigl(\phi^2(C_m)\bigr) < \delta$, contradicting our initial assumptions.
\end{proof}

\begin{lemma}[Third compactness lemma]\label{lemma-compactness-3}
For any $\delta>0$ let $\tau_\delta$ be as above in
Lemma~\ref{lemma-compactness-3-setup}.  Then
\[
E_\delta = \bigl\{C\in X \mid \taum(C)\geq \tau_\delta,\text{ and } \forall
\tau\in[0, \tau_\delta]: \delta\leq \hu(\phi^\tau(C))\leq 2-\delta \bigr\}
\]
is a compact subset of $X$.
\end{lemma}
\begin{proof}
If $C_m\in X$ is a sequence with $\delta\leq \hu(C_m)\leq 2-\delta$ such that
$\hu(\phi^\tau(C_m))\geq \delta$ for all $\tau\in[0, \tau_\delta]$, then we have just
shown in Lemma~\ref{lemma-compactness-3-setup} that $\gvol(C_m)\geq \gvol_\delta$.
Hence, by Lemma~\ref{lemma-compactness-2}, the sequence $C_m$ has a convergent
subsequence.  Denoting this subsequence again by $C_m$, we consider its limit
$C_m\rto C_*$ and show that $C_*\in E_\delta$.

If $\taum(C_*)<\tau_\delta$ then a
$\tau_*\in [0, \taum(C_*)] \subset [0, \tau_\delta]$ with
$\hu\bigl(\phi^{\tau_*}(C_*)\bigr) \leq \delta/2$ would exist.  By continuity of the
semiflow it would then follow that $\phi^{\tau_*}(C_m)<\delta$ for large $m$, which
is not the case.  Thus we see that $\taum(C_*) > \tau_\delta$.  Continuity of the
semiflow then implies that $\phi^\tau(C_m)\rto \phi^\tau(C_*)$ for all
$\tau\in [0, \tau_\delta]$.  Continuity of the Huisken energy finally implies
$\delta\leq\hu(\phi^\tau(C_*))\leq 2-\delta$ for all $\tau\in [0, \tau_\delta]$.
Therefore $C_*\in E_\delta$
\end{proof}

\section{The Invariant set}
\label{sec-invariant-set}

\subsection{Invariant sets}
Here we analyze the invariant sets that were defined in \S\ref{dfn-invariant}.  We
use the notation
\[
\phi^{[a,b]}(C) := \{\phi^\tau(C) \mid a\leq \tau \leq b\}
\]

In addition to the invariant set \(I(h_0,h_1)\), we also consider for any \(\tau>0\)
the sets
\begin{align*}
  N^-_\tau(h_0,h_1) &= \bigl\{ C\in X \mid \phi^{[0,\tau]}(C)\subset  X(h_0, h_1)\bigr\}\\
  N^+_\tau(h_0,h_1) &= \bigl\{ \phi^\tau(C) \mid
                      \exists C\in X: \phi^{[0,\tau]}(C)\subset  X(h_0, h_1)\bigr\}\\
  I_\tau(h_0, h_1)&= N^+_\tau(h_0,h_1) \cap N^-_\tau(h_0,h_1) .
\end{align*}
By definition we have
\[
\forall 0<\tau<\tau' : \quad I_{\tau'}(h_0, h_1) \subset I_{\tau}(h_0, h_1).
\]

\begin{lemma}\label{lemma-I-compact} Assume \(0<h_0<h_1<2\).

\begin{enumerate}[\upshape(a)~]
\item The invariant set defined in \S\ref{dfn-invariant} is given by
\( I(h_0, h_1) = \bigcap_{\tau\geq 0} I_\tau(h_0,h_1)\).
 
\item \(I(h_0,h_1)\) is a compact subset of \(X\).
\end{enumerate}
\end{lemma}

\begin{proof} We simplify our notation and write $N^\pm_\tau$ for
$N^\pm_\tau(h_0,h_1)$.  Choose $\delta>0$ so that $\delta<h_0<h_1<2-\delta$, and let
$\tau_\delta$ and $E_\delta$ be as in Lemma~\ref{lemma-compactness-3}.
Lemma~\ref{lemma-compactness-3} tells us that if $\tau\geq \tau_\delta$ then
$N_\tau^- \subset E_{\delta}$.

Assume for a moment that the first assertion holds and let us begin by proving the
second assertion.  We show that $N_\tau^-$ is a closed subset of $E_{\delta}$ by
verifying that its complement is open.  Let $C\in E_{\delta} \setminus N^-_\tau$ be
given, and suppose that $C_m\rto C$ for some sequence $C_m\in E_{\delta}$.  Either
$\hu(\phi^\tau(C)) < h_0$ or the life span of the RMCF starting from $C$ satisfies
$T_C<\tau$.  In the second case there is a $\sigma\in[0, \tau)$ with
$\hu\bigl(\phi^\sigma(C)\bigr) < \delta$, by Lemma~\ref{lemma-huisken-to-zero}.
Lower semi-continuity of $T_C$ and continuity of the semiflow and the Huisken energy
imply in both cases that there is a $\sigma\in[0,\tau]$ such that
$\hu\bigl(\phi^\sigma(C_m)\bigr) < \delta$ for large $m$, which in turn means that
$C_m\not\in N^-_\tau$ for large $m$.  This shows that $N^-_\tau$ is a closed subset
of $E_\delta$.  Since $E_\delta$ is compact, we have shown that $N^-_\tau$ is
compact.

By definition $N^+_\tau = \phi^\tau(N^-_\tau)$ is the continuous image of a compact
set, so $N^+_\tau$ is also compact.  This implies $I_{\tau}(h_0,h_1)$ is compact for
every $\tau \ge 0$, and hence the set $\bigcap_{\tau \ge 0} I_{\tau}(h_0,h_1)$ is
also compact, proving the second assertion of the lemma.

\smallskip To prove the first assertion, observe first that
$ I(h_0,h_1)\subset\bigcap_{\tau>0} I_\tau(h_0,h_1)$ is automatically true.  To see
that we also have $ \bigcap_{\tau>0} I_\tau(h_0,h_1) \subset I(h_0,h_1)$, take
$C \in \bigcap_{\tau>0} I_\tau(h_0,h_1)$.  This implies that for every $m$ there
exists an orbit $\{C^m_t \mid -m\leq t\leq m\}$ of RMCF, with $C_0^m = C$ and
$\delta\leq \hu (C_t^m)\leq 2-\delta$ for all $t\in [-m, m]$.

For all $s\in [-m, m-\tau_\delta]$ we have $\phi^{\tau}(C^m_s)=C^m_{s+\tau}$, so the
flow starting at $C^m_s$ is defined for all $\tau\in[0, \tau_\delta]$, and it also
satisfies $\hu(\phi^{\tau}(C^m_s)) = \hu(C^m_{s+\tau})\geq \delta$.  Therefore
$C^m_s\in E_\delta$ for all $m$ and all $s\in[-m, m-\tau_\delta]$.

The semiflow $(\tau, C)\mapsto \phi^\tau(C)$ is continuous, so, since
$[0, \tau_\delta]\times E_\delta$ is compact the semiflow restricted to
$[0, \tau_\delta]\times E_\delta$ is uniformly continuous.  Combined with
$C^m_{s+\tau} = \phi^\tau(C^m_s)$ this implies that the orbits
$\tau\in[-m, m-\tau_\delta] \mapsto C^m_\tau$ are equicontinuous.  The Ascoli-Arzela
theorem provides a subsequence $C^{m_j}_\tau$ that converges uniformly on any bounded
interval $|\tau|\leq \tau_*$.  Continuity of the semiflow guarantees that the limit
$C^*_\tau := \lim_{j\to\infty} C^{m_j}_\tau$ is again an orbit (because
$\phi^\sigma(C^*_\tau) = \phi^\sigma\bigl(\lim_j C^{m_j}_\tau\bigr) = \lim_j
\phi^\sigma(C^{m_j}_\tau) = \lim_j C^{m_j}_{\tau+\sigma} = C^*_{\tau+\sigma} $.)
Moreover, $C^{m_j}_0=C$ for all $j$, so $C^*_0=C$.

We have shown that any $C \in \bigcap_{\tau\geq 0}I_\tau(h_0,h_1)$ lies on a complete
orbit $C^*(\tau)$ with $\delta\leq \hu(C^*_\tau)\leq 2-\delta$, and thus
$C\in I(h_0, h_1)$.
\end{proof}

\subsection{Fixed points of the flow}
\label{section-fixed} The convex self-shrinkers are exactly the fixed points of the
flow.  The hypersurfaces that are stationary under RMCF are the generalized cylinders
and the plane through the origin.  Thus the closed convex sets that correspond to
fixed points of the flow $\phi^\tau$ are either of the form\footnote{~In this section
  we abuse notation and identify a convex set $C\subset\R^{n+1}$ with the convex
  hypersurface $M=\partial C$ that is its boundary; the correspondence between convex
  sets and their boundary is one-to-one, since $C$ is the convex hull of its
  boundary.}
\[
\mc R \cdot \bigl(\sphere^{k}\times\R^{n-k}\bigr)
\]
for some $k\in\{1, \dots, n\}$ and $\mc R\in \SO_{n+1}$, where
\( \sphere^k=S^k_{\sqrt{2k}}, \) or else they are a half-space, i.e.
\[
\mc R\cdot\bigl(\R^{n}\times\{0\}\bigr), \text{ where } \R^n\times\{0\} =
\partial\bigl(\R^n\times[0,\infty)\bigr)
\]
again for some $\mc R\in\SO_{n+1}$.

The Huisken functional (as defined in \S\ref{sec-Huisken}) of
$\sphere^{k}\subset\R^{k+1}$ is given by
\[
\hu(\sphere^{k}) =(4\pi)^{-k/2}\int_{\sphere^k}e^{-(2k)/4}dH^k
=\sqrt{4\pi}\left(\frac{k}{2e}\right)^{k/2} \, \Gamma\Bigl(\frac{k+1}{2}\Bigr)^{-1}
\]
where $\omega_k = 2\pi^{\frac{k+1}{2}}/\Gamma\bigl(\frac{k+1}{2}\bigr)$ is the
$k$-dimensional measure of the $k$-dimensional unit sphere.

Since the Huisken functional decreases monotonically under RMCF, we have
\[
\hu(\sphere^{k}) = \hu(\sphere^{k}\times\R) > \hu(\sphere^{k+1}) \qquad \mbox{for
  all}\,\, k \geq 1.
\]
Using Stirling's approximation for $n!=\Gamma(n+1)$, one can verify
$\hu(\sphere^k)\to \sqrt2$ as $k\to\infty$, and thus we obtain
\eqref{eq-critical-values}.

\subsection{Quotient by $\SO_{n+1}$}
The group $\SO_{n+1}$ acts continuously on $X$.  Since the group is compact the
quotient space $X/\SO_{n+1}$ is again a Hausdorff space.  RMCF is invariant under
rotations, so the flow $\phi^\tau$ also defines a flow on $X/\SO_{n+1}$.  We abuse
notation and also denote this quotient flow by
$\phi^\tau: X/\SO_{n+1}\to X/\SO_{n+1}$.

The invariant sets $I(h_0, h_1)$ are also invariant under rotations, so we can also
consider their quotients $I(h_0, h_1)/\SO_{n+1}$.

\subsection{Fixed points in $X$ and RMCF as gradient flow on $X/\SO_{n+1}$}
As we observed in \S\ref{sec-fixed-points}, the RMCF on the quotient space
$X/\SO_{n+1}$ has a discrete set of fixed points, corresponding to the sphere and
generalized cylinders $\Sigma^k$, and the hyperplane $\Pi$ through the origin (seen
as boundary of half-space).

\begin{prop}
The Huisken energy $\hu: X/\SO_{n+1}\to\R_+$ is a Lyapunov function for RMCF,
i.e.~for any $C\in X/\SO_{n+1}$ for all $\tau \geq 0$ one has
$\hu(\phi^\tau(C))\leq \hu(C)$, and if for some $\tau>0$ one has
$\hu\bigl(\phi^\tau(C)\bigr) = \hu\bigl(C\bigr)$, then $C$ is a fixed point for the
semiflow $\phi^\tau$.
\end{prop}

\begin{proof}
This is exactly what Huisken's monotonicity theorem says.
\end{proof}

The following consequence of our compactness theorems and the existence of Huisken's
Lyapunov function is common in the theory of dynamical systems.  We include a proof
for completeness.
\begin{prop} \label{prop-descr-I} Suppose we are given two regular values
$0<h_0< h_1<2$ of the Huisken energy, i.e.~suppose $h_0,h_1$ satisfy
$0 < h_0 < h_1 < 2$ and $h_j\neq \hu(\Sigma^k), \hu(\Pi)$ for all $k$.  Then the
invariant set $I(h_0, h_1)/\SO_{n+1}$ consists exactly of all fixed points
$\Sigma^k$, with $h_0<\hu(\Sigma^k)<h_1$, and, if $h_0<1<h_1$, also the hyperplane
$\Pi$ through the origin, as well as all connecting orbits between these fixed
points.
\end{prop}
\begin{proof}
By definition, $I( h_0, h_1)/\SO_{n+1}$ consists of all
$C_0 \in X(h_0, h_1)/\SO_{n+1}$ for which there exists an orbit
$C:\R\to X(h_0, h_1)/\SO_{n+1}$ of RMCF with $C(0)=C$.  We will use the standard
Lyapunov function argument to show that both limits of $C(\tau)$ as
$\tau \to\pm\infty$ exist, and that they are fixed points of the semiflow.

The $\omega$-limit set of the orbit $\{C(\tau): \tau \in\R\}$ is the set of all
possible limit points $\lim C(\tau_i)$ for any sequence $\tau_i\to\infty$.
Alternatively,
\[
\omega\bigl(\{C(\tau)\mid \tau \in\R\}\bigr) =
\bigcap_{\tau\in\R}\,\overline{\{C(\sigma)\mid \sigma\geq \tau\}}\;.
\]
If $\tau \in\R$ then $\overline{\{C(\sigma)\mid \sigma\geq \tau\}}$ is a compact,
connected, forward invariant subset of $I(h_0,h_1)/\SO_{n+1}$.  Since the family of
sets $\overline{\{C(\sigma)\mid \sigma\geq \tau\}}$ decreases as $\tau\to\infty$,
their intersection over all $\tau\in\R$ is also compact, connected, and forward
invariant.

The Huisken functional restricted to $\omega\bigl(\{C(\tau)\mid \tau \in\R\}\bigr)$
equals $\lim_{\tau\to\infty} \hu(C(\tau))$.  In particular, $\hu$ is constant on the
$\omega$-limit set.  It follows that the $\omega$-limit set consists of fixed points.
Since there are only finitely many fixed points $\Sigma^k$, and since the
$\omega$-limit set is connected, it consists of exactly one fixed point, say,
$\Sigma^k$ (or $\Pi$).  This implies that $C(\tau)\to\Sigma^k$ (or $\Pi$) as
$\tau\to\infty$.

A very similar argument replacing the $\omega$-limit with the $\alpha$-limit set
\[
\alpha\bigl(\{C(\tau)\mid \tau \in \R\}\bigr) =
\bigcap_{\tau\in\R}\overline{\{C(\sigma)\mid \sigma\leq \tau\}}
\]
shows that $\lim_{\tau\to-\infty}C(\tau)$ exists and it is a fixed point of the flow.
\end{proof}

\subsection{Point symmetric convex sets}
\label{sec-point-symmetry}
From now on consider the space $X_s$ of closed convex sets that are invariant under
point reflection, defined as in section \ref{sec-point-symm},
\[
X_s = \{C\in X \mid C = -C\}.
\]
and the corresponding quotient space $X_s /\SO_{n+1}$.

The property that describes the set $X_s$ is invariant under the mean curvature flow
and hence the arguments of Lemma \ref{lemma-I-compact} show that the set
\[
I_s(h_0,h_1) := I(h_0,h_1)\cap X_s
\]
is a compact subset of $X_s$.

The reason we restrict ourselves to this class is that it provides a natural way to
restrict ourselves to convex sets that are either compact or of the form
$\tilde C\times \R^k$ for some compact convex set $\tilde C$.  This is shown in the
next lemma where we also gather some other properties of $X_s(h_0, h_1)$.

\goodbreak
\begin{lemma}
\label{lemma-rescaling-time}~
\begin{enumerate}[\upshape(a)]
\item Every $C\in X_s$ is either compact, or of the form $\tilde C\times \R^k$ for
some compact symmetric convex set $\tilde C\subset \R^{n-k+1}$.

\item The unrescaled MCF of $C$ has a finite extinction time $\Tx(C)>0$, which is a
continuous function of $C\in X_s$.

\item The extinction time scales parabolically,
i.e.~$\Tx(\kappa \, C) = {\kappa^{2}} \, \Tx(C)$.  In particular, the extinction time
of $\Tx(C)^{-1/2}\, C$ is always exactly $1$.

\item Given any $C \in X_s $, with $\Tx(C)\geq 1$ the Rescaled Mean Curvature Flow
$\phi^\tau ( C)$ is defined for all $\tau > 0$.

\item The set
\begin{equation} \label{eqn-Zs} Z_s := \{C\in X_s \mid \Tx(C) = 1\}/ \SO_{n+1},
\end{equation}
is closed in $X_s/\SO_{n+1}$ and forward invariant under Rescaled Mean Curvature
Flow.
\end{enumerate}
\end{lemma}

\begin{proof}
(a) If $C$ does not contain a ray it must be compact.  If $C$ does contain a ray,
then by symmetry it also contains the reflection of this ray, and thus it contains a
line parallel to the two rays.  This implies that $C$ is a product $\bar C\times\R$,
where $\bar C\subset \R^{n}$ is a symmetric closed convex subset.  Proceeding by
induction on the dimension $n$ we find that $C$ is either compact or
$\tilde C\times\R^k$ for some compact symmetric convex $\tilde C\subset\R^{n-k+1}$.

(b) If $C$ is compact then its MCF vanishes in finite time.  If
$C=\tilde C\times\R^k$ with $\tilde C$ compact, then the MCF $\tilde C_t$ of
$\tilde C$ vanishes in finite time.  The MCF of $C$ is $C_t=\tilde C_t\times \R^k$,
and therefore vanishes at the same time.  Continuous dependence follows from
Proposition \ref{lemma-semiflow-continuous}.

(c) This follows from the scaling invariance of MCF.

(d) By definition, $\phi^\tau(C) = e^{\tau/2} C_{1-e^{-\tau}}$, where $\hat C_t$ is
the MCF starting at $C$.  Since $\Tx(C)\geq 1$, $\hat C_t$ exists for all $t \leq 1$,
and therefore $\phi^\tau(C)$ exists for all $\tau \geq 0$.

(e) Suppose $\Tx(C)=1$, and let $\hat C_t$ be the MCF starting at $C$.  Then,
$\Tx(\hat C_t) = 1-t$ for any $t\in(0, 1)$, and hence
\[
\Tx\bigl(\phi^\tau(C)\bigr) = \Tx \bigl(e^{\tau/2} \hat C_{1-e^{-\tau}}\bigr)
=e^{\tau}\, \Tx \bigl(\hat C_{1-e^{-\tau}}\bigr) =e^{\tau}\bigl(1-(1-e^{-\tau})\bigr)
=1.
\]

Continuity of the extinction time implies that $Z_s$ is a closed subset of
$X_s/\SO_{n+1}$.
\end{proof}

\smallskip

\begin{prop}~\null
\begin{enumerate}[\upshape (a)]
\item $\hu(C)\geq \hu(\Sigma^n)$ for all $C\in Z_s$.
\item For any $h\in (0, 2)$ the set $\{C\in Z_s \mid \hu(C)\leq h\}$ is compact.
\end{enumerate}
\end{prop}

\begin{proof} To prove (a) consider any $C\in Z_s$.  Since $C$ is symmetric, it is of
the form $\tilde C\times\R^{n-k}$ for some compact convex symmetric
$\tilde C\subset\R^{k+1}$.  By Huisken's theorem the unrescaled MCF
$\hat {\tilde C} _t$ starting at $\tilde C$ shrinks to a point, and by the assumption
$C\in Z_s$ we know that this happens at $t=T_{\rm max}(C) = 1$.  It follows that the
rescaled MCF $\phi^\tau(C)$ converges to the cylinder $\Sigma^k$.  Therefore
$\hu(C)\geq \hu(\Sigma^k)\geq \hu(\Sigma^n)$.

Part (b) of the Proposition now follows from Lemma~\ref{lemma-compactness-3} since we
have just shown that the Huisken functional $\hu$ is bounded from below on $Z_s$ by
$\hu(\Sigma^n)$, and given that $\tau_{\max}(C) = \infty$ by the definition of the
set $Z_s$.

\end{proof}

\section{Topology of the invariant set}
\label{sec-topology}

For given $0 < h_0<h_1< 2$ we had defined the invariant set $I_s(h_0, h_1)$ to be the
set of all $C\in Z_s$ for which there exists a complete orbit
$\{C_\tau\}_{\tau\in\R} \subset Z_s$ of RMCF with $C_0=C$ and
$h_0\leq \hu(C_\tau) \leq h_1$ for all $\tau\in\R$.  In this section we assume
further that
\begin{equation}
\label{eq-h0h1-good}
0<h_0<\hu(\Sigma^n) \text{ and } \hu(\Sigma^1)< h_1 <2.
\end{equation}
As we observed in~\eqref{eq-Isn}, the isolated invariant set $I_s(h_0, h_1)$ does not
depend on $h_0,h_1$ as long they satisfy \eqref{eq-h0h1-good}, and thus we may write
\[
I_s^n = I_s(h_0, h_1)
\]
where $0<h_0<\hu(\Sigma^n)$, and $\hu(\Sigma^1)<h_1<2$.

\subsection{Path connectedness of $I_s(h_0,h_1)$}\label{ss-pc}

\begin{prop}\label{prop-path}
If $0<h_0<\hu(\Sigma^n)$ and $\hu(\Sigma^1)<h_1<2$ then the space $I_s(h_0,h_1)$ is
path connected.
\end{prop}

\begin{proof}
  We say that ``$C_0, C_1\in I_s(h_0,h_1)$'' are connected if there is a path, i.e.~a
  continuous function $f:[0,1]\to I_s(h_0,h_1)$ with $f(0)=C_0$, $f(1)=C_1$.  The
  lemma claims that all $C_0, C_1\in I_s(h_0, h_1)$ are connected.

  The relation ``$C_0, C_1$ are connected'' is an equivalence relation.  E.g.~the
  relation is transitive because if $f_1, f_2:[0,1]\to I_s(h_0,h_1)$ are paths with
  $f_2(0)=f_1(1)$ then the concatenation
  \[
    f_3(s) = \left\{
      \begin{aligned}
        f_1(2s), \,\,\,\, &s\in [0,1/2] \\
        f_2(2s-1), \,\,\,\, &s \in [1/2,1].
      \end{aligned}
    \right.
  \]
  is a path from $f_1(0)$ to $f_2(1)$.

  We will first show that each fixed point $\Sigma^j$ is connected to $\Sigma^n$.
  Then we show that any $C\in I_s(h_0,h_1)$ that is not a fixed point is connected to
  some $\Sigma^j$.

  For any $k\in\{1, \dots,n\}$ Brian White \cite{Wh} and Haslhofer-Hershkovits
  \cite{HH} constructed an orbit $\tilde C^k_\tau$ of RMCF connecting $\Sigma^k$ to
  $\Sigma^n$.  By reparametrizing this orbit we can construct a path
  $f_k:[0,1]\to I_s(h_0,h_1)$ by setting
  \[
    f_k(s) =
    \begin{cases}
      \Sigma^k \quad &s = 0 \\
      \tilde{C}^k_{\tan(\pi(s-1/2))} \quad &s\in (0,1) \\
      \Sigma^n \quad &s = 1
    \end{cases}
  \]
  The path $s\mapsto f_k(s)$ is continuous at $0<s<1$ because the semiflow is
  continuous (Proposition \ref{lemma-semiflow-continuous}), and it is also continuous
  at the endpoints $s=0, 1$ because $\tilde C^k_\tau \to \Sigma^k$ as
  $\tau\to-\infty$ and $\tilde C^k_\tau\to\Sigma^n$ as $\tau\to+\infty$.  Hence,
  $f_k(s)$ defines a path in $I_s(h_0,h_1)$ connecting $\Sigma^k$ and $\Sigma^n$.

  Next, suppose $C\in I_s(h_0, h_1)$ is not a fixed point.  This means that there is
  an entire solution $C_\tau$ of RMCF with $C=C_0$ and for which
  $C_\tau\in I_s(h_0,h_1)$ for all $\tau\in\R$.  Since the $C_\tau$ are point
  symmetric they are all of the form $N^k_\tau \times \R^{n-k}$ for some
  $k\in\{1, \dots, n\}$, where $N^k_\tau\subset\R^{k+1}$ is a compact convex set and
  evolves by RMCF.  In view of $\hu(N^k_\tau)=\hu(N^k_\tau\times\R^{n-k})$ we know
  that $h_0 < \hu(N^k_\tau) < h_1$ for all $\tau\in\R$, and thus $N^k_\tau$ converges
  to $\sphere^k\subset\R^{k+1}$ as $\tau\to\infty$.  Therefore
  $C_\tau= N^k_\tau\times\R^{n-k}$ converges to $\Sigma^k$ as $\tau\to\infty $, and
  \[
    \bar{f}(s) =
    \begin{cases}
      C_{-\log(1-s)} \quad &s\in [0,1) \\
      \Sigma^k \quad & s = 1
    \end{cases}
  \]
  defines a path from $C$ to $\Sigma^k$.
\end{proof}

\smallskip

\subsection{Contractibility of $I_s(h_0,h_1)$.}

\smallskip

Recall that we have \emph{conjectured} (see Conjecture \ref{conj-main}) that the
invariant set $I_s^n$ is homeomorphic to an $n-1$-simplex.  In section \ref{sec-conj}
we will present some arguments supporting the conjecture in the case $n = 3$.  For
general $n$, at this point we only know that $I_s^n$ is a connected and compact
metrizable space.  The subset
\[
I_{sc}^n = \bigl\{ C\in I_s^n \mid C\text{ is compact}\bigr\}
\]
is open in dense in $I_s^n$, and because Rescaled Mean Curvature Flow starting at any
compact surface converges to a round sphere, it follows that RMCF defines a
deformation retraction of $I_{sc}^n$ onto $\{\Sigma^n\}$.  While this does not imply
that $I_s^n$ is contractible\footnote{For example a sphere $S^n$ is the union of a
  contractible open dense set $S^n\setminus\{p\}\equiv B^n$, and one point $\{p\}$.}
we will show in this section $I_s^n$ is homologically trivial in the following weaker
sense.

\begin{prop} \label{thm-I-no-homotopy} There exists a nested sequence of contractible
compact neighborhoods $N_1\supset N_2\supset\cdots$ of $I_s^n$ with
$I_s^n = {\bigcap}_{k\geq 1}N_k$.
\end{prop}

We begin our proof of Proposition \ref{thm-I-no-homotopy} by showing that the space
$Z_s$ itself is contractible.

\begin{lemma}
The sets $X_s$, $X_s/\SO_{n+1}$ and $Z_s$ are contractible.
\end{lemma}
\begin{proof}
\textit{$X_s$ is contractible: } For any closed convex set $C\subset\R^{n+1}$ one
defines its support function $h_C:S^n\to (0, \infty]$ by
\[
h_C(x) = \sup_{y\in C} \langle x, y\rangle.
\]
The support function determines the convex set by
\[
C = {\bigcap}_{x\in S^n}\bigl\{ y\in \R^{n+1} \mid \langle x, y\rangle \leq
h_C(x)\bigr\}
\]
A convex set $C$ belongs to $X_s$ if and only if its support function $h_C$ is an
even function.

Let $\psi:[0,1]\times(0, \infty] \to (0,\infty]$ be a deformation retraction of the
interval $(0, \infty]$ onto the point $1\in(0,\infty]$.  For example, one could
choose
\[
\psi(\theta, x) = \frac{\theta +x}{1+\theta x}.
\]
For any convex set $C$ we define $\Psi_\theta(C)$ to be the set
\[
\Psi_\theta(C) = {\bigcap}_{x\in S^n} \left\{ y\in \R^{n+1}\,\big|\, \langle y,
x\rangle \leq \psi(\theta, h_C(x)) \right\}
\]
If $C\in X_s$ then $\Psi_0(C)=C$ and $\Psi_1(C)$ is the unit sphere.  Thus
$\Psi_\theta$ defines a deformation retraction of $X_s$ onto $\Sigma^n$.

\smallskip

\textit{$X_s/\mathrm{\SO}(n+1, \R)$ is also contractible.  }  The explicit retraction
$\Psi_\theta$ of $X_s$ is invariant under the action of $\mathrm{\SO}(n+1, \R)$ and
therefore provides a deformation retraction of
\[
X_s/\mathrm{\SO}(n+1, \R)
\]
onto $\Sigma^n$.

\textit{$Z_s$ is contractible.  }  If $C\in Z_s$ then the deformation
$\Psi_\theta(C)$ will in general not stay within $Z_s$.  We must therefore modify the
maps $\Psi_\theta$ to keep their images within~$Z_s$.  Consider the map
$\Lambda: X_s/\mathrm{\SO}(n+1, \R) \to Z_s$ defined in terms of the extinction time
from Lemma~\ref{lemma-rescaling-time} by
\[
\Lambda(C) = T_{\mathrm{ex}}(C)^{-1/2}C.
\]
This map is a retraction of $X_s/\mathrm{\SO}(n+1, \R)$ onto $Z_s$ ($\Lambda$ is
continuous and $\Lambda|_{Z_s}$ is the identity).  Therefore
$\Lambda\circ \Psi_\theta$ is a deformation retraction of $Z_s$ onto the point
$\Sigma^n\in Z_s$.
\end{proof}

For any $h_1\in(\hu(\Sigma^1), 2)$ consider
\[
N(h_1) = \{C\in Z_s \mid \hu(C)\leq h_1 \}
\]
For any $C\in Z_s$ we define
\[
T_{h_1}(C) = \inf \{ \tau \geq 0 \mid \phi^\tau(C)\in N(h_1)\}.
\]
\begin{lemma}
For all $C\in Z_s$ one has $T_{h_1}(C)<\infty$.  The function $T_{h_1}:Z_s\to\R$ is
continuous.  The map $F:Z_s\to Z_s$ defined by
\[
F(C) = \phi^{T_{h_1}(C)}(C)
\]
is a retraction from $Z_s$ to $N(h_1)$.  In particular, $N(h_1)$ is contractible.
\end{lemma}
\begin{proof}
We begin by proving continuity of $T_{h_1}$.

We abbreviate $T_1=T_{h_1}(C)$.  Let $\epsilon>0$ be given.  Since
$\tau\mapsto\hu(\phi^\tau(C))$ is strictly decreasing, and since
$\hu(\phi^{T_1}(C)) = h_1$ it follows that
$h_\epsilon := \hu(\phi^{T_1+\epsilon}(C)) < h_1$.  Both the flow and the Huisken
functional are continuous, so there is an open neighborhood $U\subset Z_s$ of $C$
such that $\hu(\phi^{T_1+\epsilon}(C')) < h_1$ for all $C'\in U$.  This implies that
$T_{h_1}(C') < T_1+\epsilon$ for all $C'\in U$, which shows that $T_{h_1}$ is upper
semicontinuous.

Suppose on the other hand that there is a sequence $C_n\to C$ for which
$T_{h_1}(C_n) \leq T_1-\epsilon$ holds.  We may then extract a subsequence such that
$T_{h_1}(C_n) \to T_0$ for some $T_0\leq T_1-\epsilon$.  By continuity of the flow
and the Huisken functional we have
$\hu(\phi^{T_0}(C)) = \lim_{n\to\infty} \hu(\phi^{T_{h_1}(C_n)}(C_n))
=\lim_{n\to\infty} h_1 = h_1$.  Since the Huisken functional is strictly decreasing
along the flow, we conclude $h_1 = \hu(\phi^{T_1}(C)) < \hu(\phi^{T_0}(C)) = h_1$, a
clear contradiction.

We have therefore also shown that $T_{h_1}$ is lower semicontinuous, and thus
$T_{h_1}$ is continuous.  Continuity of the flow implies that the map
$F(C) = \phi^{T_{h_1}(C)}(C)$ is continuous.

For any $C\in N(h_1)$ one has $T_{h_1}(C)=0$ and thus $F(C)=C$.  Therefore $F$ is a
retraction.
\end{proof}

As an immediate consequence of the previous lemma we obtain.

\begin{lemma}
The sets $N(h_1)$ are forward invariant neighborhoods of $I_s^n$, and
\[
I_s^n = {\bigcap}_{\tau\geq 0}\phi^\tau\bigl(N(h_1)\bigr).
\]
\end{lemma}

\begin{proof}[Proof of Proposition \ref{thm-I-no-homotopy}] Choose
$N_k = \phi^k\bigl(N(h_1)\bigr)$.  The previous lemmas imply that each $N^k$ is
contractible and $ I_s^n = {\bigcap}_{k \geq 1} N_k.$ This concludes the proof of the
proposition.
\end{proof}

While it is generally not true that the intersection of a nested family of compact
contractible spaces is itself contractible\footnote{~Consider as an example the
  closure $\Gamma$ in $\R^2$ of the graph of $y=\sin(\pi/x)$, $x\in\R\setminus\{0\}$.
  Any $\epsilon$-neighborhood in $\R^2$ of $\Gamma$ is contractible, but $\Gamma$
  itself is not path-wise connected, having three path components.  Its simplicial
  cohomology group in degree $0$ is $H^0(\Gamma)=\Z^3$.  On the other hand $\Gamma$
  is connected, and its Čech cohomology group in degree $0$ is
  $\check H^0(\Gamma) = \Z$.  }, we can now claim, using Proposition
\ref{thm-I-no-homotopy} that $I_s^n$ is homologically trivial in the sense of the
following two propositions.

\begin{prop}
The Čech cohomology groups\footnote{~For a reference on Čech cohomology we refer the
  reader to either of the books by Bott and Tu \cite{BTu} or Spanier \cite{Spa}.} of
$I_s^n$ are those of one point space, i.e.~$\check H^k(I_s^n)=0$ for $k\geq 1$ and
$\check H^0(I_s^n)=\Z$.
\end{prop}

\begin{proof}
Each $N_k$ is contractible so its Čech cohomology is that of a point.  By continuity
of Čech cohomology it follows that the Čech cohomology groups of
$I_s^n = {\bigcap}_{k\geq 1} N_k $ also is that of a point.
\end{proof}

\begin{prop}
If $I_s^n$ is a neighborhood retract, then $I_s^n$ is contractible.
\end{prop}
\begin{proof}
Recall that by definition $I_s^n$ is a neighborhood retract if there is a
neighborhood $U\supset I_s^n$ and a continuous map $\psi:U\to I_s^n$ with
$\psi|_{I_s^n}=\mathrm{id}_{I_s^n}$.  For sufficiently large $k$ we have
$N_k\subset U$.  The space $N_k$ is contractible, so there exists a homotopy
$F_\theta:N_k\to N_k$ with $F_0=\mathrm{id}_{N_k}$ and with $F_1$ constant.  Then
$\psi\circ F_\theta|_{I_s^n}$ is a homotopy between the identity on $I_s^n$ and a
constant map, which shows that $I_s^n$ is contractible.
\end{proof}

\subsection{Proof of Theorem \ref{thm-main}}
Let us now collect the results of previous sections which give the proof of Theorem
\ref{thm-main} as stated in the introduction.  Part (a) was shown in Proposition
\ref{prop-good}.  Part (b) was shown in Proposition \ref{prop-path}.  Part (c) was
shown in Lemma \ref{lemma-I-compact}.  Finally, part (d) was shown in Proposition
\ref{thm-I-no-homotopy}.

\subsection{Proof of Theorem~\ref{thm-no-isolated-orbits}}\label{sec-proof-no-isolation}
Let $\Gamma=\{C_\tau \mid \tau\in\R\}\subset I_s^n$ be a complete orbit, and let
$\mc U\subset I_s^n$ be an open subset with $\mc U\cap \Gamma\neq\varnothing$.

If $\Gamma$ is only a one-point set, say $\Gamma=\{C_*\}$, then $C_*$ must be a fixed
point of RMCF.  In this case there exists another orbit
$\tilde\Gamma= \{\tilde C_\tau \mid \tau\in\R\}\subset I_s^n$ that converges to $C_*$
either as $\tau\to\infty$ or $-\infty$.

We may therefore assume that $\Gamma$ is not a fixed point.  Let
$C_\pm = \lim_{\tau\to\pm\infty} C_\tau$.  By shrinking the neighborhood $\mc U$, if
necessary, we can arrange that there are open neighborhoods $\mc N_\pm \ni C_\pm$
with $\mc N_\pm\cap\mc U=\varnothing$.  The theorem follows if we can show that
$\mc U$ contains a $\tilde C\in I_s^n\setminus\Gamma$.  Aiming for a contradiction we
now assume that $\mc U\cap\Gamma = \mc U$.

Since $\tau\in\R\mapsto C_\tau\in\Gamma$ is a homeomorphism, we may assume that
$\mc U\subset\Gamma$ is an open interval.  Let $\mc K = I_s^n\setminus\mc U$.
Corresponding to the pair $(I_s^N, \mc K)$ we have a long exact sequence on
cohomology
\begin{equation}
\label{eq-long-exact}
\check H^0(I_s^n)\longrightarrow \check H^0(\mc K) \stackrel{d^*}
\longrightarrow \check H^1(I_s^n, \mc K)\longrightarrow \check H^1(I_s^n)\longrightarrow \cdots
\end{equation}
We know that $I_s^n$ is connected.  For $n\geq 3$ it is also true that $\mc K$ is
connected.  Indeed, every $C'\in\mc K$ lies on a complete orbit
$\{C'_\tau\}\subset I_s^n$ with $C'=C'_0$, and for which $C'_\tau\in\mc K$ holds for
all $\tau\leq 0$ or for all $\tau\geq 0$.  Then
$C'_-=\lim_{\tau\to \pm\infty}C'_\tau$ is one of the fixed points, and we have found
a path from any $C'\in\mc K$ to one of the fixed points $\Sigma^k$ that stays in
$\mc K$.  By considering the simple connecting orbits (see
\S\ref{sec-simple-connections-unique} and also Figure~\ref{fig:standardflow}) one
sees that any two fixed points $\Sigma^k$ and $\Sigma^\ell$ can be connected by a
concatenation of simple connecting orbits.  Moreover, when $n\geq3$ there is more
than one such path between two fixed points.  Since, by assumption, the open set
$\mc U$ only intersects one orbit (namely $\Gamma$), we see that all fixed points are
connected by a sequence of simple connecting orbits all of which lie in $\mc K$.
Hence $\mc K$ is path connected.

Since $I_s^n$ and $\mc K$ both are connected, the map
$\check H^0(I_s^n)\longrightarrow \check H^0(\mc K)$ is an isomorphism.  Exactness of
the sequence \eqref{eq-long-exact} implies that
$d^*: \check H^0(\mc K) \longrightarrow \check H^1(I_s^n, \mc K)$ is the zero map.
Therefore $\check H^1(I_s^n, \mc K)\longrightarrow \check H^1(I^n_s)$ is injective.
Since $\check H^1(I_s^n)=(0)$ we conclude that $\check H^1(I_s^n, \mc K) = (0)$.

On the other hand, by excision
$\check H^1(I_s^n, \mc K) \cong \check H^1(\mc U, \partial\mc U)\cong \Z$ because
$\mc U$ is homeomorphic to an interval.  The contradiction shows that the open set
$\mc U$ must intersect other orbits than just $\Gamma$, as claimed in
Theorem~\ref{thm-no-isolated-orbits}.

\section{Shadowing heteroclinic orbits}
\label{sec-hetero}

In this section we show a few properties of the set
$I_s(h_0,h_1) := I(h_0, h_1) \cap X_s(h_0, h_1) $ defined in Definition
\ref{dfn-invariant}, having in mind that one of our long-term goals is to prove
Conjecture \ref{conj-main} below.  The main focus of this section is proving Theorem
\ref{thm-shadowing} whose proof is expected to play a role in tackling
Conjecture \ref{conj-main}, as we hall see in section \ref{sec-conj}.  We begin by
observing some further properties of $I_s(h_0, h_1)$.  Recall that in section
\ref{sec-invariant-set} we have showed Proposition \ref{prop-descr-I}, whose proof
easily yields the following.

\begin{prop} \label{prop-good} Assume that $0 < h_0 < h_1 <2$.  Then the invariant
set $I_s(h_0,h_1)$ consists exactly of all fixed points $\Sigma^k$, with
$h_0<\hu(\Sigma^k)<h_1$ and all connecting orbits between these fixed points.
\end{prop}

\begin{proof} The proof is similar to the proof of Proposition \ref{prop-descr-I}.
Note that hyperplanes through the origin do not belong to the class $I_s(h_0,h_1)$
due to point symmetry assumed.  Indeed, if $P\subset\R^{n+1}$ is a hyperplane
containing the origin, then it corresponds to a half-space
$H = \{x\in\R^{n+1} \mid \langle a, x\rangle\geq 0\}\in X$.  Such a half-space does
not have the required point symmetry and thus it does not belong to~$X_s$.
\end{proof}

Recall that if $C\in I_s(h_0,h_1)$, by definition there exists an eternal solution
$\{C_\tau\}_{\tau\in \mathbb{R}}$ to RMCF, with $C_0 = C$ and
$C_\tau\in X(h_0,h_1)\cap X_s$, for all $\tau\in \mathbb{R}$.

\begin{remark}\label{rem-Tm} Any such solution satisfies $T_{\max} =1$, that is the
solution $\hat C_t$ of the unrescaled MCF has extinction time $T_{\max}=1$.
Otherwise, if $T_{\max} > 1$, then $C_{\tau}$ would converge to infinity as
$\tau\to \infty$, contradicting that $C \in I_s(h_0,h_1)$.  On the other hand, if
$T_{\max} < 1$, then the solution $C_{\tau}$ would not be eternal and again
$C\notin I_s(h_0,h_1)$.  Hence, $T_{\max} = 1$.
\end{remark}

The following proposition is a direct consequence of Theorem 1.2 in \cite{BLL} and
Proposition \ref{prop-good}.

\begin{prop}[c.f.  Theorem 1.2 in  \cite{BLL}]\label{prop-ncl}
Suppose that $0 < h_0 < h_1 < 2$.  Every solution $\{C_\tau\}_{\tau\in \mathbb{R}}$
to RMCF, with $C_0 \in I_s(h_0,h_1)$, is $\alpha$-noncollapsed.
\end{prop}

\begin{proof} By Proposition \ref{prop-good} one can see that the invariant set
$I_s(h_0, h_1)$ consists exactly of all fixed points $\Sigma^k$ with
$h_0<\hu(\Sigma^k)<h_1$ and the connecting orbits between these fixed points.  In
view of rescaling \eqref{eq-rescaling} this implies that
$\cup_{t \in (-\infty, 1)} \hat C_t = \R^{n+1}$.  Theorem 1.2 in \cite{BLL} then
directly shows that the solution $\{\hat{C}_t\}_{t\in (-\infty,1)}$, and hence its
rescaling $\{C_\tau\}_{\tau\in \mathbb{R}}$ are $\alpha$-noncollapsed.
\end{proof}

Recall that our long term goal is to prove Conjecture~\ref{conj-main}.  In section
\ref{sec-conj} we will prove this conjecture for $n=3$ under an assumption that we
think can be removed in the near future.  The goal of this section is to prove
Theorem \ref{thm-shadowing} using Propositions \ref{prop-nbhd-sigma2} and
\ref{prop-nbhd-sigma2-backwards}.  As it will turn out in section \ref{sec-conj}, we
will need these two propositions to prove the Conjecture when $n =3$, under the
additional assumption.

\smallskip

For any dimension $n \geq 2$, we define $\Gamma(i,j)$ as the set of points (convex
sets) that lie on a connecting orbit from $\Sigma^i$ to $\Sigma^j$, i.e.
\begin{equation}\label{eq-edges}
\Gamma(i,j) = \left\{C\in X_s/\SO_{n+1} \;\Bigg|\;
\parbox{15em}{\centering\text{$\exists$ RMCF $\{C_\tau\}_{\tau\in\R}$ with $C_0=C$, }\\
  $C_\tau \to \Sigma^i \; (\tau\to-\infty)$, \\
  $C_\tau \to\Sigma^j \; (\tau\to\infty)$.} \, \right\}
\end{equation}

\smallskip

We have just seen in Proposition \ref{prop-ncl} that all sets $\Gamma(i,j)$, if
non-empty, consist of $\alpha$-noncollapsed solutions.  Furthermore, the monotonicity
of Huisken's energy implies that when $i > j$, the above sets are empty and when
$i=j$ they consist of just one element $\Sigma^i$.  Solutions in the sets
$\Gamma(i,j)$, $i <j$, in any dimension, have been constructed by White \cite{Wh} and
Haslhofer-Hershkovitz \cite{HH}, implying that the sets $\Gamma(i,j)$, $i <j$ are
{\em non-empty.  }  The uniqueness, up to rigid motions, of {\em two-convex} orbits
in $\Gamma(i, j)$ follows from the results in \cite{ADS1, ADS2}.  This in particular
implies that uniqueness holds when $n=2$.  In addition it was shown in \cite{CHHW}
that orbits in $\Gamma (n-1, n)$ are unique for any $n \geq 3$ without the assumption
of two-convexity.

\smallskip

From now on we will restrict ourselves to dimension $n=3$, i.e.~the case of
$3$-dimensional convex hypersurfaces in $\R^4$.  In dimension $n=3$, more can be said
regarding uniqueness: first, the uniqueness of orbits connecting $\Sigma^1$ to
$\Sigma^2$ and orbits connecting $\Sigma^2$ to $\Sigma^3$ follows by the results in
\cite{ADS2, CHHW}.  However, orbits from $\Sigma^1$ to $\Sigma^3$ are not in general
unique.  They are unique, if one assumes $SO_2\times SO_2$ symmetry as shown in
\cite{DH}.  In the same paper it was shown that there is a one parameter family of
orbits connecting $\Sigma^1$ to $\Sigma^3$ that are only $\SO_2$-symmetric.  The
uniqueness of $\SO_2$-symmetric orbits connecting $\Sigma^1$ to $\Sigma^3$, up to
rigid motions and rescaling, has been recently discussed in \cite{CDDHS}.

\smallskip

The above existence and the uniqueness results suggest that Conjecture
\ref{conj-main} should hold (at least in the case $n = 3$), as we expect the
one-parameter family of $\SO(2)$-orbits constructed in \cite{DH} to fill out the two
dimensional simplex, which we can visualize as a triangle whose vertices are fixed
points and whose edges are $\Gamma(i,j)$, where $i < j$ and $i,j \in \{1,2,3\}$.

\smallskip

Recall the definition of $Z_s$ in \eqref{eqn-Zs}.  Our main goal in this section is
to show the following.

\smallskip

\begin{remark}
By the results \cite{ADS2, CHHW} discussed above, the orbits in $\Gamma(1, 2)$ and
$\Gamma(2, 3)$ in Theorem~\ref{thm-shadowing} consist only of one orbit.
\end{remark}

\begin{figure}[th]
\centering \includegraphics[width=0.6\textwidth]{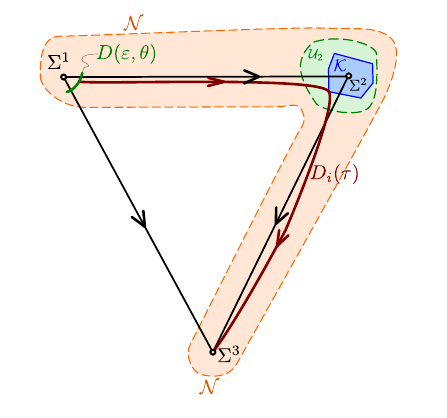}
\caption{Proof of Theorem~\ref{thm-shadowing}.}
\label{fig:proof-thm}
\end{figure}

Since $\Sigma^3$ is stable, there is a small neighborhood $\mc U_3\ni \Sigma^3$ such
that $\phi^\tau(\mc U_3) \subset \mc N$ for all $\tau\geq 0$.  The proof of Theorem
\ref{thm-shadowing} will use this fact as well as the following two
Lemmas~(\ref{prop-nbhd-sigma2} and~\ref{prop-nbhd-sigma2-backwards}).

\begin{lemma}
\label{prop-nbhd-sigma2} There is a neighborhood $\mc U_2\subset Z_s$ of $\Sigma^2$
such that for every $C\in \mc U_2$ one either has $\phi^\tau(C)\to \Sigma^2$ or
$\phi^\tau(C)\to\Sigma^3$ as $\tau\to\infty$, and in both cases one has
$\phi^{[0,\infty)}(C)\subset \mc N$.
\end{lemma}
\begin{proof}
If this were not true then there would exist sequences $C_j\to\Sigma^2$ and
$\tau_j>0$, with $\phi^{\tau_j}(C_j)\not\in \mc N$.  Since $C_j\in Z_s$, by
definition we have $T_{\max}(C_j)=1$, that is the semiflow $\phi^\tau(C_j)$ exists
for all $\tau \geq 0$.  Since $\phi^\tau(\Sigma^2)=\Sigma^2$ for all $\tau\geq 0$,
the continuity of the semiflow (Proposition \ref{lemma-semiflow-continuous}) implies
that $\tau_j\to\infty$.  By Lemma \ref{lemma-huisken-cont} it also follows from
$C_j\to\Sigma^2$ that $\hu(C_j)\to\hu(\Sigma^2) < 2$.

Proposition \ref{lemma-rescaling-time} tell us that each $C_j$ is either compact or
of the form $\bar C_j \times \R^{3-k}$ for some compact symmetric convex set
$\bar C_j \subset \R^{k+1}$.  Hence by Huisken's theorem \cite{Hu} we have
$\lim_{\tau \to \infty} \phi^\tau(C_j) = \Sigma^k$ and because $C_j \to \Sigma^2$, by
the monotonicity of Huisken's energy we must have either
$\lim_{\tau \to \infty} \phi^\tau(C_j) = \Sigma^3$ or
$\lim_{\tau \to \infty} \phi^\tau(C_j) = \Sigma^2$.

For each $j$, consider the solution to RMCF given by
$\tilde C_j (\tau) = \phi^{\tau_j+\tau}(C_j)$, $\tau \in [-\tau_j, +\infty)$.  From
our discussion above we have $\lim_{\tau \to \infty} \tilde C_j (\tau) = \Sigma^2$ or
$\lim_{\tau \to \infty} \tilde C_j(\tau) = \Sigma^3$ and
$\lim_{j \to \infty} \tilde C_j(-\tau_j) = \Sigma^2$.  Hence, the monotonicity of
Huisken's energy along the flow this implies that
\[
\delta \le \mathcal{H}(\tilde{C}_j(\tau)) \le 2 - \delta,
\]
for all $\tau\in [-\tau_j,\infty)$ and some $\delta >0$.

This implies that for every $\tau_0 \in \R$, there exists $j_0$ so that for all
$j \ge j_0$, $\tilde{C}_j(\tau_0) \in X(\delta, 2-\delta)\cap Z_s$.  By Lemma
\ref{lemma-compactness-3} we can extract a subsequence, call it still
$\tilde{C}_j(\tau_0)$ which converges, as $j\to\infty$, to
$\tilde{C}_{\infty}(\tau_0) \in X(\delta, 2-\delta)\cap Z_s$.  By Lemma
\ref{lemma-semiflow-continuous} one can show that
$\lim_{j\to\infty} \phi^{\tau}(\tilde{C}_j(\tau_0)) =
\phi^{\tau}(\tilde{C}_{\infty}(\tau_0))$, for all $\tau\in [\tau_0,\infty)$ and the
convergence is uniform on compact subsets of $\tau\in [\tau_0,\infty)$.  Letting
$\tau_0\to -\infty$, using a diagonal procedure for extracting a convergent
subsequence and the uniqueness of the limit, we conclude that a subsequence, call it
again $\tilde{C}_j(\tau)$ converges as $j\to\infty$ to a solution
$\tilde{C}_{\infty}(\tau)$ to RMCF and $\tau\in (-\infty,\infty)$, and hence,
$\tilde{C}_{\infty}(0) \in I(\delta,2-\delta)$.  Since
$\tilde C_j(0) = C_j\not\in\mc N$ for all $j$, we have
$\tilde C_\infty(0)\not\in\mc N$.

By Proposition \ref{prop-descr-I}, the solution
$\{\tilde C_\infty(\tau)\}_{\tau\in (-\infty,\infty)} \in I_s(\delta, 2-\delta)$ must
be a connecting orbit between two fixed points $\Sigma^i$ and $\Sigma^j$.  It follows
from $\hu(\tilde C_j(\tau))\leq \hu(C_j)\to\hu(\Sigma^2)$ that the Huisken energy of
the limits $\tilde C_\infty(\pm\infty)$ is bounded from above by $\hu(\Sigma^2)$.
Hence $\tilde C_\infty(\tau)$ either is constant, in which case
$\tilde C_\infty(\tau)=\Sigma^2$ or $\Sigma^3$ for all $\tau\in\R$, or else
$\tilde C_\infty(\tau)$ is the connecting orbit from $\Sigma^2$ to $\Sigma^3$, which
is unique by \cite{CHHW}.  The fixed points $\Sigma^2$, $\Sigma^3$, and the
connecting orbit between them are all contained in the neighborhood $\mc N$.
Therefore we have $\tilde C_\infty(0)\in\mc N$, contradicting our earlier observation
that $\tilde C_\infty(0)\not\in\mc N$.
\end{proof}

\begin{lemma}
\label{prop-nbhd-sigma2-backwards}
There is a neighborhood $\mc U_2\ni \Sigma^2$ such that for every ancient orbit
$C:(-\infty, 0] \to Z_s$ of RMCF with $C(0) \in\mc U_2$ one either has
$C(\tau)\to \Sigma^2$ or $C(\tau)\to\Sigma^1$ as $\tau\to-\infty$, and in both cases
one has $C(\tau)\in \mc N$ for all $\tau\leq 0$.
\end{lemma}

\begin{proof}
This lemma follows by very similar arguments as those in the proof of the previous
lemma, hence we will just briefly sketch it.  Assume the lemma were not true, which
would mean there existed a sequence $C_j \in Z_s$ and $\tau_j < 0$, so that
$C_j \to \Sigma^2$, lying on an ancient orbit $C_j(\tau)$, for $\tau\in (-\infty,0]$,
$C_j(0) = C_j$, and $C_j(\tau) \not\in \mathcal{N}$.  As in Lemma
\ref{prop-nbhd-sigma2} we conclude that $\tau_j\to -\infty$.  Consider the solutions
to RMCF given by $\tilde{C}_j(\tau) = C_j(\tau_j + \tau)$, that are defined for
$\tau \in (-\infty,-\tau_j]$.

Huisken's monotonicity implies that
\[\mathcal{H}(C_j) = \mathcal{H}(\tilde{C}_j(-\tau_j))\le
\mathcal{H}(\tilde{C}_j(\tau)),\] for all $\tau\in (-\infty, -\tau_j]$ and similarly
as in the proof of Lemma \ref{prop-nbhd-sigma2} we get that
\[\mathcal{H}(\Sigma_2) - \delta \le \mathcal{H}(\tilde{C}_j(\tau)) \le 2-\delta,\]
for all $\tau\in (-\infty,-\tau_j]$, some small $\delta > 0$ and $j$ sufficiently
big.  As in the proof of Lemma \ref{prop-nbhd-sigma2} we conclude that a subsequence,
call it again $\tilde{C}_j(\tau)$ converges as $j\to\infty$ to a solution
$\tilde{C}_{\infty}(\tau)$ to rescaled MCF which is defined for all
$\tau\in (-\infty,\infty)$ and hence $\tilde C_{\infty}(0) \in I_s(\delta,2-\delta)$.
Since $\tilde C_j(0) = C_j\not\in\mc N$ for all $j$, we have
$\tilde C_\infty(0)\not\in\mc N$.

By Proposition \ref{prop-descr-I}, the solution
$\{\tilde C_\infty(\tau)\}_{\tau\in (-\infty,\infty)} \in Z_s$ must be a connecting
orbit between two fixed points $\Sigma^i$ and $\Sigma^j$.  It follows from
$\hu(\tilde C_j(\tau))\geq \hu(C_j)\to\hu(\Sigma^2)$ that the Huisken energy of the
limits $\tilde C_\infty(\pm\infty)$ is bounded from below by $\hu(\Sigma^2)$.
Therefore, $\tilde C_\infty(\tau)$ is either constant, in which case
$\tilde C_\infty(\tau)=\Sigma^2$ or $\Sigma^1$ for all $\tau\in\R$, or else
$\tilde C_\infty(\tau)$ is the connecting orbit from $\Sigma^2$ to $\Sigma^1$, which
is unique by \cite{ADS1}.  The fixed points $\Sigma^2$, $\Sigma^1$, and the
connecting orbit between them are all contained in the neighborhood $\mc N$.
Therefore we have $\tilde C_\infty(0)\in\mc N$, contradicting our earlier observation
that $\tilde C_\infty(0)\not \in\mc N$.

\end{proof}

\subsection*{Ellipsoids}
For $\alpha=(\alpha_1,\alpha_2,\alpha_3,\alpha_4)\in\R^4$ with $\alpha_j\geq 0$ we
consider the ellipsoids
\[
E_1(\alpha) := \left\{ x\in \R^4 \mid
\alpha_1^2x_1^2+\alpha_2^2x_2^2+\alpha_3^2x_3^2+\alpha_4^2x_4^2\leq 1\right\}
\]
Let $T_\alpha = T_{\max}\bigl(E_1(\alpha)\bigr)$ be the MCF-extinction time of
$E_1(\alpha)$ and define
\[
E(\alpha) = T_\alpha^{-1/2}\,E_1(\alpha).
\]
Then the MCF starting at $E(\alpha)$ becomes singular at time
$T_{\max}(E(\alpha))=1$, and therefore the rescaled MCF $\phi^\tau(E(\alpha))$ is
defined for all $\tau\geq 0$, and converges to $\Sigma^k$ for some~$k\in\{1,2,3\}$.

By scaling we have
\[
E(\lambda \alpha) = E(\alpha)
\]
for all $\lambda\in(0,\infty)$ and all $\alpha\in\R^4$ with $\alpha_i\geq 0$
($1\leq i\leq 4$).

Some special cases of the ellipsoids are as follows:
\begin{itemize}
\item The ellipsoid $E(1, 1, 0, 0)$ is up to scaling $S^1\times\R^2$.
\item The ellipsoid $E(1, 1, \alpha, 0)$ is a cylinder, namely
\[
E(1, 1, \alpha, 0) = E(1, 1, \alpha) \times \R
\]
where $E(1, 1, \alpha)\subset\R^3$ is the analogously defined lower dimensional
ellipsoid.
\end{itemize}

The ellipsoid $E(1, 1, \alpha, \alpha)$ is $O(2)\times O(2)$ symmetric, and its
forward evolution by RMCF is also $O(2)\times O(2)$ symmetric.  Therefore we have
$\phi^\tau(E(1,1,\alpha,\alpha))\to \Sigma^3$ (the limit cannot be $\Sigma^1$ because
it is not $O(2)\times O(2)$ invariant.)

\subsection*{Proof of Theorem~\ref{thm-shadowing}}.
Choose a compact set $\mc K\subset \mc U_2$ with $\Sigma^2\in\mc K$, so that
$\Sigma^2$ lies in the interior of $\mc K$.  For any small $\varepsilon>0$ and
$\theta\in[0,1]$ we consider the ellipsoids
\[
D(\varepsilon, \theta) = E(1, 1, \varepsilon, \theta\varepsilon).
\]
For each small $\varepsilon>0$ we note that
\begin{itemize}
\item $\phi^{[0, \infty)}(D(\varepsilon, 1))$ consists of $O(2)\times O(2)$ symmetric
hypersurfaces.  Hence
$\phi^{[0, \infty)}(D(\varepsilon, 1)) \cap \mc U_2 = \varnothing$.
\item $\phi^{[0,\infty)}(D(\varepsilon, 0))$ consists of cylinders of the form
$C\times\R$.  This orbit must therefore converge to $\Sigma^2$ as $\tau\to\infty$.
In particular, $\phi^{[0,\infty)}(D(\varepsilon, 0)) \cap \mc K\neq\varnothing$.
\end{itemize}
We now let $\theta_1\in (0,1)$ be the largest value of $\theta$ for which
$\phi^{[0,\infty)}(D(\varepsilon, \theta))$ intersects $\mc K$, i.e.
\[
\theta_1(\varepsilon) = \sup \bigl\{\theta\in[0,1] \mid
\phi^{[0,\infty)}(D(\varepsilon, \theta))\cap\mc K\neq \varnothing \bigr\}.
\]
Since $\mc K$ is compact it follows that
\[
\phi^{[0,\infty)}(D(\varepsilon, \theta_1))\cap\mc K\neq \varnothing
\]
and
\begin{equation}
\label{eq-notinK}
\phi^{[0,\infty)}(D(\varepsilon, \theta_1))\cap\mathring {\mc K} = \varnothing.
\end{equation}
For each $\varepsilon>0$ let $\tau_\varepsilon\geq 0$ be the last time that
$\phi^{\tau}(D(\varepsilon, \theta_1))$ is in $\mc K$,
\[
\tau_\varepsilon = \sup \{\tau\geq 0 \mid \phi^\tau(D(\varepsilon, \theta_1)) \in \mc
K\}.
\]
If we let $\varepsilon\to0$ then $D(\varepsilon, \theta_1(\varepsilon))$ converges to
the fixed point $E(1, 1, 0, 0) = \Sigma^1$.  This implies that
$\tau_\varepsilon\to\infty$ as $\varepsilon\to 0$.  Consider the sequence of
solutions to RMCF defined by
\[
D_i(\tau) = \phi^{\tau_{\varepsilon_i}+\tau}\bigl(D(\varepsilon_i,
\theta_1(\varepsilon_i))\bigr),
\]
for $\tau\in [-\tau_{\varepsilon_i}, \infty)$.  By monotonicity of Huisken's
functional we have
\[
\mathcal{H}(D_i(\tau)) \le \mathcal{H}(D_i(-\tau_{\varepsilon_i})) =
\mathcal{H}(D(\varepsilon_i, \theta_1(\varepsilon_i)),
\]
for all $\tau\in [-\tau_{\varepsilon_i},\infty)$.  Since
$D(\varepsilon_i,\theta_1(\varepsilon_i))$ converges to $\Sigma^1$ as $i\to\infty$,
by Lemma \ref{lemma-huisken-cont} we have that for any $\delta > 0$ there exists an
$i_0$ so that for all $i\ge i_0$ and all $\tau\in [-\tau_{\varepsilon_i},\infty)$ we
have
\[
\mathcal{H}(D_i(\tau)) \le 2 - \delta.
\]
On the other hand, we know that for every fixed $i$, $D_i(\tau)$ converges as
$\tau\to\infty$ to $\Sigma^3$, and hence by monotonicity of Huisken's functional we
also have $\mathcal{H}(\Sigma_3) \le \mathcal{H}(D_i(\tau))$, for all
$\tau\in [-\tau_{\varepsilon_i},\infty)$, hence implying that
$D_i(\tau) \in X(\delta, 2-\delta)\cap Z_s$ for all $i \ge i_0$ and all
$\tau\in [-\tau_{\varepsilon_i},\infty)$, and some $\delta >0$.  By Lemma
\ref{lemma-compactness-3} and Proposition \ref{lemma-semiflow-continuous} we have
that $D_i(\tau)$ converges uniformly on bounded time intervals.  The limit
\[
D_\infty(\tau) = \lim _{i\to\infty} D_i(\tau)
\]
is a complete orbit of RMCF with the following properties:
\begin{itemize}
\item $D_\infty(0)\in \mc K$, since $D_i(0) \in \mc K$ for every $i$.
\item $D_{\infty}(\tau) \not\in \mathring{\mc K}$ for all $\tau \in \mathbb{R}$, due
to \eqref{eq-notinK}

\end{itemize}
The second property implies that $D_{\infty}(\tau)$ cannot converge to $\Sigma^2$ as
$\tau\to \pm\infty$, because $\Sigma^2\in\mathring{\mc K}$.  On the other hand
$D_\infty(\tau)$ is a complete orbit so its limits $D_\infty(\pm\infty)$ are fixed
points.  If the $\lim_{\tau\to\infty} D_{\infty}(\tau) = \Sigma^1$, by monotonicity
of Huisken's functional we would also have
$\lim_{\tau\to-\infty} D_{\infty}(\tau) = \Sigma^1$, implying that
$D_{\infty}(\tau) \equiv \Sigma^1$ for all $\tau\in (-\infty,\infty)$, contradicting
the property that $D_{\infty}(0) \in \mc K$.  Hence, the
$\lim_{\tau\to\infty} D_{\infty}(\tau) = \Sigma^3$, and therefore $D_{\infty}(\tau)$
is a compact eternal solution to RMCF.  In other words, $D_{\infty}(\tau)$ is a
complete orbit so its $\lim_{\tau\to-\infty} D_{\infty}(\tau)$ is also a fixed point.
If it were $\Sigma^3$, then by monotonicity of Huisken's functional we would have
that $D_{\infty}(\tau) \equiv \Sigma^3$ for all $\tau\in (-\infty,\infty)$, which
again contradicts the fact that $D_{\infty}(0) \in \mc K$.  Since $D_{\infty}(\tau)$
can not converge to $\Sigma^2$, then the
$\lim_{\tau\to-\infty}D_{\infty}(\tau) = \Sigma^1$.  Hence, we have
\[
D_\infty(-\infty) = \Sigma^1,\qquad D_\infty(+\infty) = \Sigma^3.
\]
Since $D_\infty(0)\in \mc K\subset \mc U_2$ it follows from
Lemmas~\ref{prop-nbhd-sigma2} and~\ref{prop-nbhd-sigma2-backwards} that
$D_\infty(\tau)\in \mc N$ for all $\tau\in\R$, as desired.

\section{Arguments supporting Conjecture \ref{conj-main}}
\label{sec-conj}

In this final section we will show that Conjecture \ref{conj-main} holds when $n=3$,
under an additional assumption which we believe can be removed.  This result is based
on the works \cite{ADS2, CHHW} regarding the classification of ancient compact
non-collapsed solutions from $\Sigma^1 \to \Sigma^3$ and $\Sigma^2 \to \Sigma^3$
(c.f~the discussion in section \ref{sec-hetero}), as well as the recent work in
\cite{CDDHS} regarding the uniqueness of {\em bubble-sheet ovals} in $\R^4$.  The
latter are ancient compact non-collapsed solutions to MCF in $\R^4$ connecting
$\Sigma^1$ to $\Sigma^3$, whose {\em bubble sheet matrix has rank 2} (for more
details and definitions regarding the bubble sheet matrix we refer the reader to
\cite{DH1}).  In this work it was shown that:

\begin{theorem} [c.f.  Theorem 1.3 in  \cite{CDDHS}]
Any bubble-sheet oval in $\mathbb{R}^{4}$ belongs, up to space-time rigid motion and
parabolic dilation, to the class $\mathcal{A}^{\circ}$ of
$\mathbb{Z}_2^2\times \textrm{O}(2)$-symmetric, ancient, noncollapsed, bubble-sheet
ovals constructed in \cite{DH}.
\end{theorem}

\begin{remark}
\label{rem-bel}
It is strongly believed that all compact ancient solutions which have a bubble sheet
$S^1\times\R^2$ as tangent flow at $-\infty$ are immediately of rank two.  In fact
this is addressed in a forthcoming work in preparation by K.  Choi, Haslhofer and
Hershkovits.
\end{remark}

Having in mind Remark \ref{rem-bel}, whose statement still needs to be shown, we will
present below the proof of Conjecture \ref{conj-main}.

\begin{theorem}[Conjecture \ref{conj-main} for $n=3$ assuming that the statement in Remark \ref{rem-bel} holds]\label{thm-conjecture}
Assume that $n=3$, i.e.~consider three-dimensional ancient solutions in $\R^4$.
Furthermore, assume that the statement in Remark \ref{rem-bel} holds true.  The
invariant set $I_s(h_0,h_1) = I(h_0, h_1) \cap X_s$, with $0 < h_0 < h_1 < 2$,
containing the generalized cylinders as fixed points is homeomorphic to a
two-dimensional simplex, i.e.~a triangle.
\end{theorem}

Before we proceed with the proof of the above theorem, we set $h_0=\delta$,
$h_1=2- \delta$, for $\delta >0$, and recall that the invariant set $I_s(h_0, h_1)$
consists of all $C\in X$ such that $C=-C$, and for which there is an entire solution
$\{C_\tau\}_{\tau\in\R} \subset X(h_0,h_1)$ of RMCF with $C_0=C$.  Furthermore,
$\{C_\tau\}_{\tau\in\R}$ are orbits connecting $\Sigma^1 \to \Sigma^2$, or
$\Sigma^2 \to \Sigma^3$, or $\Sigma^1 \to \Sigma^3$, and thus satisfy
\begin{equation} \label{eq-cond} \delta < \mc H(\Sigma^3) < \mc H(C_\tau) < \mc
H(\Sigma^1) < 2-\delta
\end{equation}
for all $\tau\in\R$.

\subsection*{Proof of Theorem \ref{thm-conjecture}}
Let $\mathcal{A}^{\circ}$ be the class of bubble-sheet ovals constructed in
\cite{DH}.  In \cite{DH, CDDHS} the class $\mathcal{A}^{\circ}$ consists by
definition of equivalence classes of ancient solutions to MCF whose vanishing time is
$T_{\max}=0$, and where two ancient Mean Curvature Flows $\{N^1_t\mid t<0\}$ and
$\{N^2_t\mid t<0\}$ are considered equivalent if there exist a rotation
$\mc R\in\SO_{n+1}$ and a scale factor $\lambda>0$ such that
$N^2_t = \lambda\mc R N^1_{\lambda^2t}$ for all $t<0$.

To connect the construction in \cite{DH,CDDHS} with our set up in this work, we
identify each ancient MCF $\{N_t\mid t<0\}$ with $T_{\max}=0$ with a complete orbit
$\{C_\tau\mid \tau\in\R\}$ of RMCF, by letting $C_\tau$ be the convex set whose
boundary is given by
\[
\partial C_\tau = e^{\tau/2}N_{-e^{-\tau}}.
\]
The action of $\SO_{n+1}$ on hypersurfaces $(\mc R, N)\mapsto \mc R\cdot N$ and on
convex sets $(\mc R, C)\mapsto \mc R\cdot C$are equivalent, while the parabolic
rescaling $\{N_t\}\mapsto \{\lambda N_{\lambda^2 t}\}$ of ancient solutions is
equivalent with translation in time $\{C_\tau\}\mapsto \{C_{\tau-2\ln\lambda}\}$ of
the corresponding complete orbit of RMCF.

Thus we can identify any equivalence class $[\mc O] \in\mc A^\circ$ of ancient
solutions of MCF with an equivalence class of complete orbits
$\{C_\tau\mid \tau\in\R\}$ of the flow $\phi^\tau$ on $X_s/\SO_{n+1}$, where two
orbits $\{C^1_\tau\}$ and $\{C^2_\tau\}$ are considered equivalent if there is a
$\tau_0\in\R$ such that $C^2_\tau = C^1_{\tau+\tau_0}$ for all $\tau\in\R$.

In \cite{CDDHS} (c.f.~Corollary 1.5) it was shown that $\mathcal{A}^{\circ}$ is
homeomorphic to the half open interval $[0,1)$, namely that there exists a
homeomorphism: $ h: [0,1) \to \mathcal{A}^{\circ}$,
$h(\alpha) = [\mathcal{O}_\alpha] \in \mathcal{A}^{\circ}$.  The endpoint $\alpha =0$
corresponds to the $\mathrm{O}(2) \times \mathrm{O}(2)$-symmetric connecting orbit
from $\Sigma^1$ to $\Sigma^3$ which was constructed in \cite{HH}.

Each equivalence class of orbits $h(\alpha) = [\mathcal{O}_\alpha]$ can be
represented by a complete orbit $\{C_\tau\mid \tau\in\R\}$ of RMCF, which is unique
up to translation in $\tau$: for any two complete orbits
$C_{\alpha,\tau}^1, C^2_{\alpha,\tau} \in [\mathcal{O}_\alpha]$ of RMCF that
represent $h(\alpha)=[\mc O_\alpha]$, there exists a $\tau_0$ such that
$C_{\alpha, \tau+\tau_0}^1=C_{\alpha,\tau}^2$.

Let $C_{\alpha,\tau}$ be a representative of $h(\alpha)$, chosen so that
$(\alpha,\tau)\mapsto C_{\alpha,\tau}$ is a continuous map into $X_s/\SO_{n+1}$.  For
each $\alpha\in[0, 1)$ the Huisken energy $\hu(C_{\alpha,\tau})$ is a continuous
function of $(\alpha, \tau)$ which is strictly decreasing in $\tau$, and for which
\[
\lim_{\tau\to-\infty}\hu(C_{\alpha, \tau}) = \hu\bigl(\Sigma^1\bigr),\qquad
\lim_{\tau\to+\infty}\hu(C_{\alpha, \tau}) = \hu\bigl(\Sigma^3\bigr).
\]
It follows that for each $\alpha\in[0, 1)$ and
$\eta\in\bigl(\hu(\Sigma^3), \hu(\Sigma^1)\bigr)$ there is a unique
$\bar\tau(\alpha, \eta)\in\R$ with
\[
\hu\bigl(C_{\alpha, \bar\tau}\bigr) = \eta.
\]
We now consider the map
$\psi:[0,1)\times(\hu(\Sigma^3), \hu(\Sigma^1))\to I_s(h_0, h_1)$ defined by
\[
\psi(\alpha, \eta) = C_{\alpha, \bar\tau(\alpha, \eta)} \,.
\]
\begin{claim}\label{claim-psi-has-limits}
For each $\eta\in(\hu(\Sigma^3), \hu(\Sigma^1))$ the limit
\[
\psi(1, \eta) := \lim_{\alpha\to 1}\psi(\alpha,\eta)
\]
exists, and convergence is uniform for $\eta\in(\hu(\Sigma^3), \hu(\Sigma^1))$.

For each $\eta\in(\hu(\Sigma^3), \hu(\Sigma^1))$ the limit $\psi(1, \eta)$ is the
unique $C_*\in\Gamma(1,2)\cup\Gamma(2,3)$ with $\hu(C_*)=\eta$.

We also have
\[
\lim_{\eta\searrow\hu(\Sigma^1)} \psi(\alpha, \eta) = \Sigma^1, \qquad
\lim_{\eta\nearrow\hu(\Sigma^3)} \psi(\alpha, \eta) = \Sigma^3
\]
where the convergence is uniform in $\alpha\in[0, 1)$.
\end{claim}

This implies that the map $\psi$ extends to a continuous map
$\psi : [0, 1]\times[\hu(\Sigma^1), \hu(\Sigma^3)] \to I_s(h_0, h_1)$.  This extended
map is constant on the top and bottom edges $[0, 1]\times\{\hu(\Sigma^k)\}$
($k\in\{1,3\}$).

\begin{claim}\label{claim-psi-injective}
$\psi(\alpha, \eta) = \psi(\alpha', \eta')$ if and only if
$(\alpha, \eta)=(\alpha', \eta')$ or $\eta=\eta'\in\{\hu(\Sigma^1), \hu(\Sigma^3)\}$.
\end{claim}
This implies that $\psi$ is a homeomorphism between $I_s(h_0, h_1)$ and the space one
obtains by collapsing the top and bottom edges of the rectangle
$[0, 1]\times\{\hu(\Sigma^k)\}$, i.e.~$[0, 1]\times\{\hu(\Sigma^k)\} / \sim$ where
$\sim$ is the equivalence relation
$(\alpha,\eta)\sim(\alpha', \eta')\iff (\alpha,\eta)=(\alpha',\eta')\text{ or
}\eta=\eta'\in\{1, 3\}$.  Figure~\ref{fig:triangle} exhibits a homeomorphism between
$[0, 1]\times[\hu(\Sigma^3),\hu(\Sigma^1)] / \sim$ and a 2-simplex.

\begin{figure}[t]
\centering 
\includegraphics[width=0.7\textwidth]{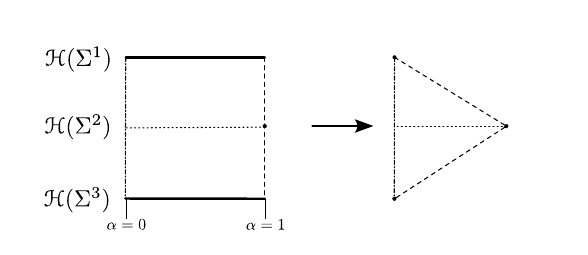}
\caption{A homeomorphism between $[0,1]\times[\hu(\Sigma^3),\hu(\Sigma^1)]/\!\sim$ and
  a triangle, which collapses the top and bottom edges of the rectangle.  It is given
  by $\varphi(\alpha,\eta) = (r(\eta)\alpha, \eta)$ where
  $r:[\hu(\Sigma^3), \hu(\Sigma^1)]\to[0,1]$ is piecewise linear with
  $r(\hu(\Sigma^1))=r(\hu(\Sigma^3))=0$ and $r(\hu(\Sigma^2))=1$.  }
\label{fig:triangle}
\end{figure}

\subsection*{Proof of Claim~\ref{claim-psi-has-limits}}
Let $\eta\in(\hu(\Sigma^3), \hu(\Sigma^1))$ be given.  Then the set
$\{\psi(\alpha, \eta) \mid 0\leq \alpha<1\}$ is contained in the compact set
$I_s(h_0, h_1)$.  It follows that any sequence $\alpha_k\nearrow 1$ has a subsequence
$\alpha_{k_j}$ for which $\psi(\alpha_{k_j}, \eta)$ converges.  Since
$\hu(\psi(\alpha, \eta))=\eta$, we know that
$\hu\bigl(\lim_j \psi(\alpha_{k_j}, \eta)\bigr) = \eta$.  For each $k$ there is a
complete orbit $C_{k, \tau}\in I_s(h_0,h_1)$ of RMCF with
$C_{k, 0} = \psi(\alpha_k, \eta)$.  We can choose the subsequence $\alpha_{k_j}$ so
that the orbits $C_{k_j, \tau}$ converge uniformly for bounded $\tau$ intervals.  It
follows that $C_{\infty,\tau} := \lim_j C_{k_j, \tau}$ is a complete orbit of RMCF.
If $C_{\infty,\tau}$ connects $\Sigma^1$ to $\Sigma^3$, then by the classification in
\cite{CDDHS}\cite{DH} it would have to be one of the $h(\alpha)$, and one would
conclude that the set of connecting orbits from $\Sigma^1$ to $\Sigma^3$ is no longer
homeomorphic to a half open interval $[0, 1)$, this in contradiction with
\cite{CDDHS}.  Thus $C_{\infty,\tau}$ either connects $\Sigma^1$ to $\Sigma^2$, or
$\Sigma^2$ to $\Sigma^3$, or the whole orbit $C_{\infty,\tau}$ is the fixed point
$\Sigma^2$.  Which of these alternatives occurs is completely determined by the fact
that $\hu(C_{\infty,0})=\eta$: if $\eta>\hu(\Sigma^2)$ then $C_{\infty, \tau}$
connects $\Sigma^1$ with $\Sigma^2$, if $\eta=\hu(\Sigma^2)$ then $C_{\infty, \tau}$
connects $\Sigma^2$ with $\Sigma^3$, and if $\eta=\hu(\Sigma^2)$ then
$C_{\infty, \tau}=\Sigma^2$ for all $\tau\in\R$.

The connecting orbits between $\Sigma^1$ and $\Sigma^2$, and between $\Sigma^2$ and
$\Sigma^3$ are unique.  We see that the limit $C_{\infty, 0}$ is the unique point
$C_*$ on the connecting orbit between $\Sigma^1$ and $\Sigma^2$ or $\Sigma^2$ and
$\Sigma^3$, respectively, for which $\hu(C_*)=\eta$; if $\eta=\hu(\Sigma^2)$ then
$C_*=\Sigma^2$.  Thus the limit $C_{\infty,0}$ does not depend on the particular
subsequence $k_j$, and therefore the whole family $\psi(\alpha,\eta)$ converges to
$C_*$ as $\alpha\nearrow1$.

We now show that $\psi(\alpha, \eta)\to\Sigma^1$ as $\eta\nearrow \hu(\Sigma^1)$.
Since the family $\psi(\alpha, \eta)$ lies in the compact set $I_s(h_0, h_1)$ we have
to show that if $\alpha_k\in[0, 1)$ and $\eta_k$ are sequences with
$\eta_k\nearrow\hu(\Sigma^1)$ for which the limit $\psi(\alpha_k, \eta_k)$ exists,
then the limit must be $\Sigma^1$.  Let $(\alpha_k, \eta_k)$ be such a sequence and
consider the complete orbits $C_{k,\tau}$ with $C_{k,0} = \psi(\alpha_k, \eta_k)$.
Then $\hu(\Sigma^3)<\hu(C_{k,\tau}) < \hu(\Sigma^1)$ for all $\tau$, while
$\hu(C_{k,0}) =\hu(\psi(\alpha_k, \eta_k))= \eta_k\nearrow \hu(\Sigma^1)$.  Since
$\hu(C_{k,\tau})$ is non increasing in $\tau$ we conclude that the limit
$C_{\infty, \tau}:=\lim_{k\to\infty}C_{k,\tau}$ is a complete orbit with
$\hu(C_{\infty,\tau})\equiv \hu(\Sigma^1)$ for all $\tau\leq0$.  This implies that
$C_{\infty, \tau} = \Sigma^1$ for all $\tau$ and in particular, that
$\lim_{k\to\infty}\psi(\alpha_k, \eta_k) = C_{\infty,0}=\Sigma^1$, as claimed.

The proof that $\psi(\alpha,\eta)\to\Sigma^3$ as $\eta\searrow\hu(\Sigma^3)$ follows
along the same lines.

\subsection*{Proof of Claim~\ref{claim-psi-injective} } 
It follows from the results in \cite{CDDHS} that the orbits
$h(\alpha_1), h(\alpha_2)$ are disjoint if $0\leq \alpha_1\neq\alpha_2<1$, so that
$\psi:[0, 1)\times(\hu(\Sigma^3), \hu(\Sigma^1)) \to I_s(h_0, h_1)$ is injective.  We
have shown that $\psi$ is constant on $[0, 1)\times\{\hu(\Sigma^j)\}$ for
$j\in\{1, 3\}$.  To complete the argument we observe that
$\psi(1, \eta)\in\Gamma(1,2)\cup\Gamma(2, 3)$ for
$\eta\in[\hu(\Sigma^3), \hu(\Sigma^1)]$, and that $\hu(\psi(1, \eta))=\eta$ so that
$\psi$ also is injective on $\{1\}\times[\hu(\Sigma^3), \hu(\Sigma^1)]$.

\appendix

\section{The structure of complete convex hypersurfaces}

In this section we present some basic known results on the general structure of
convex hypersurfaces.

Convex hypersurfaces are studied in both convex geometry and differential geometry of
submanifolds.  As a result there are different notions of convexity preferred in
different subjects.  In this paper, we use the following definition.

\begin{definition}[Convex hypersurfaces in $\mathbb{R}^{n+1}$]
\label{def-convexity}

A hypersurface $M \subset \mathbb{R}^{n+1}$ is called convex if it is the boundary of
a convex set $C$ of non-empty interior.  \end{definition}

\smallskip

The following is a known observation on the structure of convex sets.  For its proof
we refer the reader to \cite{CD}.  However, this result has been known for a long
time (see for example the work by Hung-Hsi Wu \cite{Wu}).

\begin{prop} \label{prop-classification} Let $M=\partial C$ be the boundary of a
closed convex set $C \subset \R^{n+1}$ with non-empty interior, that is $M$ is a
convex hypersurface in $\mathbb{R}^{n+1}$.  Then either $M=\mathbb{R}^n$ or
$M= \mathbb{R}^k \times \hat M$, for some $0\le k < n$ and $\hat M=\partial \hat C$
where $\hat C\subset \mathbb{R}^{n+1-k}$ is a closed convex set with non-empty
interior which {\em contains no infinite line.  } Moreover, such $\hat M$ is either
homeomorphic to $\mathbb{S}^{n-k}$ or $\mathbb{R}^{n-k}$.  In the former case
$\hat M = \partial \hat C$ is a compact hypersurface.

\end{prop}

\smallskip

The following theorem shown in \cite{Wu} concerns with complete non-compact
hypersurfaces $M$ that contain no infinite line.

\begin{prop}[Hung-Hsi Wu \cite{Wu}]
\label{prop-Wu}
Assume that $M =\partial C$ is a convex hypersurface in $\mathbb{R}^{n+1}$ that is
homeomorphic to $\R^n$ and contains no infinite line.  Then, coordinates can be so
chosen that $\{ x_{n+1}=0 \}$ is a supporting hyperplane to the convex set $C$ at the
origin, and it has the following additional properties:
\begin{enumerate}
\item[\bfseries\upshape a.]  Let $D:=\pi(C)$, where
$\pi: \R^{n+1} \to \{ x_{n+1} =0 \}$ is the standard orthogonal projection, and let
$D^\circ$ be its interior relative to the hyperplane $ \{ x_{n+1} =0 \}$.  Then, $M$
can be expressed as the graph of a nonnegative convex function $u: D^\circ \to \R$.
If $M$ is $C^\infty$, then $u$ is a $C^\infty$ function on $D^\circ$.  \smallskip

\item[\bfseries\upshape b.]  For every $x_0 \in D \setminus D^\circ$, $\pi^{-1}(x_0)$
is a semi-infinite line segment.  \smallskip
\item[\bfseries\upshape c.]  If the image $\gamma(M)$ of the Gauss map
$\gamma: M \to S^{n}$ has non-empty interior relative to $S^n$, then for any $c >0$
the level set $M \cap \{ x_{n+1} = c \}$ is homeomorphic to $S^{n-1}$.  This
homeomorphism is a diffeomorphism if $M$ is $C^\infty$ smooth.

\end{enumerate}

\end{prop}

\smallskip

Since the set $D:=\pi(C)$ that is defined in the previous Proposition will play
crucial role in the discussion of uniqueness of MCF solutions, we will refer to it as
the {\em shadow} of $M$ or $C$.

\begin{definition}[The shadow of a convex hypersurface]
\label{defn-shadow} Let $M =\partial C$ be a convex hypersurface in
$\mathbb{R}^{n+1}$ that is homeomorphic to $\R^n$ and contains no infinite line.  The
set $D:=\pi(C)$, defined as in Proposition \ref{prop-Wu} is called the shadow of $M$
or $C$.
\end{definition}

\smallskip

Below we summarize more classical facts about convex sets.  Reference??  \bigskip

Let $C\subset \R^{n+1}$ be a closed convex set.

\begin{lemma}[when $C$ contains a line]
If $C$ contains a line $\ell\subset C$ and $x\in C$ is given, then $C$ contains the
line through $x$ that is parallel to $\ell$.  In particular, $C$ is a product of
$\ell$ and a convex set $C'\subset \ell^\perp$.
\end{lemma}
\begin{proof}
Assume $\ell= \{ a+\sigma b \mid \sigma\in\R\}$ for suitable vectors
$a, b\in\R^{n+1}$ and let $x\in C$ be given.  We will show
$\{x+\sigma b\mid \sigma\in\R\}\subset C$.

For any $\rho\in\R$ the line segment connecting $x$ and $a+\rho b$ is contained in
$C$.  Thus for all $\rho\in\R$ and $\theta\in[0,1]$
\[
(1-\theta)x + \theta a + \theta\rho b \in C
\]
For $n=1, 2, 3, \dots$ choose $\theta_n=n^{-1}$, $\rho_n=n\sigma$.  Then
\[
x_n = (1-n^{-1})x + n^{-1}a + \sigma b\in C.
\]
Let $n\to\infty$ and recall that $C$ is closed to conclude that $C$ contains
$\lim_{n\to\infty} x_n = x+\sigma b$.
\end{proof}

\begin{lemma}
\label{lem:rays}
If $C$ contains a ray $\ell_+ = \{a+\sigma b \mid \sigma\geqslant 0\}$ and if
$x\in C$ is any point, then $C$ contains the ray starting at $x$ in the same
direction as $\ell_+$, i.e.~$\{x+\sigma b\mid \sigma\geq 0\}$.
\end{lemma}
\begin{proof}
The proof is the same as in the previous case.
\end{proof}

\begin{lemma}
If the closed convex set $C\subset\R^{n+1}$ does not contain a ray, then $C$ is
compact.
\end{lemma}
\begin{proof}
Let $a\in C$ and, assuming $C$ is not bounded, we choose a sequence of points
$p_k\in C$ with $\|p_k\|\to \infty$.  The line segments connecting $a$ and $p_k$ are
all contained in $C$.  Let $b_k = p_k/\|p_k\| \in S^n$, and pass to a subsequence for
which $b_k\to B\in S^n$.  Then the points $x_k = a + \sigma b_k$ all belong to $C$
provided $0\leq \sigma\leq \|p_k\|$.  We let $k\to\infty$ and find that
$\lim x_k = a+\sigma b\in C$ for any $\sigma\geq 0$.  Hence $C$ contains the ray in
the direction $b$ starting at $a$.
\end{proof}

\begin{lemma}
\label{lem:convex-is-Lipschitz}
If $h:B_R(0)\to\R$ is a convex function that is bounded by $|h(x)|\leq M$ for all
$x\in B_R(0)$, then $h$ is Lipschitz on $B_{(1-\delta)R}(0)$ with Lipschitz constant
$2M/\delta R$.
\end{lemma}
\begin{proof}
Let $p, q\in B_{(1-\delta)R}(0)$ with $h(q)>h(p)$ be given, and extend the line
segment $pq$ until it intersects $\partial B_R(0)$, say at the point $r$.  Then the
length of $qr$ is at least $\delta R$.  Convexity of $h$ implies
\[
\frac{h(q)-h(p)}{\|q-p\|} \leq \frac{h(r)-h(q)}{\|r-q\|} \leq \frac{2M}{\delta R}.
\qedhere
\]
\end{proof}

This directly implies the following:

\begin{lemma}
\label{lem:convex-in-cylinder-is-graph}
Let $C\subset\R^{n+1}$ be a closed convex set, and let $p, q$ be two points with
$B_\delta(p)\subset\interior C$ and $B_{\delta}(q)\cap C=\varnothing$.  Let $pq$ be
the line segment connecting $p$ and $q$.  Then $\partial C\cap B_{\delta/2}(pq)$ is
the graph of a Lipschitz continuous function in a coordinate system in which the
segment $pq$ is the vertical axis.  The Lipschitz constant of the function is bounded
by $2\|p-q\|/\delta$.
\end{lemma}

\section{Some useful facts}
A family of hypersurfaces $M_t\subset \R^n\times\R$ of the form
\[
M_t = \bigl\{(x, y)\in\R^n\times\R \mid \|x\| = r(y, t)\bigr\}
\]
evolves by MCF if
\[
r_t=\frac{r_{yy}}{1+r_y^2}-\frac{n-1}{r}
\]

\subsection{Expanding solitons from cones}
The hypersurfaces $M_t$ are a self similar expanding soliton exactly when
$r(y, t) = \sqrt{t}\, E(y/\sqrt t)$ for some function $E:\R\to(0,\infty)$, which then
must satisfy
\begin{equation}
\label{eq:expanding-soliton}
\frac{E''(\eta)}{1+E'(\eta)^2}+\frac{\eta}{2}E'(\eta)-\frac{1}{2}E(\eta)-\frac{n-1}{E(\eta)}=0.
\end{equation}
We recall that it was shown in \cite{ACP95,Helm12} that for each $a>0$ there is a
unique solution $E_a:\R\to\R$ of \eqref{eq:expanding-soliton} with
\[
E_a(0) = a,\qquad E_a'(0)=0.
\]
This solution is an even function which is defined for all $\eta\in\R$, which is
strictly increasing for $\eta>0$, and for which
\[
\lim_{\eta\to\infty}\frac{E_a(\eta)}{\eta} = A_a>0
\]
exists.

Let $R>0$ be given.  Let $\bs u\in \R^n\times\{0\}$ be any unit vector, and consider
the solution to MCF given by
\[
\forall (\hat x, \hat y)\in\R^n\times\R:\quad (\hat x, \hat y)\in M_t \iff \|\hat x -
R\bs u\| = \sqrt{t}\, E_{2R}\bigl(\frac{\hat y}{\sqrt t}\bigr).
\]
For $t\searrow0$ this solution converges to the cone $\|\hat x - R\bs u\|=A|\hat y|$,
where $A=A_{2R} = \lim_{\eta\to\infty}E_{2R}(\eta)/\eta$.  We get a solution to RMCF
by setting $\hat x=e^{-\tau/2}x$, $\hat y=e^{-\tau/2}y$, and $t=1-e^{-\tau}$:
\[
\|x-Re^{\tau/2}\bs u\| = \sqrt{e^\tau-1} E_{2R}\bigl(\frac{y}{\sqrt{e^\tau-1}}\bigr).
\]

\begin{lemma}\label{lemma-clearing-out-with-expanders}
For any $R>0$ there is an $A>0$ such that for any $C\in X$ which lies in the region
$\{(x,y)\in\R^n\times\R : \|x-R\bs u\|\geq A|y|\}$ the set $\phi^{\tau}(C)$ lies in
the region $\|x\|\geq R$ if $\tau\geq 2$.
\end{lemma}
\begin{proof}
Given $R$ we choose $A=A_{2R}$ as above.  If $C$ is a closed convex set that lies in
the region $\|x-R\bs u\|\geq A|y|$, then $\phi^\tau(C)$ is contained in the region
\[
\|x-Re^{\tau/2}\bs u\|\geq E_{2R}(0)\sqrt{e^\tau-1} = 2R\sqrt{e^\tau-1}.
\]
This implies
\[
\|x\|\geq \|x-Re^{\tau/2}\bs u\| - Re^{\tau/2}\geq
\bigl(2\sqrt{e^\tau-1}-e^{\tau/2}\bigr)R.
\]
If $\tau\geq 2$ then
$2\sqrt{e^\tau-1}-e^{\tau/2} \geq 2\sqrt{e^2-1}-e \approx 2.337\ldots > 1$.
\end{proof}

\subsection{The BLT pancake}
In \cite{BLT21} Bourni, Langford, and Tinaglia established the existence and
uniqueness of an $O(n)\times O(1)$ symmetric ancient solution that fills the slab
$\{(x, y)\in\R^n\times\R : |y|<\pi\}$ as $t\to-\infty$.

We will denote the BLT pancake of width $\pi$ that becomes singular at time $t=0$ by
$\hat P_t\subset\R^{n+1}, t<0$.  For $t\to-\infty$ the pancake $P_t$ is contained in
a pill-box $B^n_{R(t)}\times[-\pi, \pi]$ whose radius is asymptotically given by
$R_{\rm BLT}(t)=-t+o(t)$.  For $t\nearrow0$ the pancake shrinks to a round point at
the origin, according to Huisken and Gage-Hamilton's theorem.

\begin{lemma}\label{lemma-clearing-out-with-pancakes}
There is a $\varrho>0$ such that for any $\delta>0$ there exist $r, T>0$ with the
property that for any $C\in X$ that is disjoint from the set
$B^n_r\times[-\delta, \delta]\subset\R^{n+1}$, the image $\phi^T (C)$ is disjoint
from $B^{n+1}_{\varrho/\delta}$.
\end{lemma}
\begin{proof}
For any $\theta\in\R$ we consider the time translate $\hat P_{t-2}$ which becomes
singular at time $t=2$.  The corresponding RMCF is given by
\[
P_\tau = e^{\tau/2}\hat P_{-1-e^{-\tau}}.
\]
Next we translate in the rescaled time by an amount $\vartheta>0$ and consider
$Q_\tau:=P_{\tau-\vartheta} = e^{(\tau-\vartheta)/2} \hat P_{-1-e^{\vartheta-\tau}}$.
At time $\tau=0$ we have $Q_0 = e^{-\vartheta/2}\hat P_{-1-e^{\vartheta}}$, which is
contained in
$B^n_{R_{\rm BLT}(-1-e^\vartheta)}\times[-\pi e^{-\vartheta/2}, \pi
e^{-\vartheta/2}]$.  We choose $\vartheta = 2\ln \frac{\pi}{\delta}$, so that
$\pi e^{-\vartheta/2}=\delta$, and let
$r= R_{\rm BLT}(-1-e^{\vartheta}) = R_{\rm BLT}(-1-\pi^2/\delta^2)$.  This ensures
that $Q_0\subset B^n_r\times [-\delta, \delta]$ so that $Q_0$ is disjoint from $C$.

Choose $T=2\vartheta$.  Then $\phi^T(C)$ must be disjoint from
\[
Q_T = e^{\vartheta/2}\hat P_{-1-e^{-\vartheta}} = \frac{\pi}{\delta} \hat
P_{-1-\delta^2/\pi^2}.
\]
We may assume that $\delta<\pi$, so that $1<1+\delta^2/\pi^2<2$.  Let $\varrho>0$ be
small enough so that $B^{n+1}_{\varrho/\pi}$ is contained in $ \hat P_t$ for all
$t\in[-2, -1]$.  These choices imply that $\phi^T(C)$ is disjoint from
$ \frac{\pi}{\delta}B^{n+1}_{\varrho/\pi} = B^{n+1}_{\varrho/\delta}, $ as claimed.
\end{proof}

\subsection{Decay of the Gaussian mass}
Even though we did not use it in this paper, the following result does seem like it
might be useful in another setting.

\begin{lemma}\label{lemma-mass-decay}
On any time interval $[\tau_0, \tau_1]$ the mass of $C_\tau$ is bounded from below by
\[
\gvol(C_{\tau_1}) \geq \gvol(C_{\tau_0}) - \frac{1}{\sqrt
  2}\hu(C_{\tau_0})\sqrt{\tau_1-\tau_0}.
\]
\end{lemma}

\begin{proof}
The normal velocity of $\partial C_\tau$ is $V=H+\frac12 \langle x, N\rangle$, so the
rate at which the mass of $C_\tau$ decreases satisfies
\begin{align*}
  -\frac{d \gvol(C_\tau)}{d\tau} &= \int_{\partial C_\tau}e^{-\|x\|^2/4}V\frac{dH^n}{(4\pi)^{n/2}} \\
                                 &\leq \left(\int_{\partial C_\tau}e^{-\|x\|^2/4}\frac{dH^n}{(4\pi)^{n/2}}\right)^{1/2}
                                   \left(\int_{\partial C_\tau}e^{-\|x\|^2/4}V^2\frac{dH^n}{(4\pi)^{n/2}}\right)^{1/2}\\
                                 &=\sqrt{\hu(C_\tau)}\, \left(\int_{\partial C_\tau}e^{-\|x\|^2/4}V^2 \frac{dH^n}{(4\pi)^{n/2}}\right)^{1/2}
\end{align*}
The last integral is the rate at which $\hu(C_\tau)$ decreases:
\[
\int_{\partial C_\tau}e^{-\|x\|^2/4}V^2 \frac{dH^n}{(4\pi)^{n/2}} =
-\frac{d\hu(C_\tau)}{d\tau}.
\]
We therefore get
\[
-\frac{d \gvol(C_\tau)}{d\tau} \leq \sqrt{-\hu(C_\tau)\frac{d\hu(C_\tau)}{d\tau}}
=\sqrt{-\frac12\frac{d\hu(C_\tau)^2}{d\tau}}.
\]
On any time interval $[\tau_0, \tau_1]$ we therefore have
\begin{align*}
  \gvol(C_{\tau_0})-\gvol(C_{\tau_1})
  &\leq \int_{\tau_0}^{\tau_1} \sqrt{-\frac12\frac{d\hu(C_\tau)^2}{d\tau}} \,d\tau\\
  &\leq \sqrt{\int_{\tau_0}^{\tau_1}\frac 12 \, d\tau}\sqrt{-\int_{\tau_0}^{\tau_1}\frac{d\hu(C_\tau)^2}{d\tau}\, d\tau}\\
  &\leq \sqrt{\tfrac 12 (\tau_1-\tau_0)}\, \sqrt{\hu(C_{\tau_0})^2-\hu(C_{\tau_1})^2}\\
  &\leq \frac{1}{\sqrt 2}\hu(C_{\tau_0}) \sqrt {\tau_1-\tau_0}.
\end{align*}
\end{proof}

\end{document}